\documentclass[reqno, 11pt]{amsart}
\usepackage[left=1.03in, right=1.03in, top=1.05in, bottom=1.05in]{geometry}  

\usepackage{todonotes}
\presetkeys{todonotes}{inline}{}
\makeatletter
\providecommand\@dotsep{5}
\makeatother
\usepackage{amssymb, amsmath, amsthm, enumerate,  marginnote, mathabx, mathrsfs, mathtools, tikz, bm}
\usetikzlibrary{arrows,calc}
\usetikzlibrary{arrows.meta, cd, patterns}
\usepackage{bbm} 
\usepackage{esint} 
\usepackage{verbatim}

\usepackage{hyperref}
\hypersetup{
    colorlinks=true,
    linkcolor=blue,
    filecolor=magenta,
}

\usepackage[normalem]{ulem}

\usetikzlibrary{decorations.pathreplacing}
\usetikzlibrary{decorations,calligraphy}

\usepackage{todonotes}

\usepackage{arydshln, multirow} 

\usepackage{comment}
\numberwithin{equation}{section}
\newtheorem{theorem}{Theorem}[section]
\newtheorem{corollary}[theorem]{Corollary}
\newtheorem{lemma}[theorem]{Lemma}
\newtheorem{proposition}[theorem]{Proposition}
\newtheorem{definition}[theorem]{Definition}

\newtheorem{conjecture}{Conjecture}

\newtheorem{example}[theorem]{Example}

\theoremstyle{definition}

\newtheorem*{induction hypothesis}{Induction hypothesis}

\theoremstyle{remark}
\newtheorem*{remark}{Remark}

\def\inn#1#2{\langle#1,#2\rangle}

\newcommand{\uthree}{u_{3,1}}

\def\bu{u}

\def\lesim{\lesssim}

\newcommand{\set}[1]{ \left\{ #1 \right\} }

\newcommand{\supp}{\mathrm{supp}\,}
\newcommand{\xisupp}{\mathrm{supp}_{\xi}\,}

\newcommand{\R}{\mathbb{R}}
\newcommand{\Z}{\mathbb{Z}}
\newcommand{\N}{\mathbb{N}}

\newcommand{\ud}{\mathrm{d}}
\newcommand{\Scn}{\mathscr{S}(\R^n)}

\newcommand{\xip}{\tfrac{\xi}{|\xi|}}

\newcommand{\be}{\mathbf{e}}
\newcommand{\ba}{\mathbf{a}}
\newcommand{\bg}{\mathbf{g}}

\newcommand{\floor}[1]{\lfloor #1 \rfloor }
\def\lc{\lesssim}
\def\gc{\grtsim}

\def\bbone{{\mathbbm 1}}

\newcommand{\Be}{\begin{equation}}
\newcommand{\Ee}{\end{equation}}

\newcommand{\Bm}{\begin{multline}}
\newcommand{\Em}{\end{multline}}
\newcommand{\Bea}{\begin{eqnarray}}
\newcommand{\Eea}{\end{eqnarray}}
\newcommand{\Beas}{\begin{eqnarray*}}
\newcommand{\Eeas}{\end{eqnarray*}}
\newcommand{\Benu}{\begin{enumerate}}
\newcommand{\Eenu}{\end{enumerate}}
\newcommand{\Bi}{\begin{itemize}}
\newcommand{\Ei}{\end{itemize}}

\def\intslash{\rlap{\kern  .32em $\mspace {.5mu}\backslash$ }\int}
\def\qsl{{\rlap{\kern  .32em $\mspace {.5mu}\backslash$ }\int_{Q_x}}}

\def\N{\mathbb N}

\def\set#1{{ \left\{ #1 \right\} }}
\def\floor#1{{\lfloor #1 \rfloor }}
\def\emph#1{{\it #1 }}
\def\diam{{\text{\rm  diam}}}

\def\Ga{\Gamma}
\def\ga{\gamma}

\def\inn#1#2{\langle#1,#2\rangle}

\def\meas{{\text{\rm meas}}}

\def\lc{\lesssim}
\def\gc{\gtrsim}

\newcommand{\eps}{\varepsilon}
 
\newcommand{\ka}{\kappa}
\newcommand{\la}{\lambda}

\def\om{\omega}

\def\fN{{\mathfrak {N}}}

\def\bbN{{\mathbb {N}}}

\def\bbR{{\mathbb {R}}}

\def\bbZ{{\mathbb {Z}}}

\def\cE{{\mathcal {E}}}
\def\cF{{\mathcal {F}}}

\def\cS{{\mathcal {S}}}

\def\cZ{{\mathcal {Z}}}

\title[Sobolev improving for averages over curves]{Sobolev improving for averages over curves in
$\mathbb{R}^4$}

\author[D. Beltran]{David Beltran}
\author[S. Guo]{Shaoming Guo}
\author[J. Hickman]{Jonathan Hickman}
\author[A. Seeger]{Andreas Seeger}
\date{\today}

\address{David Beltran: Department of Mathematics, University of Wisconsin, 480 Lincoln Drive, Madison, WI, 53706, USA.}
\email{dbeltran@math.wisc.edu}
\address{Shaoming Guo: Department of Mathematics, University of Wisconsin, 480 Lincoln Drive, Madison, WI, 53706, USA.}
\email{shaomingguo@math.wisc.edu}
\address{Jonathan Hickman: School of Mathematics, James Clerk Maxwell Building, The King's Buildings, Peter Guthrie Tait Road, Edinburgh, EH9 3FD, UK.}
\email{jonathan.hickman@ed.ac.uk}
\address{Andreas Seeger: Department of Mathematics, University of Wisconsin, 480 Lincoln Drive, Madison, WI, 53706, USA.}
\email{seeger@math.wisc.edu}

\begin{document}

\maketitle

\begin{abstract} We study $L^p$-Sobolev improving for averaging operators $A_{\gamma}$ given by convolution with a compactly supported smooth density $\mu_{\gamma}$ on a non-degenerate curve. In particular, in 4 dimensions we show that $A_{\gamma}$ maps $L^p(\R^4)$ to the Sobolev space $L^p_{1/p}(\R^4)$ for all $6 < p < \infty$. This implies the complete optimal range of $L^p$-Sobolev estimates, except possibly for certain endpoint cases. The proof relies on decoupling inequalities for a family of cones which decompose the wave front set of $\mu_{\gamma}$. In higher dimensions, a new non-trivial necessary condition for $L^p(\R^n) \to L^p_{1/p}(\R^n)$ boundedness is obtained, which motivates a conjectural range of estimates.
\end{abstract}




\section{Introduction} For $n\ge 2$ let $\gamma \colon I \to \R^n$ be a smooth curve,\footnote{Throughout, any curve is tacitly assumed to be simple (that is, $\gamma$ is injective) and regular ($\gamma'$ is non-vanishing).} where $I \subset \R$ is a compact interval, and $\chi \in C^{\infty}(\R)$ be a bump function supported on the interior of $I$. Consider the averaging operator 
\begin{equation}\label{averaging operator}
    A_{\gamma}f(x) := \int_{\R} f(x - \gamma(s))\,\chi(s) \, \ud s;
\end{equation}
in particular, $A_{\gamma}f=\mu_\gamma \ast f$, where $\mu_{\gamma}$ is the measure given by the push-forward of $\chi(s)\ud s$ under~$\gamma$.

The goal of this paper is to study sharp $L^p$-Sobolev improving bounds for the operator $A_{\gamma}$ for a wide class of curves in $\R^4$. To state the main theorem, we say a smooth curve $\gamma \colon I \to \R^n$ is \textit{non-degenerate} if there is a constant $c_0 > 0$ such that 
\begin{equation}\label{eq:nondegenerate}
    |\det(\gamma'(s), \cdots, \gamma^{(n)}(s))| \geq c_0 \qquad \textrm{for all $s \in I$}
\end{equation}
or, equivalently, the $n-1$ curvature functions of $\gamma$ are all bounded away from 0.

\begin{theorem}\label{non-degenerate theorem} If $\gamma \colon I \to \R^4$ is non-degenerate and $6 < p < \infty$, then
\begin{equation*}
   \|A_{\gamma}f\|_{L^p_{1/p}(\R^4)} \lesssim_{p,\gamma, \chi} \|f\|_{L^p(\R^4)}.
\end{equation*}
\end{theorem}

This result is sharp up to $p = 6$ in the sense that the $L^p \to L^p_{1/p}$ bound fails whenever $2 \leq p < 6$: see Proposition~\ref{intro necessity proposition} below.  Furthermore, interpolation with the elementary $L^2 \to L^2_{1/4}$ inequality and duality give the complete range of $L^p \to L^p_{\alpha}$ estimates for all $1 \leq p \leq \infty$, except possibly for endpoint cases.\medskip

In higher dimensions no $L^p \to L^p_{1/p}$ estimates are currently known to hold for such averaging operators, although it is natural to conjecture that the following holds.

\begin{conjecture}\label{main conjecture} If $\gamma \colon I \to \R^n$ is non-degenerate and $2n - 2 < p < \infty$, then
\begin{equation}\label{conjecture equation}
   \|A_{\gamma}f\|_{L^p_{1/p}(\R^n)} \lesssim_{p,\gamma, \chi} \|f\|_{L^p(\R^n)}.
\end{equation}
\end{conjecture}

If true, then the above conjectured range would be sharp except for some endpoint cases, viz.

\begin{proposition}\label{intro necessity proposition} Let $2 \leq p \leq \infty$. If $\gamma \colon I \to \R^n$ is non-degenerate and the inequality 
\begin{equation*}
    \| A_\gamma f \|_{L^p_\alpha(\R^n)} \lesssim_{p,\gamma,\chi} \| f \|_{L^p(\R^n)}
\end{equation*}
holds, then we must have $\alpha \leq \min \Big\{ \frac{1}{n}\Big(\frac{1}{2} + \frac{1}{p} \Big), \frac{1}{p} \Big\}$.
\end{proposition}

As in the case of Theorem \ref{non-degenerate theorem}, the sharp estimates for $1 \leq p \leq 2n-2$ would follow from Conjecture \ref{main conjecture} by interpolation with the $L^2 \to L^2_{1/n}$ inequality and duality, except the endpoint regularity estimates for $ \frac{2n-2}{2n-3} \leq p \leq 2n-2$, $p\neq 2$.
\medskip

In the euclidean plane Conjecture~\ref{main conjecture} is an elementary consequence of the decay of the Fourier transform of the measure $\mu_{\gamma}$. In higher dimensions the problem is significantly more difficult, owing to the weaker rate of Fourier decay.\footnote{In particular, the curve is no longer a Salem set.} The $n=3$ case was established up to the $p=4$ endpoint by Pramanik and the fourth author \cite{PS2007}, conditional on the sharp Wolff-type `$\ell^p$-decoupling' inequality for the light cone. The sharp decoupling inequality was later proved by Bourgain--Demeter \cite{BD2015}, thus establishing the bounds for the averaging operators unconditionally. Theorem~\ref{non-degenerate theorem} verifies the $n=4$ case of Conjecture~\ref{main conjecture} up to the $p =6$ endpoint. The proof strategy behind Theorem~\ref{non-degenerate theorem} is based on that used to study the $n=3$ case in \cite{PS2007}, although significant new features and additional complications arise in the four-dimensional setting. To overcome these difficulties, advantage is taken of recent new advances in the understanding of decoupling theory. A key tool is the Bourgain--Demeter--Guth decoupling theorem for curves \cite{BDG2016}. 

The first stage of the argument relies on a careful decomposition of the operator in the frequency domain. This part of the proof is inspired by the analysis of the helical maximal function appearing in \cite{BGHS-helical} (see also \cite{PS2007}). Indeed, the maximal problem treated in \cite{BGHS-helical} shares a number of essential features with Theorem~\ref{non-degenerate theorem}. In particular, for both problems it is natural to microlocalise the operator with respect to a pair of nested cones in the frequency domain (see the introductory discussion in \cite{BGHS-helical} for more details).  However, a quick comparison between this paper and \cite{BGHS-helical} shows that the methods and overall proof scheme differ on a number of key points. For instance, the frequency decomposition used here is significantly more involved than that used in \cite{BGHS-helical}, owing to additional complications which arise when working in $\R^4$ rather than $\R^3$. Furthermore, whilst decoupling plays an important r\^ole in the current paper, the analysis in \cite{BGHS-helical} relies on square function estimates. One useful feature of decoupling (as opposed to the use of square functions) is that decoupling inequalities are readily iterated. We make use of this fact in a fundamental way when decomposing the operator with respect to the different frequency cones.




\subsection{Corollaries} Theorem~\ref{non-degenerate theorem} has a number of consequences which follow immediately from known arguments.\medskip

\noindent\textit{Extension to finite type curves.}  Using arguments from \cite{PS2007}, one can show that  Theorem~\ref{non-degenerate theorem} implies bounds for a more general class of curves. We say a smooth curve $\gamma \colon I \to \R^n$ is of \textit{finite maximal type} if there exists $d \in \N$ and a constant $c_0 > 0$ such that 
\begin{equation}\label{finite type}
    \sum_{j=1}^d |\inn{\gamma^{(j)}(s)}{\xi}| \geq c_0|\xi| \qquad \textrm{for all $s \in I$, $\xi \in \R^n$.}
\end{equation}
For fixed $s$, the smallest $d$ for which \eqref{finite type} holds for some $c_0 > 0$ is called the \textit{type of $\gamma$ at $s$}. The type is an upper semicontinuous function, and the supremum of the types over all $s \in I$ is referred to as the \textit{maximal type} of $\gamma$.

\begin{corollary}\label{finite type corollary} If $\gamma \colon I \to \R^4$ is of maximal type $d \in \N$ and $\max\{6, d\} < p < \infty$, then 
\begin{equation*}
    \|A_{\gamma}f\|_{L^p_{1/p}(\R^4)} \lesssim_{p,\gamma, \chi} \|f\|_{L^p(\R^4)}.
\end{equation*}
\end{corollary}

This result is sharp up to endpoints (for further discussion of endpoint cases, see \S\ref{Christ ex sec}) regarding the range of $p$ for which the regularity of order $1/p$ holds. In the range $2 \leq p \leq \max \{6,d\}$, the inequalities resulting from interpolation with the $L^2(\R^4) \to L^2_{1/d}(\R^4)$ estimates are also sharp, up to the regularity endpoint, for $d \geq 6$ and the non-degenerate $d=4$ case; for $d=5$ one expects, however, better bounds to hold in this range (see Figure \ref{finite type regularity range}). There are also natural extensions of Conjecture~\ref{main conjecture} and Proposition~\ref{intro necessity proposition} which deal with finite maximal type curves in higher dimensions: see \S\ref{nec sec} below.\medskip




\noindent\textit{Endpoint lacunary maximal estimates.} For the measure $\mu_{\gamma}$ introduced above, define the family of dyadic dilates $\mu_{\gamma}^k$ for $k \in \Z$ by
\begin{equation*}
\inn{\mu_{\gamma}^k}{f} = \inn{\mu_{\gamma}}{f(2^k\,\cdot\,)}  
\end{equation*}  
and consider the associated convolution operators $A_{\gamma}^kf := \mu_{\gamma}^k\ast f$. If $\gamma$ is of finite maximal type, then a well-known and classical result (see, for instance, \cite{DR1986}) states that the associated lacunary maximal function \begin{equation*}
    \mathcal{M}_{\gamma}f := \sup_{k \in \Z} |A_{\gamma}^k f|
\end{equation*}
is bounded on $L^p$ for all $1 < p \leq \infty$. A difficult problem is to understand the endpoint behaviour of these operators near $L^1$. By an off-the-shelf application of the main theorem from \cite{SW2011}, Corollary~\ref{finite type corollary} implies an endpoint bound for $\mathcal{M}_{\gamma}$ in the $n = 4$ case.

\begin{corollary}\label{dyadic maximal corollary} If $\gamma \colon I \to \R^4$ is of finite maximal type, then the lacunary maximal function $\mathcal{M}_{\gamma}$ maps the (standard isotropic) Hardy space $H^1(\R^4)$ to $L^{1, \infty}(\R^4)$. 
\end{corollary}

In particular, by \cite[Theorem 1.1]{SW2011}, Corollary~\ref{dyadic maximal corollary} follows from \textit{any} $L^p \to L^p_{1/p}$ bound for the associated averaging operator $A_{\gamma}$ for $2 \leq p < \infty$ (that is, one does not require $L^p \to L^p_{1/p}$ for the sharp range of $p$ for this application). Note that, prior to this paper, no such bounds $L^p$-Sobolev bounds were known for $n \geq 4$; thus the question of the $H^1(\R^n)$ to $L^{1, \infty}(\R^n)$ boundedness of lacunary maximal associated to finite maximal type (or even non-degenerate) curves remains open for $n \geq 5$.

\subsection*{Outline of the paper} This paper is structured as follows: 
\begin{itemize}
    \item In \S\ref{moment red sec} we discuss a simple reductions to a class of model curves. 
    \item In \S\ref{nec sec} we derive necessary conditions for $L^p$-Sobolev improving inequalities for our averaging operators. In particular, we establish Proposition~\ref{intro necessity proposition}. 
    \item In \S\S\ref{red aux sec}--\ref{sec:J=4} we present the proof of Theorem~\ref{non-degenerate theorem}.
    \item In \S\ref{sec:decoupling} we discuss certain decoupling inequalities used in the proof of Theorem~\ref{non-degenerate theorem}.
    \item There are three appendices which deal with various auxiliary results and technical lemmas used in the main argument. 
\end{itemize}




\subsection*{Notational conventions} Given a (possibly empty) list of objects $L$, for real numbers $A_p, B_p \geq 0$ depending on some Lebesgue exponent $p$ or dimension parameter $n$ the notation $A_p \lesssim_L B_p$, $A_p = O_L(B_p)$ or $B_p \gtrsim_L A_p$ signifies that $A_p \leq CB_p$ for some constant $C = C_{L,p,n} \geq 0$ depending on the objects in the list, $p$ and $n$. In addition, $A_p \sim_L B_p$ is used to signify that both $A_p \lesssim_L B_p$ and $A_p \gtrsim_L B_p$ hold. Given $a$, $b \in \R$ we write $a \wedge b:= \min \{a, b\}$ and  $a \vee b:=\max \{a,b\}$. 
The length of a multiindex $\alpha\in \bbN_0^n$ is given by $|\alpha|=\sum_{i=1}^n{\alpha_i}$.

\subsection*{Acknowledgements}
{The authors thank the American Institute of Mathematics for funding their collaboration through the SQuaRE program, also  supported in part  by the National Science Foundation. D.B. was partially supported by NSF grant DMS-1954479. S.G. was partially supported by  NSF grant DMS-1800274. A.S. was partially supported by  NSF grant DMS-1764295 and by a Simons fellowship.  This material is partly based upon work supported by the National Science Foundation under Grant No. DMS-1440140 while the authors were in residence at the Mathematical Sciences Research Institute in Berkeley, California, during the Spring 2017 semester.}




\section{Reduction to perturbations of the moment curve}\label{moment red sec}

A prototypical example of a smooth curve satisfying the non-degeneracy condition \eqref{eq:nondegenerate} is the \textit{moment curve} $\gamma_{\circ} \colon \R \to \R^n$, given by
\begin{equation*}
    \gamma_{\circ}(s) := \Big(s, \frac{s^2}{2}, \dots, \frac{s^n}{n!} \Big). 
\end{equation*}
Indeed, in this case the determinant appearing in \eqref{eq:nondegenerate} is everywhere equal to 1. Moreover, at small scales, any non-degenerate curve can be thought of as a perturbation of an affine image of $\gamma_{\circ}$. To see why this is so, fix a non-degenerate curve $\gamma \colon I \to  \R^n$ and $\sigma \in I$, $\lambda > 0$ such that $[\sigma - \lambda, \sigma+\lambda] \subseteq I$. Denote by $[\gamma]_{\sigma}$ the $n\times n$ matrix
\begin{equation*} 
    [\gamma]_{\sigma}:=
    \begin{bmatrix}
    \gamma^{(1)}(\sigma) & \cdots & \gamma^{(n)}(\sigma)
    \end{bmatrix},
\end{equation*}
where the vectors $\gamma^{(j)}(\sigma)$ are understood to be \textit{column} vectors. Note that this is precisely the matrix appearing in the definition of the non-degeneracy condition \eqref{eq:nondegenerate} and is therefore invertible by our hypothesis. It is also convenient to let $[\gamma]_{\sigma,\lambda}$ denote the $n \times n$ matrix
\begin{equation}\label{gamma transformation}
[\gamma]_{\sigma,\lambda} := [\gamma]_{\sigma} \cdot D_{\lambda},
\end{equation}
where $D_{\lambda}:=\text{diag}(\lambda, \dots, \lambda^n)$, the diagonal matrix with eigenvalues $\lambda$, $\lambda^2, \dots, \lambda^n$. Consider the portion of the curve $\gamma$ lying over the subinterval $[\sigma-\lambda, \sigma+\lambda]$. This is parametrised by the map  $s \mapsto \gamma(\sigma + \lambda s)$ for $s \in [-1,1]$.  The degree $n$ Taylor polynomial of $s \mapsto \gamma(\sigma + \lambda s)$ around $\sigma$ is given by
\begin{equation}\label{Taylor}
  s \mapsto  \gamma(\sigma) + [\gamma]_{\sigma,\lambda} \cdot \gamma_{\circ}(s),
\end{equation}
which is indeed an affine image of $\gamma_{\circ}$. Furthermore, by Taylor's theorem, the original curve $\gamma$ agrees with the polynomial curve \eqref{Taylor} to high order at $\sigma$. 

Inverting the affine transformation $x \mapsto  \gamma(\sigma) + [\gamma]_{\sigma,\lambda} \cdot x$ from \eqref{Taylor}, we can map the portion of $\gamma$ over $[\sigma - \lambda, \sigma + \lambda]$ to a small perturbation of the moment curve. 

\begin{definition}\label{rescaled curve def} Let $\gamma \in C^{n+1}(I;\R^{n})$ be a non-degenerate curve and $\sigma \in I, \lambda>0$ be such that $[\sigma-\lambda, \sigma+ \lambda] \subseteq I$. The \textit{$(\sigma,\lambda)$-rescaling of $\gamma$} is the curve $\gamma_{\sigma,\lambda} \in C^{n+1}([-1,1];\R^{n})$ given by
\begin{equation*}
    \gamma_{\sigma,\lambda}(s) := [\gamma]_{\sigma,\lambda}^{-1}\big( \gamma(\sigma+\lambda s) - \gamma(\sigma) \big).
\end{equation*}
\end{definition}

It follows from the preceding discussion that 
\begin{equation*}
   \gamma_{\sigma,\lambda}(s) = \gamma_{\circ}(s) + [\gamma]_{\sigma,\lambda}^{-1} \mathcal{E}_{\gamma,\sigma,\lambda}(s)  
\end{equation*}
where $\mathcal{E}_{\gamma,\sigma,\lambda}$ is the remainder term for the Taylor expansion \eqref{Taylor}. In particular, if $\gamma$ satisfies the non-degeneracy condition \eqref{eq:nondegenerate} with constant $c_0$, then
\begin{equation*}
    \|   \gamma_{\sigma,\lambda} - \gamma_{\circ} \|_{C^{n+1}([-1,1];\R^n)} \lesssim c_0^{-1} \lambda \,  \| \gamma \|_{C^{n+1}(I)}^n.
\end{equation*}
Thus, if $\lambda>0$ is chosen to be small enough, then the rescaled curve $\gamma_{\sigma,\lambda}$ is a minor perturbation of the moment curve. In particular, given any $0 < \delta < 1$, we can choose $\lambda$ so as to ensure that $\gamma_{\sigma,\lambda}$ belongs to the following class of \textit{model curves}.

\begin{definition} Given $n \geq 2$ and $0 < \delta < 1$, let $\mathfrak{G}_n(\delta)$ denote the class of all smooth curves $\gamma \colon [-1, 1] \to \R^n$ that satisfy the following conditions: 
\begin{enumerate}[i)]
    \item $\gamma(0) = 0$ and $\gamma^{(j)}(0) = \vec{e}_j$ for $1 \leq j \leq n$;
    \item $\|\gamma - \gamma_{\circ}\|_{C^{n+1}([-1,1])} \leq \delta$.
\end{enumerate}
Here $\vec{e}_j$ denotes the $j$th standard Euclidean basis vector and
\begin{equation*}
    \|\gamma\|_{C^{n+1}(I)} := \max_{1 \leq j \leq n + 1} \sup_{s \in I} |\gamma^{(j)}(s)| \qquad \textrm{for all $\gamma \in C^{n+1}(I;\R^n)$.}
\end{equation*}
\end{definition}

Given any $\gamma \in \mathfrak{G}_n(\delta)$, condition ii) and the multilinearity of the determinant ensures that $\det[\gamma]_s = \det[\gamma_{\circ}]_s + O(\delta) = 1 + O(\delta)$. Thus, there exists a dimensional constant $c_n > 0$ such that if $0 < \delta < c_n$, then any curve $\gamma \in \mathfrak{G}_n(\delta)$ is non-degenerate and, moreover, satisfies $\det [ \gamma]_s \geq 1/2$. Henceforth, it is always assumed that $\delta > 0$ satisfies this condition, which we express succinctly as $0 < \delta \ll 1$.

Turning back to the Sobolev improving problem for the averages $A_{\gamma}$, the above observations facilitate a reduction to the class of model curves. To precisely describe this reduction, it is useful to make the choice of cutoff function explicit in the notation by writing $A[\gamma,\chi]$ for the operator $A_{\gamma}$ as defined in \eqref{averaging operator}.

\begin{proposition}\label{quant nondeg prop} Let $\gamma \colon I \to \R^n$ be a non-degenerate curve, $\chi \in C^{\infty}_c(\R)$ be supported on the interior of $I$ and $0 < \delta \ll 1$. There exists some $\gamma^* \in \mathfrak{G}_n(\delta)$ and $\chi^* \in C^{\infty}_c(\R)$ such that 
\begin{equation*}
    \|A[\gamma,\chi]\|_{L^p(\R^n) \to L^p_\alpha(\R^n)} \sim_{\gamma, \chi, \delta, p, \alpha} \|A[\gamma^*,\chi^*]\|_{L^p(\R^n) \to L^p_\alpha(\R^n)}
\end{equation*}
for all $1 \leq p < \infty$ and $0 \leq \alpha \leq 1$. Furthermore, $\chi^*$ may be chosen to satisfy $\supp \chi^* \subseteq [-\delta,\delta]$.
\end{proposition}

\begin{proof} The proof follows by decomposing the domain of $\gamma$ into small intervals and applying the rescaling described in Definition~\ref{rescaled curve def} on each interval. This decomposition in $s$ induces a decomposition of the derived operator $(1 - \Delta)^{\alpha/2} A[\gamma, \chi]$. The upper bound then follows from the triangle inequality and the stability of the estimates under affine transformation (together with a simple pigeonholing argument).\smallskip

The proof of the lower bound is more subtle since one must take into account possible cancellation between the different pieces of the decomposition. To get around this, we observe that 
\begin{equation}\label{smaller cutoff}
    \|A[\gamma, \chi_0]\|_{L^p(\R^n) \to L^p_{\alpha}(\R^n)} \lesssim_{\gamma, \chi_0}  \|A[\gamma, \chi_1]\|_{L^p(\R^n) \to L^p_{\alpha}(\R^n)}
\end{equation}
holds whenever $\chi_0$, $\chi_1 \in C^{\infty}_c(\R)$ are supported in $I$ and $\chi_1(s) = 1$ for all $s \in \supp \chi_0$. Once this is established, it is possible to localise in $s$ and rescale to deduce the desired bound. 

To prove \eqref{smaller cutoff} note, after possibly applying a translation and a dilation, one may write
\begin{equation*}
   A[\gamma, \chi_0]f(x) = \int_{\R} f(x-\gamma(s)) \tilde{\chi}_0\circ \gamma(s) \chi_1(s)\,\ud s 
\end{equation*}
where the function $\tilde{\chi}_0 \in C^{\infty}_c(\R^n)$ is supported in $[-\pi, \pi]^n$. Consequently, by performing a Fourier series decomposition,
\begin{equation*}
    \tilde{\chi}_0\circ \gamma(s) = \frac{1}{(2 \pi)^n} \sum_{k \in \Z^n} a_k e^{i \inn{x}{k}}  e^{-i\inn{x - \gamma(s)}{k}}
\end{equation*}
where the sequence $(a_k)_{k\in \Z^n}$ of Fourier coefficients is rapidly decaying. Thus, if $\mathrm{Mod}_k$ denotes the modulation operator $\mathrm{Mod}_k \,g(x) := e^{i \inn{x}{k}}g(x)$, then
\begin{equation*}
  A[\gamma, \chi_0]f(x) =  \frac{1}{(2 \pi)^n} \sum_{k \in \Z^n}  a_k \cdot \mathrm{Mod}_k \circ A[\gamma, \chi_1]\circ \mathrm{Mod}_{-k}\,f(x).
\end{equation*}
By analytic interpolation, it follows that 
\begin{equation*}
    \|\mathrm{Mod}_k\|_{L^p_{\alpha}(\R^n) \to L^p_{\alpha}(\R^n)} \lesssim (1 + |k|)^{\alpha} \qquad \textrm{for all $0 \leq \alpha \leq 1$}
\end{equation*}
and therefore
\begin{align*}
    \|A[\gamma, \chi_0]f\|_{L^p_{\alpha}(\R^n)} &\lesssim \sum_{k \in \Z^n} |a_k| (1 + |k|)^{\alpha} \cdot \|A[\gamma, \chi_1]\circ \mathrm{Mod}_{-k}\,f\|_{L^p_{\alpha}(\R^n)} \\
    &\lesssim_{\gamma, \chi_0} \|A[\gamma, \chi_1]\|_{L^p(\R^n) \to L^p_{\alpha}(\R^n)}\|f\|_{L^p(\R^n)},
\end{align*}
using the rapid decay of the Fourier coefficients. 
\end{proof}

As a consequence of Proposition~\ref{quant nondeg prop}, it suffices to fix $\delta_0 > 0$ and prove Theorem~\ref{non-degenerate theorem} and Proposition~\ref{intro necessity proposition} in the special case where $\gamma \in \mathfrak{G}_4(\delta_0)$ and $\supp \chi \subseteq I_0 := [-\delta_0,\delta_0]$. Thus, henceforth, we work with some fixed $\delta_0$, chosen to satisfy the forthcoming requirements of the proofs. For the sake of concreteness, the choice of $\delta_0 := 10^{-10^5}$ is more than enough for our purposes.




\section{Necessary conditions}\label{nec sec} 

\subsection{General \texorpdfstring{$L^p \to L^p_{\alpha}$}{} estimates} If $\gamma \colon I \to \R^n$ is of maximal type $d$, then the van der Corput lemma shows that the Fourier transform of any smooth density $\mu_{\gamma}$ on $\gamma$ satisfies 
\begin{equation}\label{Fourier decay}
    |\hat{\mu}_{\gamma}(\xi)| \lesssim_{\gamma} (1 + |\xi|)^{-1/d}.
\end{equation}
This readily implies that  
\begin{equation}\label{L2 Sobolev}
    \|A_{\gamma}f\|_{L^2_{1/d}(\R^n)} \lesssim_{\gamma} \|f\|_{L^2(\R^n)}.
\end{equation}
Consider the case where $\gamma$ is non-degenerate, so that $d = n$. By interpolating against \eqref{L2 Sobolev}, Conjecture~\ref{main conjecture} formally implies that $A_{\gamma}$ maps $L^p$ to $L^p_{\alpha}$ for all $p \geq 2$ and
\begin{equation}\label{conjectured regularity range}
 \alpha <   \alpha_{\mathrm{cr}}(p) := \min \Big\{ \frac{1}{n}\Big(\frac{1}{2} + \frac{1}{p} \Big), \frac{1}{p} \Big\},
\end{equation}
with the equality case also holding in the restricted range $p > 2n - 2$. It is an interesting question what happens at the endpoint in the range $2 < p \leq 2n - 2$.


\begin{figure}
\begin{tikzpicture}[scale=4] 

\begin{scope}[scale=2]
\draw[thick,->] (-.05,0) -- (1.05,0) node[below] {$ \frac{1}{p}$};
\draw[thick,->] (0,-.05) -- (0,0.4) node[left=0.15cm] {$\alpha$};

\draw (-0.025,0.33) -- (0.025,0.33) node[left=0.35cm] {$ \frac{1}{n}$};

\draw (-0.025, 0.25) -- (0.025, 0.25) node[left=0.35cm] {$\frac{1}{2n-2}$};

\draw (0.5,-0.025) -- (0.5,.025) node[below= 0.35cm] {$ \frac{1}{2}$}; 

\draw (1,-0.025) -- (1,.025); 

\draw (0,0) -- (0.5,0.33) -- (1,0) ; 

\draw (0,0) -- (0.25,0.25)  -- (0.5,0.33)-- (0.75,0.25) -- (1,0) ; 

\node[circle,draw=black, fill=black, inner sep=0pt,minimum size=3pt] at (0.5,0.33) {};

\node[circle,draw=black, fill=white, inner sep=0pt,minimum size=4pt] at (0.25,0.25) {};

\node[circle,draw=black, fill=white, inner sep=0pt,minimum size=4pt] at (0.75,0.25) {};

\draw (0.25, -0.025) -- (0.25, 0.025) node[below=0.35cm] {$\frac{1}{2n-2}$}; 

\draw (0.75, -0.025) -- (0.75, 0.025) node[below=0.35cm] {$\frac{2n-3}{2n-2}$};

\fill[pattern=north west lines, pattern color=blue] (0,0) -- (0.25,0.25)  -- (0.5,0.33)-- (0.75,0.25) -- (1,0) -- (0.5,0.33) --(0,0) ;

\fill[pattern=north west lines, pattern color=blue!40] (0,0) -- (0.5,0.33) -- (1,0) --(0,0) ;

\node[circle,draw=black, fill=black, inner sep=0pt,minimum size=3pt] at (0,0) {};

\node[circle,draw=black, fill=black, inner sep=0pt,minimum size=3pt] at (1,0) {};

\end{scope}

\end{tikzpicture}

\caption{Conjectured range of $A_{\gamma} \colon L^p(\R^n) \to L^p_{\alpha}(\R^n)$ boundedness for $\gamma$ non-degenerate. The inner triangle follows from the elementary $L^2$ estimate. The goal is to establish the $L^p(\R^n) \to L^p_{1/p}(\R^n)$ bound at the `kink' point $p_{\mathrm{cr}} = 2n-2$ (or, equivalently, $p_{\mathrm{cr}}' = \tfrac{2n-3}{2n-2}$). }
\label{conjectured range} 

\end{figure}
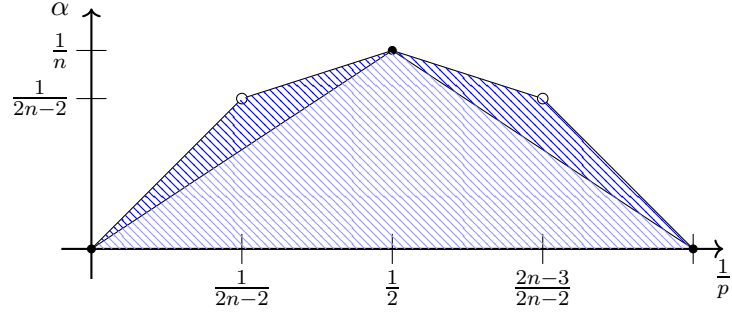

%

The range of conjectured bounds is represented in Figure~\ref{conjectured range}. The two constraints appearing in the definition of the critical regularity exponent $\alpha_{\mathrm{cr}}(p)$ agree precisely when $p$ corresponds to the critical Lebesgue exponent
\begin{equation*}
    p_{\mathrm{cr}} := 2n - 2,
\end{equation*}
which manifests as a `kink' in the $L^p$-Sobolev diagram.

By a simple scaling argument (see, for instance, \cite[pp.81-82]{PS2007}), Conjecture~\ref{main conjecture} further implies bounds for $A_{\gamma}$ under a finite type hypothesis. In view of Corollary~\ref{finite type corollary}, it is reasonable to conjecture the following. 


\begin{figure}
\begin{tikzpicture}[scale=4] 

\begin{scope}[scale=2]
\draw[thick,->] (-.05,0) -- (1.05,0) node[below] {$ \frac{1}{p}$};
\draw[thick,->] (0,-.05) -- (0,0.4) node[left=0.15cm] {$\alpha$};

\draw (-0.025,0.30) -- (0.025,0.30) node[left=0.35cm] {$ \frac{1}{d}$};

\draw (-0.025, 0.23) -- (0.025, 0.23) node[left=0.35cm] {$\frac{1}{2n-2}$};

\draw (0.5,-0.025) -- (0.5,.025) node[below= 0.35cm] {$ \frac{1}{2}$}; 

\draw (1,-0.025) -- (1,.025);

\draw (0,0) --(0.23,0.23)  ;
\draw[dashed] (0.23,0.23) -- (0.37,0.30) ;
\draw  (0.37,0.30) -- (0.63,0.30);
\draw[dashed]  (0.63,0.30) --(0.77,0.23) ;
\draw (0.77,0.23) -- (1,0) ; 

\node[circle,draw=black, fill=black, inner sep=0pt,minimum size=3pt] at (0.5, 0.3) {};

\node[circle,draw=black, fill=white, inner sep=0pt,minimum size=4pt] at (0.37, 0.30) {};

\node[circle,draw=black, fill=white, inner sep=0pt,minimum size=4pt] at (0.63,0.30) {};

\node[circle,draw=black, fill=white, inner sep=0pt,minimum size=4pt] at(0.23,0.23) {};

\node[circle,draw=black, fill=white, inner sep=0pt,minimum size=4pt] at(0.77,0.23) {};

\draw (0.23, -0.025) -- (0.23, 0.025) node[below=0.35cm] {$\frac{1}{2n-2}$}; 

\draw (0.37, -0.025) -- (0.37, 0.025) node[below=0.35cm] {$\frac{2n-d}{2d}$};

\draw (0.63, -0.025) -- (0.63, 0.025) node[below=0.35cm] {$\frac{3d-2n}{2d}$}; 

\draw (0.77, -0.025) -- (0.77, 0.025) node[below=0.35cm] {$\frac{2n-3}{2n-2}$}; 

\node[circle,draw=black, fill=black, inner sep=0pt,minimum size=3pt] at (0,0) {};

\node[circle,draw=black, fill=black, inner sep=0pt,minimum size=3pt] at (1,0) {};

\fill[pattern=north west lines, pattern color=blue] (0,0) --(0.23,0.23)  -- (0.37,0.30)-- (0.63,0.30) --(0.77,0.23) -- (1,0) ;

\end{scope}

\end{tikzpicture}

\begin{tikzpicture}[scale=4] 

\begin{scope}[scale=2]
\draw[thick,->] (-.05,0) -- (1.05,0) node[below] {$ \frac{1}{p}$};
\draw[thick,->] (0,-.05) -- (0,0.3) node[left=0.15cm] {$\alpha$};

\draw (-0.025, 0.15) -- (0.025, 0.15) node[left=0.35cm] {$\frac{1}{d}$};

\draw (0.15, -0.025) -- (0.15, 0.025) node[below=0.35cm] {$\frac{1}{d}$};

\draw (0.85, -0.025) -- (0.85, 0.025) node[below=0.35cm] {$\frac{d-1}{d}$};

\draw (0.5,-0.025) -- (0.5,.025) node[below= 0.35cm] {$ \frac{1}{2}$}; 

\draw (1,-0.025) -- (1,.025);

\draw (0,0) -- (0.15,0.15) ; 
\draw (0.15,0.15) -- (0.85,0.15) ; 
\draw (0.85,0.15) -- (1,0) ;

\fill[pattern=north west lines, pattern color=blue] (0,0) -- (0.15,0.15)  -- (0.85,0.15) -- (1,0)  ;

\node[circle,draw=black, fill=black, inner sep=0pt,minimum size=3pt] at (0.5, 0.15) {};

\node[circle,draw=black, fill=white, inner sep=0pt,minimum size=4pt] at (0.15,0.15) {};

\node[circle,draw=black, fill=white, inner sep=0pt,minimum size=4pt] at (0.85,0.15) {};

\node[circle,draw=black, fill=black, inner sep=0pt,minimum size=3pt] at (0,0) {};

\node[circle,draw=black, fill=black, inner sep=0pt,minimum size=3pt] at (1,0) {};

\end{scope}

\end{tikzpicture}

\caption{Conjectured range of $A_{\gamma} \colon L^p(\R^n) \to L^p_{\alpha}(\R^n)$ boundedness for $\gamma$ of maximal type $d$. The upper diagram corresponds to $d < 2n-2$ whilst the lower diagram corresponds to $d \geq 2n-2$.}
\label{finite type conjectured range} 

\end{figure}
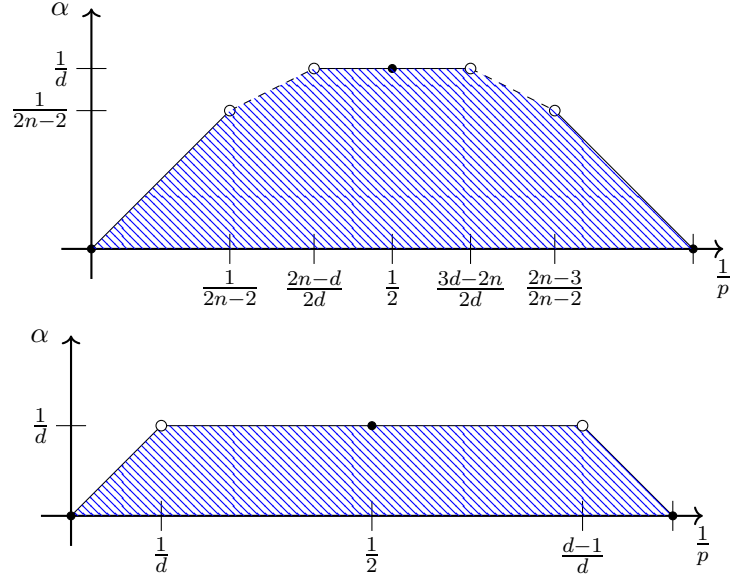

%

\begin{conjecture}\label{finite type conjecture} If $\gamma \colon I \to \R^n$ is of maximal type $d$, then the operator $A_{\gamma}$ maps $L^p$ to $L^p_{\alpha}$ for all $p \geq 2$ and 
\begin{equation}\label{finite type regularity range}
 \alpha \leq   \alpha_{\mathrm{cr}}(d;p) := \min \Big\{ \alpha_{\mathrm{cr}}(p), \frac{1}{d} \Big\}
\end{equation}
with strict inequality if $\min \{2n - 2, d\} \leq p \leq \max \{2n - 2, d\}$.
\end{conjecture} 

The range of conjectured bounds is represented in Figure~\ref{finite type conjectured range}.

\begin{remark}
Using the fact that $(I - \Delta)^{\alpha/2}
\colon L^p_{\alpha +\beta}(\R^n) \to L^p_{ \beta}(\R^n)$ is an isomorphism together with a duality argument, any $L^p \to L^p_{\alpha}$ estimate for $A_{\gamma}$ immediately implies a corresponding $L^{p'} \to L^{p'}_{\alpha}$ estimate. 
\end{remark}

The condition $\alpha \leq 1/d$ is clearly necessary. Indeed, by duality and interpolation, any $L^p \to L^p_{\alpha}$ estimate implies an $L^2 \to L^2_{\alpha}$ estimate for the same value of $\alpha$. However, a slight refinement of \eqref{Fourier decay} shows that the $L^2$ estimate \eqref{L2 Sobolev} is sharp in the sense that the regularity exponent on the left-hand side cannot be taken larger than $1/d$.

\subsection{Band-limited examples} The remainder of this section discusses the necessity of the conditions \eqref{conjectured regularity range} and \eqref{finite type regularity range}. To begin, given $\lambda > 0$, consider the family of band-limited Schwartz functions
\begin{equation*}
\cZ_\la : =\big\{f\in \cS(\bbR^n): \supp\hat{f}\subset\{\xi \in \hat{\R}^n :\la/2\le|\xi|\le 2\la\} \big\}.
\end{equation*}
By elementary Sobolev space theory, the desired necessary conditions are a consequence of the following proposition.

\begin{proposition}\label{necessity proposition} If $\gamma \colon I \to \R^n$ is a smooth curve satisfying the non-degeneracy hypothesis \eqref{eq:nondegenerate} and $p\ge 2$, then
\begin{equation*}
   \sup\big\{ \|A_{\gamma}f\|_{L^p(\R^n)}: f\in L^p\cap\cZ_\la, \quad\|f\|_{L^p(\R^n)} = 1\big \} \gtrsim_{p,\gamma} \la^{-\alpha_{\mathrm{cr}}(p)}. 
\end{equation*}
\end{proposition}

This directly implies Proposition~\ref{intro necessity proposition} and, moreover, shows that the $(p,\alpha)$-ranges in \eqref{conjectured regularity range} and Conjecture~\ref{finite type conjecture} are optimal up to endpoints.\footnote{If $\gamma$ is finite type curve, then the points $s$ for which the type of $\gamma$ at $s$ is strictly larger than $d$ are isolated. Consequently, any necessary condition for the non-degenerate problem is automatically a necessary condition for the finite type problem. The necessity of the additional constraint $\alpha \leq 1/d$ is discussed in the previous subsection.} 

Proposition~\ref{necessity proposition} is based on testing the estimate against two examples, corresponding to the two constraints inherent in the minimum appearing in the definition of $\alpha_{\mathrm{crit}}(p)$.

\subsection{Dimensional constraint: \texorpdfstring{$\alpha \leq 1/p$}{}}  The condition $\alpha \leq 1/p$ is well-known and appears to be folkloric; in lieu of a precise reference, the details are given presently. 

\begin{lemma}\label{1/p nec lemma} If $\gamma \colon I \to \R^n$ is a smooth curve and $p \geq 2$, then 
\begin{equation*}
    \sup\big\{ \|A_{\gamma} f\|_{L^p(\R^n)}: f\in L^p(\R^n)\cap\cZ_\la, \quad\|f\|_{L^p(\R^n)}=1\big \} \gtrsim_{p,\gamma} \la^{-1/p}.
\end{equation*}
\end{lemma}

\begin{proof} Since the operator $A_{\gamma}$ is self-adjoint and commutes with frequency projections, given $1 \leq p \leq 2$ it suffices to show
\begin{equation*}
     \sup\big\{ \|A_{\gamma} f\|_{L^p(\R^n)}: f\in L^p(\R^n)\cap\cZ_\la, \quad\|f\|_{L^p(\R^n)}=1\big \} \gtrsim_{p,\gamma} \la^{-1/p'}.
\end{equation*}

Fix $\beta \in C^{\infty}_c(\hat{\R}^n)$ a real-valued even function with (inverse) Fourier transform $\check{\beta}$ satisfying $\check{\beta}(0) = 1$ and
\begin{equation}\label{1/p nec 1}
    \supp \beta \subseteq \{\xi \in \hat{\R}^n : 1/2 \leq |\xi| \leq 2 \}.
\end{equation}
In addition, let $\psi \in C^{\infty}_c(\R^n)$ be non-zero, non-negative and supported in a ball centred at the origin of radius $c$, where $0 < c < 1$ is a sufficiently small constant (independent of $\lambda$) chosen to satisfy the requirements of the forthcoming argument. With these bump functions define
\begin{equation*}
    f := (\beta_{\lambda^{-1}})\;\widecheck{}\; \ast \psi_{\lambda}
\end{equation*}
where $\beta_{\lambda^{-1}}(\xi) := \beta(\lambda^{-1} \xi)$ and $\psi_{\lambda}(x) := \psi(\lambda x)$. The condition \eqref{1/p nec 1} implies $f \in L^p(\R^n) \cap \mathcal{Z}_{\la}$, whilst direct calculation shows  that
\begin{equation}\label{1/p nec 2}
    \|f\|_{L^p(\R^n)} \sim \lambda^{-n/p}.
\end{equation}

By a simple computation, $A_{\gamma}f = K^{\lambda} \ast \psi_{\lambda}$ where
\begin{equation*}
    K^{\lambda}(x) := \lambda^n \int_{\R} \check{\beta}\big(\lambda(x - \gamma(s))\big) \,\chi(s)\ud s.
\end{equation*}
The key claim is that, provided $0 < c < 1$ is chosen sufficiently small (independently of $\lambda$), the pointwise inequality
\begin{equation}\label{1/p nec 3}
    K^{\lambda} \ast \psi_{\lambda}(x) \gtrsim \lambda^{-1} \qquad \textrm{for all $x \in \mathcal{N}_{c\lambda^{-1}}(\gamma)$}
\end{equation}
holds, where $\mathcal{N}_{c\lambda^{-1}}(\gamma)$ denotes the $c\lambda^{-1}$-neighbourhood of the curve 
\begin{equation*}
    \{\gamma(s) : s \in \supp \chi \}.
\end{equation*} 
To see this, choose $c$ sufficiently small so that $\check{\beta}$ is bounded away from zero on a ball of radius $10c$ centred at the origin. If $x \in \mathcal{N}_{c\lambda^{-1}}(\gamma)$, then there exists some $s_0 \in \supp \chi$ such that 
\begin{equation*}
    |x - y - \gamma(s)| < 10c\lambda^{-1} \qquad\textrm{whenever $|s - s_0| \lesssim_{\gamma} \lambda^{-1}$ and $|y| \leq c \lambda^{-1}$,}
\end{equation*}
from which \eqref{1/p nec 3} follows. 

Combining \eqref{1/p nec 2} and \eqref{1/p nec 3}, one concludes that
\begin{equation*}
    \sup_{f\in L^p(\R^n)\cap\cZ_\la} \frac{\|A_{\gamma}f\|_{L^p(\R^n)}}{\|f\|_{L^p(\R^n)}} \gtrsim \frac{\lambda^{-1} \lambda^{-(n-1)/p}}{\lambda^{-n/p}}= \la^{-1/p'},
\end{equation*}
as desired. 
\end{proof}

\begin{remark}  More generally, suppose $A_{\Sigma}$ is an averaging operator defined as in \eqref{averaging operator} but now with respect to $\Sigma$ a (regular parametrisation of a) surface in $\R^n$ of arbitrary dimension. Then 
\begin{equation*}
    \sup\big\{ \|A_{\Sigma} f\|_{L^p(\R^n)}: f\in L^p(\R^n) \cap\cZ_\la, \quad\|f\|_{L^p(\R^n)}=1\big \} \gtrsim_{p,\Sigma} \la^{-\dim \Sigma /p}.
\end{equation*}
This general necessary condition follows from the proof of Lemma~\ref{1/p nec lemma} \textit{mutatis mutandis}. Further generalisations hold for appropriate classes of variable coefficient averaging operators: see, for instance, \cite{Bentsen}. 
\end{remark}

\subsection{Fourier decay constraint: \texorpdfstring{$\alpha \leq \frac{1}{n}\big(\frac{1}{2} + \frac{1}{p}\big)$}{}.} Establishing the second condition is  more involved. Here, in contrast with Lemma~\ref{1/p nec lemma}, the non-degeneracy hypothesis \eqref{eq:nondegenerate} plays a r\^ole via certain refinements of the Fourier decay estimate \eqref{Fourier decay}.  

Recall the desired bound. 

\begin{proposition}\label{Fourier nec prop} If $\gamma \colon I \to \R^n$ is a smooth curve satisfying the non-degeneracy hypothesis \eqref{eq:nondegenerate} and $p\ge 2$, then
\begin{equation*}
   \sup\big\{ \|A_{\gamma}f\|_{L^p(\R^n)}: f\in L^p(\R^n) \cap\cZ_\la, \quad\|f\|_{L^p(\R^n)}=1\big \} \gtrsim_{p,\gamma} \la^{  -\frac 1n(\frac 12+\frac 1p)}. 
\end{equation*}
\end{proposition}

This conclusion was  shown in three dimensions by Oberlin and Smith 
\cite{OS1999} for the model example of the helix in $\bbR^3$, 
$t\mapsto (\cos t, \sin t, t)$, by using DeLeeuw's restriction theorem and an analysis of a Bessel multiplier in $\bbR^2$.
Here the more general statement in  Proposition~\ref{Fourier nec prop} is shown by combining a sharp example of Wolff \cite{Wolff2000} for $\ell^p$-decoupling inequalities with a stationary phase analysis of the Fourier multiplier $\hat{\mu}_{\gamma}$.

The proof of Proposition~\ref{Fourier nec prop} is broken into stages.

\subsubsection*{\it The worst decay cone} At any given large scale, the decay estimate \eqref{Fourier decay} is only sharp for $\xi$ belonging to a narrow region around a low-dimensional cone in the frequency space. To prove Proposition~\ref{Fourier nec prop}, it is natural to test the $L^p$-Sobolev estimate against functions which are Fourier supported in a neighbourhood of this `worst decay cone'.

By Proposition~\ref{quant nondeg prop} we may assume without loss of generality that $\gamma \in \mathfrak{G}_n(\delta_0)$ for some small $0 < \delta_0 \ll 1$ and that the cutoff $\chi$ in the definition of $A_{\gamma}$ is supported in $I_0 = [-\delta_0, \delta_0]$. In view of the van der Corput lemma, the worst decay cone should correspond to the $\xi$ for which the derivatives $\inn{\gamma^{(j)}(s)}{\xi}$, $1 \leq j \leq n-1$, all simultaneously vanish for some $s \in I_0$. In order to describe this region, first note that 
\begin{equation*}
    \inn{\gamma^{(n-1)}(s_0)}{\xi_0} = 0 \quad \textrm{and} \quad \frac{\partial}{\partial s}\inn{\gamma^{(n-1)}(s)}{\xi} \Big|_{\substack{s=s_0 \\ \xi = \xi_0}} = 1
\end{equation*}
for $(s_0, \xi_0) = (0, \vec{e}_n)$, by the reduction $\gamma^{(j)}(0) = \vec{e}_j$ for $1 \leq j \leq n$. Consequently, provided the support of $\chi$ is chosen sufficiently small, by the implicit function theorem and homogeneity there exists a constant $c > 0$ and a smooth mapping 
\begin{equation*}
    \theta \colon \Xi  \to I_0, \qquad \textrm{where} \quad \Xi := \big\{ \xi = (\xi', \xi_n) \in \hat{\R}^n \backslash \{0\}: |\xi'| \leq c|\xi_n| \big\},
\end{equation*}
such that $s = \theta(\xi)$ is the unique solution in $I$ to the equation $\inn{\gamma^{(n-1)}(s)}{\xi} = 0$ whenever $\xi \in \Xi$. Note that  $\theta$ is homogeneous of degree one.

Further consider the system of $n$ equations in $n+1$ variables given by 
\Be\label{bad cone}
\begin{cases} 
&\inn{\gamma^{(j)}(s)}{\xi}=0 \qquad \text{for $1 \leq j \leq n-1$,}
\\
& \xi_n=1.
\end{cases}
\Ee 
Again, by the reduction $\gamma^{(j)}(0) = \vec{e}_j$ for $1 \leq j \leq n$, this can be solved for suitably localised $\xi$ using the implicit function theorem, expressing $s$, $\xi_1$, ... $\xi_{n-2}$ as  functions of $\xi_{n-1}$. Thus \eqref{bad cone} holds if and only if
\begin{subequations}
\Be\label{solving on the bad cone} 
 \begin{aligned} 
\xi_i&=\Gamma_i(\xi_{n-1}), \qquad 1 \leq i \leq n-2, \\s&=\theta (\Gamma_1(\xi_{n-1}),\dots, \Gamma_{n-2}(\xi_{n-1}), \xi_{n-1},1),
\end{aligned}
\Ee for some smooth functions $\Gamma_i$, $i=1,\dots, n-2$  satisfying $\Gamma_i(0)=0$. 
On $I$ we form the  $\bbR^n$-valued  function $\tau\mapsto \Gamma(\tau)$
with the first $n-2$ components as defined in \eqref{solving on the bad cone} and 
\Be\label{klast-two-comp}\Gamma_{n-1}(\tau) := \tau, \qquad \Gamma_n(\tau) := 1.
\Ee \end{subequations}
With this definition, the formul\ae\, in \eqref{solving on the bad cone} can be succinctly expressed as
\begin{equation*}
    \xi = \Gamma(\xi_{n-1}), \qquad 
    s = \theta\circ \Gamma(\xi_{n-1}).
\end{equation*}
Moreover, the `worst decay cone' can then be defined as the cone generated by the curve $\Gamma$, given by
\begin{equation*}
    \mathcal{C} := \big\{ \lambda \Gamma(\tau) : \lambda > 0 \textrm{ and } \tau \in I \big\}.
\end{equation*}

\begin{remark} For the model case $\gamma(s) =\sum_{i=1}^n \frac{s^i}{i!} \vec{e}_i$ one may explicitly compute that \[\Gamma(\tau)= \sum_{i=1}^n \frac{(- \tau)^{n-i}}{(n-i)!}\vec{e}_i.\]
\end{remark}

\subsubsection*{\it The Wolff example revisited} In analogy with the example in \cite{Wolff2000}, here we consider functions with Fourier support on a union of balls with centres lying on the worst decay cone $\mathcal{C}$. To this end, let $\varepsilon > 0$ be a small dimensional constant, chosen to satisfy the forthcoming requirements of the argument, and 
\begin{equation*}
   \fN_\eps(\la) := \Z \cap \{s \in \R : |s| \le\eps\la^{1/n}\} .
\end{equation*}
The centres of the aforementioned balls are then given by
\Be\label{choice-xi-nu}\xi^\nu := \la \Gamma  (\nu\la^{-1/n}) \qquad \textrm{for all $\nu \in \fN_\eps(\la)$.}
\Ee

Fix $\eta \in C_c^{\infty}(\hat{\R}^n)$ satisfying $\eta(\xi) = 1$ if $|\xi| \leq 1/2$ and $\eta(\xi) = 0$ if $|\xi| \geq 1$. 
Let  $0 < \rho < 1$ be another dimensional constant, again chosen small enough to satisfy the forthcoming requirements of the argument, and define Schwartz functions $g_{\nu}$ for $\nu \in \fN_\eps(\la)$  via the Fourier transform by
\begin{equation}\label{nec 1.1}
    \hat{g}_{\nu}(\xi) := \eta\big( \la^{-1/n}\rho^{-1}(\xi - \xi^{\nu})\big).
\end{equation}
We consider randomised sums of the functions \eqref{nec 1.1}. In particular, set
\Be \label{grandom}
 g^\om(x) := \sum_{\nu \in \fN_\eps(\la)} r_\nu(\om) g_{\nu}(x) \qquad \textrm{for $\om\in [0,1]$,}
\Ee 
where $\{r_\nu\}_{\nu=1}^\infty$ is the  sequence of Rademacher functions. 
We claim
\Be\label{glowerbound} 
\Big(\int_0^1\|g^\om \|_{L^p(\bbR^n)}^p \,\ud\om \Big)^{1/p}
 \sim \la^{1-\frac 1p+\frac 1{2n}}.
\Ee
To prove  this we apply Fubini's theorem and Khinchine's inequality (see, for instance, \cite[Appendix D]{Stein1970}) to see that 
 the left hand side is
 \eqref{glowerbound} is equal to
 \[\Big\|\Big(\int_0^1|g^\om|^p \,\ud\om\Big)^{1/p} \Big\|_{L^p(\bbR^n)}\sim
\Big\|\Big(\sum_{\nu\in \fN_\eps(\la)} |g_\nu|^2 \Big)^{1/2}  \Big\|_{L^p(\bbR^n)}.\]
The right-hand side of the last display is equal to
\begin{align*}
\Big\|\Big(\sum_{\nu\in \fN_\eps(\la)} |\la \rho^n \check{\eta} (\la^{1/n}\rho\,\cdot\,)|^2 \Big)^{1/2}  \Big\|_{L^p(\bbR^n)}&= \big[\# \fN_\eps(\la)\big]^{1/2} \|\la\rho^n \check{\eta}(\la^{1/n}\rho\,\cdot\,)\|_{L^p(\bbR^n)} \\ &\sim
\la^{\frac {1}{2n}+1-\frac 1p},
\end{align*}
which yields \eqref{glowerbound}. The above estimates depend on $\rho$ , but since this parameter is chosen to be a dimensional constant (independently of $\la$) this dependence is suppressed.  Also note that so far the argument is independent of the choice of the $\xi_\nu$.

\subsubsection*{Asymptotics} The next step is to study the behaviour of the multiplier $\hat{\mu}_{\gamma}$ near the support of the $\hat{g}_\nu$. The key result is Lemma~\ref{asy-lemma} below, which relies on the asymptotics of $\hat{\mu}_{\gamma}$ near the worst decay cone and the observation that the functions $\hat{g}_\nu$ with the choice of $\xi^\nu$ as in \eqref{choice-xi-nu} are supported near that cone.

Set  $\hat{g}_{+,\nu}(\xi) := \eta_+(\la^{-1/n}\rho^{-1}(\xi-\xi^\nu))$ where $\eta_+\in C^\infty_c(\widehat \bbR^n)$ is such that  $\eta_+(\xi)=1$ for $|\xi|\le 1$ and $\eta_+(\xi)=0$ for $|\xi|>3/2$, so that  $\hat{g}_{\nu} = \hat{g}_{+,\nu} \cdot \hat{g}_{\nu}$. 
Let 
\begin{equation}\label{phidef}
    \phi(\xi) := \inn{\gamma \circ \theta(\xi)}{\xi}.
\end{equation}

\begin{lemma}\label{asy-lemma} If $\eps$, $\rho > 0$ are chosen sufficiently small, then for all $\la \geq 1$ and $\nu\in \fN_\eps(\la)$ the identity
\begin{equation*}
    \hat{\mu}_{\gamma}(\xi)= e^{-i\phi(\xi)} m(\xi)
\end{equation*} 
holds on $\supp \hat{g}_{+,\nu}$ where
\begin{enumerate}[i)]
    \item $|m(\xi)|\gc \la^{-1/n} $ for $\xi \in \supp\hat{g}_{+,\nu}$;
    \item The function $a_\nu := 
m^{-1} \cdot \hat{g}_{+,\nu}$ satisfies
\begin{equation*}
    |  \partial_\xi^\alpha a_\nu(\xi)|
\le C_\alpha \la^{(1-|\alpha|)/n} \qquad \textrm{for all $\alpha \in \N_0^n$}.
\end{equation*}
\end{enumerate}
\end{lemma}

The proof, which is based on the stationary phase method and, in particular, oscillatory integral estimates from \cite{BGGIST2007}, is postponed until \S\ref{proof of asymptotics} below.

\subsubsection*{Lower bounds for the operator norm}
For each $\nu \in \fN_\eps(\la)$ define $f_\nu$ by
 \[ \hat{f}_\nu(\xi) := \frac{\hat {g}_\nu(\xi)}{\hat{\mu}_{\gamma}(\xi)}\]
 and consider the randomised sums
 $$f^\om(x) := \sum_{\nu\in \fN_\eps(\la)}  r_\nu(\om) f_\nu(x) \qquad\textrm{for $\omega \in [0,1]$}.$$
 Note that, by Lemma~\ref{asy-lemma}, the $f^\om $ are  well-defined smooth function with compact support (and with  bounds depending on $\la$). Furthermore, if $g^{\om}$ is the function defined in \eqref{grandom}, then
 \Be\label{Af=g} g^\om= A_{\gamma}f^\om.\Ee
  We proceed to estimate the $L^p$ norm of $f^\om$,
  uniformly in $\om$.

  We have  $f_{\nu} = K_{\nu} \ast g_{\nu}$ where the kernel $K_{\nu}$ is given by
\begin{equation*}
    K_{\nu}(x) := \frac{1}{(2\pi)^n} \int_{\hat{\R}^n} e^{i\inn{x}{\xi}}  \hat{\mu}_{\gamma}(\xi)^{-1} \hat{g}_{+,\nu}(\xi)\,\ud \xi.
\end{equation*}
By Lemma~\ref{asy-lemma} 
one may write $\hat{\mu}_{\gamma}(\xi)^{-1}\hat{g}_{+,\nu}(\xi)= e^{i\phi(\xi)}a_\nu(\xi) $ where $a_{\nu}$ is supported where $|\xi-\xi^\nu|\le 2\rho\la^{1/n}$ and satisfies
  $\partial_\xi^\alpha a_\nu(\xi)= O(\la^{(1-|\alpha|)/n})$.
Setting 
\begin{align}
\nonumber
    \mathcal{E}_{\nu}(\xi) &:= \phi(\xi) - \phi(\xi^{\nu}) - \inn{\partial_{\xi}\phi(\xi^{\nu})}{\xi - \xi^{\nu}},
    \\
    \label{choice-x-nu}
     x^{\nu} &:= - \partial_{\xi}\phi(\xi^{\nu}),
\end{align} 
it follows that
\begin{equation*}
   \hat{\mu}_{\gamma}(\xi)^{-1} \hat{g}_{+,\nu}(\xi) = e^{i\phi(\xi^{\nu})} e^{-i \inn{x^{\nu}}{\xi - \xi^{\nu}}} e^{i\mathcal{E}_{\nu}(\xi)} a_{\nu}(\xi).
\end{equation*}
Applying a change of variable, 
\begin{equation*}
    K_{\nu}(x) = e^{i (\inn{x}{\xi^{\nu}} + \phi(\xi^{\nu}))} \frac{\la}{(2\pi)^n} \int_{\hat{\R}^n} e^{i\inn{\la^{1/n}(x -x^{\nu})}{\xi}}  e^{i\mathcal{E}_{\nu}(\xi^{\nu} +\la^{1/n}\xi)}a_{\nu}(\xi^{\nu} + \la^{1/n}\xi)\,\ud \xi.
\end{equation*}

By  the homogeneity of $\phi$, it follows that $|\partial_{\xi}^{\alpha}e^{i\mathcal{E}_{\nu}(\xi^{\nu} +\la^{1/n}\xi)}| \lesssim_{\alpha} 1$ for all multiindices $\alpha \in \N_0^n$. On the other hand,   
Lemma~\ref{asy-lemma} implies that $$\int_{\hat{\R}^n}|\partial_{\xi}^{\alpha}a_{\nu}(\xi^{\nu} +\la^{1/n} \xi)|\,\ud \xi \lesssim_{\alpha} \la^{1/n} \qquad \textrm{for all $\alpha \in \N_0^n$.}$$ 
Repeated integration-by-parts therefore yields 
\begin{equation*}
    |K_{\nu}(x)| \lesssim_N \la^{(n+1)/n}  (1 + \la^{1/n} |x - x^{\nu}|)^{-N} \qquad \textrm{for all $N \in \N_0^n$}. 
\end{equation*}
Consequently, the pointwise inequality
\[|f_\nu(x)|\lc \la^{(n+1)/n} (1+\la^{1/n}|x-x^\nu|)^{-N}
\lc \la^{(n+1)/n}\sum_{\ell\ge 0} 2^{-\ell N}\bbone_{B_{\ell,\la}^\nu}(x)
\] 
holds with $B^\nu_{\ell,\la} := \{x \in \R^n: |x-x^{\nu}|\le 2^\ell \la^{-1/n}\}$. Hence,
\Be \label{ell sum}\|f^{\om}\|_{L^p(\R^n)}
\lc \la^{(n+1)/n} \sum_{\ell\ge 0}2^{-\ell N}  \Big\|\sum_{\nu\in \fN_\eps(\la)} \bbone_{B_{\ell,\la} ^\nu}\Big\|_{L^p(\R^n)} \qquad \textrm{for all $\om\in [0,1]$.}
\Ee

To estimate the terms of \eqref{ell sum} for $2^\ell>\eps\la^{1/n}$ use the immediate bound  
$$\Big\|\sum_{\nu\in \fN_\eps(\la)} \bbone_{B_{\ell,\la}^\nu}\Big\|_{L^p(\R^n)} \lc \#\fN_\eps(\la) \cdot 2^{\ell n/p} \la^{-1/p}
\lc 2^{\ell n/p} \la^{1/n-1/p}.$$
For $2^\ell\le\eps \la^{1/n}$ this may be improved upon using a separation property of the $x^{\nu}$: namely, 
\Be \label{separation} |x^\nu-x^{\nu'}| \gc \la^{-1/n}|\nu-\nu'|,
\Ee
provided the parameter $\varepsilon > 0$ is chosen sufficiently small (independently of $\lambda$). The property \eqref{separation} implies that the balls $B_{\ell,\la}^\nu$,
$B_{\ell,\la }^{\nu'}$ are disjoint for 
$|\nu-\nu'| \gtrsim 2^\ell$.
Assuming \eqref{separation} for a moment and taking $2^\ell\le \eps\la^{1/n}$, we obtain 
\begin{align*} \Big\|\sum_{\nu\in \fN_\eps(\la)} \bbone_{B_{\ell,\la}^\nu}\Big\|_{L^p(\R^n)}
&\le
\sum_{i=0}^{2^\ell-1} 
\Big\|\sum_{\substack{m \in \Z \\ |m| \leq \eps 2^{-\ell}\la^{1/n}}}  \bbone_{B_{\ell,\la} ^{2^\ell m+i}}\Big\|_{L^p(\R^n)}
\\
&\lc
\sum_{i=0}^{2^\ell-1} \Big(\sum_{\substack{m \in \Z \\ |m| \leq \eps 2^{-\ell}\la^{1/n}}} \big\| \bbone_{B_{\ell,\la}^{2^\ell m+i}} \big\|_{L^p(\R^n)}^p\Big)^{1/p} \\
&\lc 2^\ell (\la^{1/n}2^{-\ell} )^{1/p} (2^\ell\la^{-1/n})^{n/p}.
\end{align*}

Applying the preceding bounds to estimate the terms in \eqref{ell sum} and choosing $N>n-n/p$, this leads to the uniform estimate 
\Be \label {ft-est}
\sup_{\om\in [0,1]}\|f^\om  \|_{L^p(\R^n)} \lc \la^{(n+1)/n-(n-1)/np}.
\Ee
Thus, one concludes from \eqref{Af=g}, \eqref{glowerbound} and \eqref{ft-est}, together with the fact that the $f^\om$ are Fourier supported where $|\xi|\sim \lambda$, that
\begin{align*}
\sup_{f\in L^p\cap\cZ_\la} \frac{\|A_{\gamma}f\|_{L^p(\R^n)}}{\|f\|_{L^p(\R^n)}} \ge \frac{\big(\int_0^1\|A_{\gamma}f^\om \|_{L^p(\R^n)}^p \,\ud \om\big)^{1/p} }{\sup_{\om\in [0,1]} \|f^\om\|_{L^p(\R^n)}}=
\frac{(\int_0^1\|g^\om \|_{L^p(\R^n)} ^p\,\ud \om\big)^{1/p}}{\sup_{\om\in [0,1]} \|f^\om\|_{L^p(\R^n)}}
\gc 
\la^{  -\frac 1n(\frac 12+\frac 1p)},
\end{align*}
which is the desired bound stated in Proposition~\ref{Fourier nec prop}.

It remains to verify the crucial separation property \eqref{separation}. Recall from \eqref{choice-x-nu} and \eqref{choice-xi-nu} that $x^\nu = -\partial_{\xi} \phi\big(\lambda \Gamma(\nu\la^{-1/n})\big)$. Thus, by homogeneity, one wishes to bound
\begin{equation}\label{separation 1}
    x^{\nu} - x^{\nu'} = - \big[(\partial_{\xi} \phi)\circ\Gamma(\nu\la^{-1/n})- (\partial_{\xi} \phi)\circ\Gamma(\nu'\la^{-1/n}) \big].
\end{equation}
In particular, it suffices to show that
\begin{equation}\label{separation 2}
    \frac{d}{d\tau} (\partial_{\xi}  \phi)\circ\Gamma(\tau)=-\vec{e}_1+ O(\tau).
\end{equation}
Indeed, applying Taylor's theorem to \eqref{separation 1} and using \eqref{separation 2} to bound the linear term yields
\[x^{\nu}-x^{\nu'} = \frac{\nu-\nu'}{\la^{1/n}} \cdot \vec{e}_1  +
O \Big(\tfrac{(|\nu|+|\nu'|)|\nu-\nu'|}{\la^{2/n}} \Big)
=\frac{\nu-\nu'}{\la^{1/n}} \cdot \vec{e}_1  +
O \Big(\eps \tfrac{|\nu-\nu'|}{\la^{1/n}} \Big)
\]
for all $\nu, \nu' \in  \fN_\eps(\la)$. Choosing $\varepsilon > 0$ sufficiently small so as to control the error term establishes \eqref{separation}.

Turning to the proof of \eqref{separation 2}, we have 
$\partial_{\xi}\phi(\xi) = \gamma\circ \theta(\xi)+\inn{\gamma'\circ\theta(\xi)}{\xi} 
\partial_{\xi} \theta(\xi)$ by the definition of $\phi$ from \eqref{phidef}. Since $\inn{\gamma'\circ\theta(\xi)}{\xi} =0$ when $\xi=\Gamma(\tau)$, this yields 
$$(\partial_{\xi}\phi)\circ\Ga(\tau)
= \gamma\circ\theta(\Gamma(\tau))$$
and, consequently, 
\[\frac{d}{d\tau} (\partial_{\xi}\phi)\circ\Ga(\tau) = \gamma'\circ\theta(\Gamma(\tau)) \cdot \inn {\partial_{\xi} \theta(\Gamma(\tau))} {\Gamma'(\tau)}.\]
By the initial reductions, $\ga^{(j)}(0) = \vec{e}_j$ for $1 \leq j \leq n$, and so
\begin{equation}\label{separation 3}
  \frac{d}{d\tau} (\partial_{\xi}\phi)\circ\Ga(\tau) = \inn {\partial_{\xi} \theta(\Gamma(0))} {\Gamma'(0)} \cdot \vec{e}_{1} + O(\tau).
\end{equation}
Thus, to prove \eqref{separation 2} it suffices to show that the inner product in the above display is equal to $-1$. Differentiating the defining equation $\inn{\ga^{(n-1)}\circ \theta(\xi)}{\xi} =0$, one deduces that
\begin{equation*}
    \partial_{\xi} \theta(\xi)= - \frac{1} { \inn{\ga^{(n)}\circ\theta(\xi)}{\xi}} \ga^{(n-1)}\circ\theta(\xi).
\end{equation*}
Since, by uniqueness in \eqref{solving on the bad cone} and \eqref{klast-two-comp} together with the initial reductions, $\Gamma(0)=\vec{e}_n$ and $\theta (\vec{e}_n)=0$, it follows that $(\partial_{\xi} \theta)\circ \Gamma(0) = - \vec{e}_{n-1}$. On the other hand, from \eqref{klast-two-comp} it is clear that $\inn{\vec{e}_{n-1}}{\Gamma'(0)} = 1$. Applying these observations to the formula in \eqref{separation 3} concludes the proof.

\subsection{Proof of Lemma~\ref{asy-lemma}}
 \label{proof of asymptotics} It remains to prove Lemma~\ref{asy-lemma}. To this end, we recall an asymptotic expansion from \cite{BGGIST2007}, based on the following formula: 
 \begin{equation}
\label{Jla}\int_{-\infty}^{\infty}e^{i\lambda s^{n}} \ud s= \alpha_{n}
\lambda^{-1/n} \qquad \textrm{for $n=2,3,\dots$ and $\la>0$,}
\end{equation}
where $\alpha_n$ is given by
\begin{equation}
\label{ak}\alpha_{n}:=%
\begin{cases}
\tfrac2n\Gamma(\tfrac1n)\sin(\tfrac{(n-1)\pi}{2n}) \, & \textrm{if $n$ is odd},\\
\tfrac2n\Gamma(\tfrac1n) \exp(i\tfrac{\pi}{2n}) \, & \textrm{if $n$ is even}.
\end{cases}
\end{equation}
The derivation of \eqref{Jla} relies on contour integration arguments, whilst the formula itself yields   asymptotic expansions for integrals $\int_{\R} e^{i\lambda s^{n}} \chi(s) \,\ud s$ with $\chi\in
C^{\infty}_{0}$: see, for instance, \cite[VIII.1.3]{Stein1993} or \cite[\S7.7]{HormanderIII}. Similar asymptotic expansions remain valid under slight perturbation of the phase function $s\mapsto s^n$, as demonstrated by the following lemma proved in \cite{BGGIST2007}.
 We use
 the notation $\|g\|_{C^{m}(I)}:=\max\limits_{0\le j\le m}\sup_{x\in{I}} |g^{(j)}(x)|$.
 
\begin{lemma}[\cite{BGGIST2007}, Lemma 5.1]
\label{asymptotics-lemma} Let $0<r\le 1$, $I=[-r,r]$, $I^*=[-2r,2r]$ 
 and let  $g\in C^{2}(I^*)$.
Suppose that
\begin{equation} \notag
r\le \frac{1}{10(1+\|g\|_{C^2(I^*)})}
\end{equation}
and
let $\eta\in C^{1}_c(\R)$ be supported in $I$ and satisfy the bounds
\begin{equation}
\notag
\|\eta\|_{\infty}+\|\eta^{\prime}\|_{1} \le A_{0}, 
\text{ and } \|\eta^{\prime}\|_{\infty}\le A_{1}.
\end{equation}
Let $n\ge 2$, define
\begin{equation*}
I_{\lambda}(\eta,w) := \int_{\R} \eta(s) \exp\big(i\lambda(\sum_{j=1}^{n-2} w_{j}
s^{j}+ s^{n} + g(s) s^{n+1})\big) \ud s
\end{equation*}
and let $\alpha_{n}$ be as in \eqref{ak}. Suppose $|w_{j}|\le\delta\lambda
^{(j-n)/n}$, $j=1,\dots, n-2$.
Then there is an absolute constant $C$ such that, for $\lambda>2$,
\[
|I_{\lambda}(\eta,w)-\eta(0) \alpha_{n} \lambda^{-1/n}| \le C[A_{0}
\delta\lambda^{-1/n}+A_{1}\lambda^{-2/n}(1+\beta_{n} \log\lambda)];
\]
here $\beta_{2} := 1$ and $\beta_{n} := 0$ for $n>2$.
\end{lemma}

Lemma~\ref{asy-lemma} is obtained via a fairly direct application of the above result. 

\begin{proof}[Proof (of Lemma~\ref{asy-lemma})] Taylor expand the phase $\inn{\gamma(s)}{\xi}$ with $s=\theta(\xi)+h$ to obtain 
 \begin{equation*}
    \inn{\gamma(\theta(\xi)+h)}{\xi} = \phi(\xi) + \sum_{j=1}^{n} u_j(\xi) \frac{h^j}{j!}+ u_{n+1}(\xi,h) \frac{h^{n+1}}{(n+1)!}
    \end{equation*}
    where
    \begin{align*}
    u_j(\xi) &:=  \inn{\gamma^{(j)} \circ \theta(\xi)}{\xi}, \text{  for $1 \leq j \leq n$,}\\
    u_{n+1}(\xi,h) &:=    
    \int_0^1 (n+1)(1-t)^n \inn{\gamma^{(n+1)}(\theta(\xi)+th)}{\xi} \,\ud t .
\end{align*}

Recall that  $u_{n-1}(\xi) \equiv 0$ by the definition of $\theta(\xi)$, whilst $u_n(\xi)\sim |\xi|\sim \la$ for $|\xi'|\leq c\xi_n$. Thus, writing
\[\hat{\mu}_{\gamma}(\xi) =e^{-i\phi(\xi)}m(\xi)\]
as in the statement of the lemma, it follows that the function $m$ is given by
\begin{align*}
m(\xi) := \int_{\R} e^{-i 
(\sum_{j=1}^{n} u_j(\xi) \frac{h^j}{j!}  +u_{n+1} ( \xi,h)\frac{ h^{n+1} }{(n+1)!} )} \chi(\theta(\xi)+h) \ud h.
\end{align*}
Thus, defining
\begin{equation}\label{Psi phase}
\Psi(\xi,h) := \sum_{j=1}^{n-2} w_j(\xi) h^j  + h^n + g(\xi,h) h^{n+1}
\end{equation}
where
\begin{equation*}
    w_j(\xi) := \frac{1}{j!} \cdot \frac{u_j(\xi)}{u_n(\xi)} \qquad \textrm{and} \qquad g(\xi,h) := \frac{1}{(n+1)!} \cdot \frac{u_{n+1}(\xi,h)}{u_n(\xi)},
\end{equation*}
one may succinctly express $m$ as
\begin{equation*}
    m(\xi) = \int_{\R} e^{-i u_n(\xi) \Psi(\xi,h)} \chi(\theta(\xi)+h) \ud h.
\end{equation*}

We now turn to proving the bounds on $m$ stated in Lemma~\ref{asy-lemma}.\medskip

\noindent \textit{i)} The desired pointwise lower bound on $m$ follows from a direct application of Lemma~\ref{asymptotics-lemma}. In particular, by the definition of the $\xi^\nu$ we have $w_j(\xi^\nu)=0$ for $1 \leq j \leq n-1$ and therefore 
\begin{equation}\label{wj bounds}
    |w_j(\xi)|\lesssim \rho \la^{(1-n)/n}  \qquad \textrm{and} \qquad |g(\xi,h)|\lesssim 1 \qquad \textrm{for $\xi \in \supp \hat{g}_{\nu, +}$.}
\end{equation}
Thus, provided $\rho$ and $\varepsilon$ are chosen small enough, Lemma~\ref{asymptotics-lemma} can be applied to show that \Be\label{m lower bound}
|m(\xi)|\gtrsim \la^{-1/n} \qquad \textrm{for all $\xi \in \supp\hat{g}_{+,\nu}$,}
\Ee
as desired.\medskip

\noindent \textit{ii)} It remains to show that 
\Be\label{m derivative 1}
|\partial_{\xi}^\alpha  \big[m^{-1} \cdot \hat{g}_{+,\nu} \big](\xi)\big|\lesssim_{\alpha} \la^{(1-|\alpha|)/n} \qquad \textrm{for all $\alpha\in \bbN_0^n$.}
\Ee
The derivative $\partial_\xi^\alpha m(\xi)$ can be expressed as a sum of functions of the form 
\begin{equation*}
    m^{\alpha}_{\kappa, d}(\xi) := \int_\bbR  e^{-i 
(\sum_{j=1}^{n} u_j(\xi) \frac{h^j}{j!}  +u_{n+1} ( \xi,h)\frac{ h^{n+1} }{(n+1)!} )} h^\ka b^{\alpha}_d(\xi,h)\, \ud h, \quad d+\kappa\ge |\alpha|,
\end{equation*}
where $b^{\alpha}_d \in C^{\infty}(\hat{\bbR}^n\setminus\{0\} \times \R)$ and homogeneous of degree $-d$ in the $\xi$-variable. The key claim is
\begin{equation}\label{m derivative 2} 
    |m^{\alpha}_{\kappa, d}(\xi)| \lesssim_{\alpha} \la^{-(1+\ka)/n- d} \qquad \textrm{for 
$\xi \in \supp\hat{g}_{+,\nu}$.}
\end{equation}
Indeed, once \eqref{m derivative 2} is established it can be combined with \eqref{m lower bound} and the Leibniz rule the deduce the desired bound \eqref{m derivative 1}. 

The asserted bound \eqref{m derivative 2} follows from
\Be\label{ska int}
\Big| 
\int_{\R} e^{-i u_n(\xi)
\Psi(\xi,h)} \chi_1(\xi,h) h^{\kappa} \ud h\Big| \lc \la^{-(1+\ka)/n},
\Ee
where $\chi_1 \in C^{\infty}(\hat \bbR^n\setminus \{0\}\times \bbR)$ is homogeneous of degree zero with respect to $\xi$ and vanishes unless $|h|\lesssim \eps$. To prove \eqref{ska int} we form a dyadic decomposition of the integral. Fix $\zeta_0\in C^\infty_0(\bbR)$ such that $\zeta_0(h)=1$ for $|h|<1/2$ with $\supp \zeta \subseteq [-1,1]$. For $\ell \in \N$ set $\zeta_\ell(h)= \zeta_0(2^{-\ell}h)-\zeta_0(2^{-\ell-1}h)$ and define
\begin{equation}\label{ska loc int}
J_{\ell,\la}(\xi) :=
\int_{\R} e^{-i u_n(\xi)\Psi(\xi,h)}
\zeta_\ell(\la^{1/n}h)\chi_1(\xi,h)) h^\kappa\, \ud h.
\end{equation}
By just a size estimate we have $|J_{\ell,\la}(\xi)|\lesssim (2^\ell \la^{-1/n} )^{1+\kappa}$, which we use for $\ell\le C$.
For larger $\ell$ we use integration-by-parts.

Recall that the assumption $\xi\in \supp \hat{g}_{+,\nu}$ implies the bounds \eqref{wj bounds}. Consequently, on the support of the integrand in \eqref{ska loc int}, the dominant term in the formula for $\Psi$ as given in \eqref{Psi phase} is $h^n$. Moreover,  
$$\Big|\frac{\partial}{\partial h} \Psi(\xi, h) \Big|\sim (2^\ell \la^{-1/n})^{n-1}$$
and, similarly, 
$$\Big|\frac{\partial^i}{\partial h^i} \Psi(\xi, h) \Big|\lc \min\{(2^\ell\la^{-1/n})^{n-i},1\}.$$ 
Also, it is not difficult to show that
$$\Big|\frac{\partial^i}{\partial h^i}\big[ \zeta_\ell(\la^{1/n}h)\chi_1(\xi,h)) h^\kappa \big]\Big| \lc (2^{\ell}\la^{-1/n} )^{\ka-i}.$$
Using these bounds we derive, by $N$-fold integration-by-parts, 
\begin{equation*}
    |J_{\ell,\la}(\xi) |\lc_N (2^\ell\la^{-1/n} )^{1+\ka} 2^{-\ell n N}
\end{equation*}
and, by summing in $\ell$, obtain \eqref{ska int}.
\end{proof}

\subsection{The Christ example}\label{Christ ex sec} We close this section by making an observation regarding an endpoint case. We may rule out $L^d(\R^n) \to L^d_{1/d}(\R^n)$ boundedness under the maximal type $d$ hypothesis for $d \geq 3$. Note that this corresponds to the critical vertices in the lower diagram in Figure~\ref{finite type conjectured range}. To show the failure of the estimate, suppose that for some $t_0$ with $\chi(t_0)\neq 0$, there is a unit vector $u$ with $\inn{u}{\gamma^{(k)}(t_0) }=0$ for $k=1,\dots, d-1$  and $\inn{u}{ \gamma^{(d)} (t_0)} \neq 0$. By a rotation we can assume that $\gamma'(t_0)=\vec{e}_1$ and $u=\vec{e}_2$, the standard coordinate vectors. The $L^d(\R^n) \to L^d_{1/d}(\R^n)$ boundedness is equivalent with the statement that the multiplier 
\[|\xi|^{1/d} \upsilon (\xi)  \int_{\R}  e^{i \inn{\gamma(t)}{\xi} }  \chi (t)\,\ud t \]
belongs to the multiplier class $M^d(\R^n)$; here $\upsilon \in C^\infty$, equal to $1$ for large $\xi$ and vanishing in a neighborhood of the origin. Since $(\xi_1^2+\xi_2^2)^{1/2d} |\xi|^{-1/d}$ belongs to  $M^p(\R^n)$ for $1<p<\infty$ we may replace in the display $|\xi|^{1/d}$ with $(\xi_1^2+\xi_2^2)^{1/2d} $. Now apply the theorem by de Leeuw \cite{deLeeuw1965} on the restriction of multipliers to subspaces  
to see that 
\[ (\xi_1^2+\xi_2^2)^{1/2d}  \upsilon (\xi_1,\xi_2,0,\dots,0)  \int \chi (t) e^{i (\gamma_1(t)\xi_1+\gamma_2(t) \xi_2) } \,\ud t \]
is a multiplier in $M^d(\bbR^2)$ which implies the $L^d(\bbR^2) \to L^d_{1/d}(\bbR^2) $ boundedness of the averaging operator associated to the plane curve $(\gamma_1(t), \gamma_2(t) )$. However the latter statement can be disproved   by using the argument of  Christ \cite{Christ1995}, who considered the curve $(t,t^d)$.




\section{Initial reductions and auxiliary results}\label{red aux sec}

The remainder of the paper deals with the proof of Theorem~\ref{non-degenerate theorem}. This section contains some preliminary results, the most significant of which is the decoupling result in Theorem~\ref{Frenet decoupling theorem} which lies at the heart of the proof.




\subsection{Multiplier notation} From the reduction described in Proposition~\ref{quant nondeg prop} it suffices to consider $\gamma \in \mathfrak{G}_4(\delta_0)$ where $\delta_0$ is a small parameter, as described at the end of \S\ref{moment red sec}. If $f$ belongs to a suitable \textit{a priori} class, then the Fourier transform of $A_{\gamma}f$ is the product of $\hat{f}$ and the multiplier
\begin{equation}\label{0419e_multiplier}
    \hat{\mu}_\gamma(\xi)=\int_{\R} e^{-i \inn{\gamma(s)}{\xi}}\, \chi(s)\, \ud s.
\end{equation}
Recall, again from the reduction described in Proposition~\ref{quant nondeg prop}, that we may assume $\chi \in C^{\infty}_c(\R)$ satisfies $\supp \chi \subseteq I_0 = [-\delta_0, \delta_0]$. 

Given $m \in L^{\infty}(\hat{\R}^4)$, define the associated multiplier operator $m(D)$ by
\begin{equation*}
    m(D)f (x) :=  \frac{1}{(2 \pi)^4} \int_{\hat{\R}^4} e^{i \inn{x}{\xi}} m(\xi) \hat{f}(\xi)\,\ud \xi
\end{equation*}
 so that, in this notation, $A_{\gamma} = \hat{\mu}_{\gamma}(D)$.
We also define the associated $L^p$ multiplier norms
\begin{equation*}
    \|m\|_{M^p(\R^4)} := \|m(D)\|_{L^p(\R^4) \to L^p(\R^4)} \qquad \textrm{for $1 \leq p \leq \infty$.}
\end{equation*}

To prove Theorem~\ref{non-degenerate theorem}, we analyse various multipliers obtained by decomposing \eqref{0419e_multiplier}. To this end, given $a \in C^{\infty}(\hat{\R}^4\setminus\{0\}\times \R )$, define
\begin{equation}\label{e0401_2.1}
    m[a](\xi) := \int_{\R} e^{-i \inn{\gamma(s)}{\xi}} a(\xi; s)\chi(s)\,\ud s.
\end{equation}
Any decomposition of the symbol $a$ results in a corresponding decomposition of the multiplier. We will also use the notation $\supp_\xi a$ to denote the projection of $\supp a \subseteq \hat{\R}^4 \backslash \{0\} \times \R$ into $\hat{\R}^4 \backslash \{0\}$.




\subsection{Reduction to band-limited functions}\label{sec:bandlimited}
Given a symbol $a\in C^{\infty}(\hat{\R}^4\setminus\{0\}\times \R)$ we perform a dyadic decomposition in the frequency variable $\xi$ as follows. Fix $\eta \in C^\infty_c(\R)$ non-negative and such that 
\begin{equation*}
\eta(r) = 1 \quad \textrm{if $r \in [-1,1]$} \quad \textrm{and} \quad \supp \eta \subseteq [-2,2]   
\end{equation*}
and define $\beta^k \in C^\infty_c(\R)$ by    
\begin{equation}\label{beta def}
    \beta^k(r):=\eta(2^{-k}r) - \eta(2^{-k+1}r)
\end{equation} 
for each $k \in \Z$. By a slight abuse of notation we also let $\eta$, $\beta^k \in C^{\infty}_c(\hat{\R}^4)$ denote the functions $\eta(\xi) := \eta(|\xi|)$ and $\beta^k(\xi):=\beta^k(|\xi|)$. One may then decompose
\begin{equation}\label{symbol dec}
    a = \sum_{k = 0}^{\infty} a_k \qquad \textrm{where} \qquad  a_k(\xi; s) :=   \left\{ \begin{array}{ll}
        a(\xi; s) \cdot \beta^k(\xi) & \textrm{for $k \geq 1$} \\
         a(\xi; s) \cdot \eta(\xi) & \textrm{for $k =0$}
     \end{array} \right. .
\end{equation}

Theorem~\ref{non-degenerate theorem} is a direct consequence of the following result for multipliers localised to some dyadic frequency band. Here we work with additional absolute constants $0 < \delta_j \leq \delta_0$ for $1 \leq j \leq 3$, chosen sufficiently small for the purposes of the forthcoming arguments. In practice, we may simply take $\delta_j := \delta_0$ for $j = 1, 3$ and $\delta_2 := \delta_0^3$. It is also convenient to define $\delta_4 := 9/10$.

\begin{theorem}\label{Sobolev theorem} Let $\gamma \in \mathfrak{G}_4(\delta_0)$ and $1 \leq J \leq 4$. 
Suppose that $a \in C^{\infty}(\hat{\R}^4\setminus \{0\} \times \R)$ satisfies
\begin{equation}\label{LS symbol}
    |\partial_{\xi}^{\alpha}\partial_s^N a(\xi;s)| \lesssim_{\alpha, N} |\xi|^{-|\alpha|} \qquad \textrm{for all $\alpha \in \N_0^4$ and $N \in \N_0$}
\end{equation}
and
\begin{equation}\label{J derivatives}
   \left\{
   \begin{array}{ll}
   \displaystyle\inf_{s \in I_0}|\inn{\gamma^{(J)}(s)}{\xi}| \geq \delta_J |\xi| & \\[5pt]
\displaystyle \inf_{s \in I_0}|\inn{\gamma^{(j)}(s)}{\xi}| \leq 4\delta_j |\xi| & \textrm{for $1 \leq j \leq J-1$} 
\end{array}
\right. \qquad \textrm{for  all $\xi \in \xisupp a$.}
\end{equation}
 If $a_k$ is defined as in \eqref{symbol dec}, then
\begin{equation}\label{Sobolev ineq}
    \|m[a_k]\|_{M^p(\R^4)} \lesssim_p 2^{-k/p}
\end{equation}
for $k\ge 1$ and $p>\max\{2(J-1), 1\}$. 
\end{theorem}

The hypothesis \eqref{J derivatives} implies that
\begin{equation}\label{e200727e3.4}
|\inn{\gamma^{(j)}(s)}{\xi}| \leq 8\delta_{0} |\xi| \qquad \textrm{for  all $\xi \in \xisupp a$, $s\in I_0$ and $1 \leq j \leq J-1$.}
\end{equation}
Indeed, suppose $s_0 \in I_0$ realises the infimum in \eqref{J derivatives} and let $s \in I_0$. Then the mean value theorem implies
\begin{equation}\label{J derivatives 1}
  |\inn{\gamma^{(j)}(s)}{\xi}| \leq |\inn{\gamma^{(j)}(s_0)}{\xi}| +  \sup_{t \in I_0}|\gamma^{(j+1)}(t)||s-s_0| |\xi| \leq 8\delta_{0}|\xi|,
\end{equation}
using the fact that $\delta_j \leq \delta_0$ and the uniform derivative bounds for $\gamma \in \mathfrak{G}_4(\delta_0)$. 

\begin{proof}[Proof of Theorem~\ref{non-degenerate theorem} given Theorem~\ref{Sobolev theorem}] By the reduction from \S\ref{moment red sec} it suffices to consider $\gamma \in \mathfrak{G}_4(\delta_0)$ and $\chi \in C_c^{\infty}(\R)$ with $\supp \chi \subseteq I_0 = [-\delta_0, \delta_0]$. For $1 \leq j \leq 4$ define the sets
\begin{equation*}
    \mathscr{E}_j := \Big\{ \xi \in S^3 : \inf_{s \in I_0} |\inn{\gamma^{(j)}(s)}{\xi}| < \delta_j \Big\} \quad \textrm{and} \quad U_j := 
    \begin{cases}
    N_{\delta_j} \mathscr{E}_j \cap S^3 & \textrm{if $1 \leq j \leq 3$} \\
    N_{\delta_0} \mathscr{E}_j \cap S^3 & \textrm{if $j = 4$} 
    \end{cases},
\end{equation*}
where $N_{\delta_j} \mathscr{E}_j$ denotes the $\delta_j$-neighbourhood of $\mathscr{E}_j$ and $S^3$ denotes the unit sphere in $\hat{\R}^4$. Since $U_j$ is an open subset of $S^3$ containing the compact subset $\mathrm{clos}\,\mathscr{E}_j$, there exists a smooth function $\rho_j \colon S^3 \to [0,\infty)$ such that
\begin{equation*}
    \rho_j(\omega) = 1 \quad \textrm{for $\omega \in \mathrm{clos}\,\mathscr{E}_j$} \qquad \textrm{and} \qquad \supp \rho_j \subseteq U_j. 
\end{equation*}

For $1 \leq J \leq 4$ define $\chi_J \in C^{\infty}(S^3)$ by
\begin{equation*}
    \chi_J := \Big(\prod_{j = 1}^{J-1} \rho_j \Big) \cdot ( 1 - \rho_J)
\end{equation*}
These functions satisfy the following properties:
\begin{enumerate}[i)]
    \item If $\xi \in S^3$ and $\xi \in \supp \chi_J$, then \eqref{J derivatives} holds;
    \item $\sum_{J=1}^4 \chi_J \equiv 1$, as functions on $S^3$.
\end{enumerate}
Indeed, to see property i), note that if $\xi \in \supp \chi_J$, then $\xi \notin \mathrm{clos}\,\mathscr{E}_J$ which implies the first bound in \eqref{J derivatives}.  On the other hand, for $1 \leq j \leq J-1 \leq 3$ it follows that $\xi \in U_j$ and so there exists some $\xi_0 \in \mathscr{E}_j$ with $|\xi - \xi_0| < \delta_j$. Consequently, there exists some $s_0 \in I_0$ such that
\begin{equation}\label{J derivatives 2}
    |\inn{\gamma^{(j)}(s_0)}{\xi}| \leq |\inn{\gamma^{(j)}(s_0)}{\xi_0}| + |\gamma^{(j)}(s_0)||\xi - \xi_0| \leq 4\delta_j,
\end{equation}
which is the second bound in \eqref{J derivatives}. For property ii), note that \eqref{J derivatives 1} can be combined with the argument in \eqref{J derivatives 2} to conclude that
\begin{align*}
    \sup_{s \in I_0}|\inn{\gamma^{(j)}(s)}{\xi}| \leq 8\delta_0 \quad \textrm{for $\xi \in U_j$,  $1 \leq j \leq 3$,} \quad \textrm{and} \quad
    \sup_{s \in I_0}|\inn{\gamma^{(4)}(s)}{\xi}| \leq \tfrac{9}{10} + 4\delta_0 \quad \textrm{for $\xi \in U_4$.}
\end{align*}
Provided $\delta_0$ is sufficiently small, the non-degeneracy of $\gamma \in \mathfrak{G}_4(\delta_0)$ implies $\bigcap_{j=1}^4 U_j = \emptyset$. Since $\sum_{J=1}^4 \chi_J = 1 - \prod_{j=1}^4 \rho_j$, property ii) follows from the support conditions of the $\rho_j$.

In view of the above, we may apply Theorem~\ref{Sobolev theorem} with $a(\xi) := \chi_J(\xi/|\xi|)$ for $J = 1, 2, 3, 4$ and sum in $J$ to conclude that
\begin{equation}\label{freq loc est}
    \|\beta^k(D)A_{\gamma} f\|_{L^p(\R^4)} \lc 2^{-k/p} \|f\|_{L^p(\R^4)} \qquad \text{for all $k \geq 0$}
\end{equation}
and all $p > 6$; note that the case $k=0$ trivially follows as $A_\gamma$ is an averaging operator. To pass from the frequency localised estimates \eqref{freq loc est} to genuine $L^p$-Sobolev bounds, one may apply a Calder\'on--Zygmund estimate from \cite{PRS2011}. This argument is described in the Appendix~\ref{Sob-from-dy appendix}: see Proposition~\ref{prop:Sobolev-from-dyadic}. 
\end{proof}

The hypothesis \eqref{LS symbol} implies that $\|\mathcal{F}_{\xi}^{-1}a_k(\,\cdot\,;s)\|_{L^1(\R^4)} \lesssim 1$, where $\mathcal{F}_{\xi}^{-1}$ denotes the inverse Fourier transform in the $\xi$ variable. Consequently, it is not difficult to show that the $p = \infty$ case of \eqref{Sobolev ineq} holds for all $1 \leq J \leq 4$ (see also Lemma~\ref{gen Linfty lem}). The problem is therefore to deduce the estimate for $p$ near to $\max\{ 2(J-1), 1\}$.\medskip

Note that the proof of the $J=1$ case of Theorem~\ref{Sobolev theorem} is trivial. Indeed, here the phase function of \eqref{e0401_2.1} does not admit a critical point and the desired result follows by repeated integration-by-parts.\medskip

The proof of the $J=2$ case of Theorem~\ref{Sobolev theorem} is also straightforward. Suppose $\gamma \in \mathfrak{G}_4(\delta_0)$ and $a \in C^{\infty}(\R^4\setminus \{0\} \times \R)$ satisfies the hypotheses Theorem~\ref{Sobolev theorem} for $J = 2$, with $\delta_1 := \delta_0$ and $\delta_2 := \delta_0^3$.\footnote{The choice of $\delta_1$, $\delta_2$ is not important for the argument in the $J=2$ case, but is kept for consistency.}
Note, in particular, that
\begin{equation*}
    |\inn{\gamma''(s)}{\xi}| \ge \delta_0^3 |\xi|  \qquad \textrm{for all $(\xi; s)\in \xisupp a \times I_0$.} 
\end{equation*}
Thus, the van der Corput lemma (see, for instance, \cite[Chapter VIII, Proposition 2]{Stein1993}) implies
\begin{equation*}
    \|m[a_k]\|_{M^2(\R^4)} = \|m[a_k]\|_{L^{\infty}(\R^4)} \lesssim 2^{-k/2}.
\end{equation*}
On the other hand, by the triangle inequality, Fubini's theorem, translation-invariance and integration-by-parts (see Lemma~\ref{gen Linfty lem}),
\begin{equation*}
    \|m[a_k]\|_{M^{\infty}(\R^4)} \leq \|\mathcal{F}^{-1}a_k\|_{L^1(\R^4)} \lesssim 1.
\end{equation*}
Interpolation yields
\begin{equation*}
    \|m[a_k]\|_{M^p(\R^4)} \lesssim 2^{-k/p} \qquad \textrm{for all $2 \leq p \leq \infty$,}
\end{equation*} 
which concludes the proof for the $J=2$ case.
\medskip

From now on, we focus on the $J=3$ and $J=4$ cases of Theorem~\ref{Sobolev theorem}. These are proved in Sections \ref{sec:J=3} and \ref{sec:J=4} respectively. Of these, the $J=4$ is the heart of the matter, and its proof is the main contribution of this paper. Before turning to the proofs, we state some auxiliary results.




\subsection{The Frenet frame.}\label{Frenet subsection} At this juncture it is convenient to recall some elementary concepts from differential geometry which feature in our proof. 
Given a smooth non-denegenate curve $\gamma:I \to \R^n$, the Frenet frame is the orthonormal basis resulting from applying the Gram--Schmidt process to  the vectors
\begin{equation*}
    \{ \gamma'(s), \dots, \gamma^{(n)}(s)\},
\end{equation*}
which are linearly independent in view of the condition \eqref{eq:nondegenerate}. Defining the functions\footnote{Note that the $\tilde{\kappa}_j$ depend on the choice of parametrisation and only agree with the (geometric) curvature functions 
\begin{equation*}
    \kappa_j(s) := \frac{\langle \be_j'(s), \be_{j+1}(s) \rangle}{|\gamma'(s)|}
\end{equation*}
if $\gamma$ is unit speed parametrised. Here we do not assume unit speed parametrisation.} 
\begin{equation*}
    \tilde{\kappa}_j(s) := \langle \be_j'(s), \be_{j+1}(s) \rangle  \qquad \text{for } j=1, \dots, n-1,
\end{equation*}
one has the classical Frenet formul\ae
\begin{align*}\be_1'(s)&=  \tilde{\kappa}_1(s) \be_2(s),
    \\ \be_i'(s)&= -\tilde{\kappa}_{i-1}(s)\be_{i-1}(s) + \tilde{\kappa}_{i}(s)\be_{i+1}(s),\,\,i=2,\dots, n-1,
    \\
    \be_n'(s)&=-\tilde{\kappa}_{n-1}(s)\be_{n-1}(s).
\end{align*}
Repeated application of these formul\ae\ shows that 
\begin{equation*}
    \be^{(k)}_{i}(s) \perp \be_{j}(s) \qquad \textrm{whenever} \qquad 0 \leq  k < |i - j|.
\end{equation*}
Consequently, by Taylor's theorem 
\begin{equation}\label{Frenet bound}
    |\inn{\be_{i}(s_1)}{\be_{j}(s_2)}| \lesssim_{\gamma} |s_1 - s_2|^{|i-j|} \qquad \textrm{for $1 \leq i,j \leq n$ and $s_1, s_2 \in I$.}
\end{equation}
Furthermore, one may deduce from the definition of $\{\be_j(s)\}_{j=1}^n$ that
\begin{equation}\label{Frenet bound alt 1}
    |\inn{\gamma^{(i)}(s_1)}{\be_{j}(s_2)}| \lesssim_{\gamma} |s_1 - s_2|^{(j-i)\vee 0} \qquad \textrm{for $1 \leq i,j \leq n$ and $s_1, s_2 \in I$.}
\end{equation}

In this paper, much of the microlocal geometry of the operator $A_{\gamma}$ is expressed in terms of the Frenet frame.




\subsection{A decoupling inequality for regions defined by the Frenet frame} Let $\gamma \colon I \to \R^n$ be a non-degenerate curve.

\begin{definition}\label{def Frenet box}  Given $2 \leq d \leq n-1$ and $0 < r \leq 1$, for each $s \in I$ let $\pi_{d-1}(s;r)$ denote the set of all $\xi \in \hat{\R}^n$ satisfying the following conditions:
\begin{subequations} 
\begin{align}\label{neighbourhood 1}
     |\inn{\be_j(s)}{\xi}| &\leq r^{d+1-j} \qquad \textrm{for $1 \leq j \leq d$,} \\
     \label{neighbourhood 2}
    1/2\leq |\inn{\be_{d+1}(s)}{\xi}| &\leq 1 \\
    \label{neighbourhood 3}
    |\inn{\be_j(s)}{\xi}| &\leq 1 \qquad \textrm{for $d+2 \leq j \leq n$.}
\end{align}
\end{subequations} 
Such sets $\pi_{d-1}(s;r)$ are referred to as $(d-1,r)$-\textit{Frenet boxes}. 
\end{definition} 

\begin{definition} A collection $\mathcal{P}_{d-1}(r)$ of $(d-1,r)$-Frenet boxes is a \textit{Frenet box decomposition for $\gamma$} if it consists of precisely the $(d-1,r)$-Frenet boxes $\pi_{d-1}(s;r)$ for $s$ varying over a $r$-net in $I$. 
\end{definition}

In some instances it is useful to highlight the underlying curve and write $\pi_{d-1,\gamma}(s;r)$ for $\pi_{d-1}(s;r)$. The relevance of the $d-1$ index is made apparent in Definition~\ref{cone gen g def}. \smallskip

Central to the proof of Theorem~\ref{non-degenerate theorem} is the following decoupling inequality.

\begin{theorem}\label{Frenet decoupling theorem} Let $2 \leq d \leq n-1$, $0 \leq \delta \ll 1$, $0 < r \leq 1$ and $\mathcal{P}_{d-1}(r)$ be a $(d-1,r)$-Frenet box decomposition for $\gamma \in \mathfrak{G}_n(\delta)$. For all $2 \leq p \leq \infty$ and $\varepsilon > 0$ the inequality
\begin{equation*}
    \big\|\sum_{\pi \in \mathcal{P}_{d-1}(r)} f_{\pi} \big\|_{L^p(\R^{n})} \lesssim_{n,\gamma,\varepsilon} r^{-\alpha(p) -\varepsilon} \Big(\sum_{\pi \in \mathcal{P}_{d-1}(r)} \|f_{\pi}\|_{L^p(\R^{n})}^p\Big)^{1/p}
\end{equation*}
holds with exponent
\begin{equation*}
    \alpha(p) := \begin{cases}
    \frac{1}{2}-\frac{1}{p} & \text{ if } \quad 2 \leq p \leq d(d+1) \\[5pt]
    1-\frac{d(d+1) +2}{2p} & \text{ if } \quad d(d+1) \leq p \leq \infty
    \end{cases}
\end{equation*}
for any tuple of functions $(f_{\pi})_{\pi \in \mathcal{P}_{d-1}(r)}$ satisfying $\supp \hat{f}_{\pi} \subseteq \pi$.
\end{theorem}

This theorem corresponds to a conic version of the Bourgain--Guth--Demeter decoupling inequality for the moment curve \cite{BDG2016}. Theorem~\ref{Frenet decoupling theorem} can be deduced from the moment curve decoupling via rescaling and induction-on-scale arguments, following a scheme originating in \cite{PS2007}. The details of this argument are presented in \S\ref{sec:decoupling}.




\section{The proof of Theorem~\ref{Sobolev theorem}: The \texorpdfstring{$J=3$}{} case}\label{sec:J=3}

We now turn to the proof of Theorem~\ref{non-degenerate theorem} proper. Recall, it remains to prove the $J=3$ and $J=4$ cases of Theorem~\ref{Sobolev theorem}. Here we present the analysis of the $J=3$ case, which essentially mirrors that of \cite{PS2007}. The present section can therefore be thought of as a warm up for the significantly more involved argument used to treat $J=4$ in \S\ref{sec:J=4}.




\subsection{Preliminaries} Suppose $\gamma \in \mathfrak{G}_4(\delta_0)$ and $a \in C^{\infty}(\hat{\R}^4\setminus \{0\} \times \R)$ satisfies the hypotheses Theorem~\ref{Sobolev theorem} for $J = 3$, with $\delta_1 :=: \delta_3 := \delta_0$ and $\delta_2 := \delta_0^3$. Note, in particular, that 
\begin{equation}\label{J3 hyp}
 \left\{\begin{array}{ll}
    |\inn{\gamma^{(3)}(s)}{\xi}| \ge \delta_0 |\xi| \\[5pt]
|\inn{\gamma^{(j)}(s)}{\xi}| \le 8\delta_0 |\xi| & \textrm{for $j = 1,2$}
\end{array}
\right. \qquad \textrm{for all $(\xi; s)\in \xisupp a \times I_0$,} 
\end{equation}
as a consequence of \eqref{e200727e3.4}. If $a_k:= a \cdot \beta^k$, as introduced in \S\ref{sec:bandlimited}, this implies, via van der Corput's lemma with third order derivatives, that 
\begin{equation}\label{J3 trivial decay rate}
    \|m[a_k](\xi)\| \lesssim 2^{-k/3}.
\end{equation}
Arguing as for $J=2$, Plancherel's theorem and interpolation with a trivial $L^\infty$ estimate yields
\begin{equation*}
\| m[a_k] \|_{M^p(\R^4)}\lesssim 2^{-2k/3p} \qquad \textrm{for all $2 \leq p \leq \infty$.}
\end{equation*}
In order to obtain the improved bound $\| m[a_k]\|_{M^p(\R^4)} \lesssim 2^{-k/p}$, we decompose the symbol $a_k$ into localised pieces which admit more refined decay rates than \eqref{J3 trivial decay rate}.




\subsection{Geometry of the slow decay cone} The first step is to isolate regions of the frequency space where the multiplier $m[a]$ decays relatively slowly. Owing to stationary phase considerations, this corresponds to a region around the cone
\begin{equation*}
    \Gamma := \set{\xi\in \xisupp a : \inn{\gamma^{(j)}(s)}{\xi}=0,\: j = 1, 2, \text{ for some } s\in I_0}. 
\end{equation*}
To analyse this region, and the corresponding decay rates for $m[a]$, we make the following simple observation. 

\begin{lemma}\label{200801lem5_1} If $\xi\in \xisupp a$, then the equation $\inn{\gamma''(s)}{\xi}=0$ has a unique solution in $\tfrac{5}{4} \cdot I_0$. 
\end{lemma}

The above lemma follows from the localisation of the symbol in \eqref{J3 hyp} and \eqref{J derivatives} via the mean value theorem. The details are left to the interested reader (see \cite[Lemma 6.1]{BGHS-helical} for a proof using similar arguments).

%
\medskip

Using Lemma~\ref{200801lem5_1}, we construct a smooth mapping $\theta \colon \xisupp a  \to [-1,1]$ such that
\begin{equation*}
    \inn{\gamma'' \circ \theta(\xi)}{\xi} = 0 \qquad \textrm{for all $\xi \in \xisupp a$.}
\end{equation*}
It is easy to see that $\theta$ is homogeneous of degree 0. This function can be used to construct a natural Whitney decomposition with respect to the cone $\Gamma$ defined above. In particular, let 
\begin{equation}\label{J=3 u function}
    \bu(\xi) := \inn{\gamma' \circ \theta(\xi)}{\xi} \qquad \textrm{for all $\xi \in \xisupp a$;}
\end{equation}
this quantity plays a central r\^ole in our analysis. If $u(\xi)=0$, then $\xi \in \Gamma$ and so, roughly speaking, $u(\xi)$ measures the distance of $\xi$ from $\Gamma$.




\subsection{Decomposition of the symbols}\label{subsec:J3 dec symbols}  Consider the frequency localised symbols $a_k:=a \cdot \beta^k$, as introduced in \S\ref{sec:bandlimited}. We decompose each $a_k$ with respect to the size of $|\bu(\xi)|$. In particular, write 
    \begin{equation}\label{J3 dec}
       a_k = \sum_{\ell = 0}^{\floor{k/3}}  a_{k,\ell}
       \end{equation}
where $\floor{k/3}$ denotes the greatest integer less than or equal to $k/3$ and\footnote{The $\beta$ function should be defined slightly differently compared with \eqref{beta def} and, in particular, here $\beta(r) := \eta(2^{-2}r) - \eta(r)$. Such minor changes are ignored in the notation.}
\begin{equation}\label{J3 akell def}
    a_{k,\ell}(\xi;s) := 
    \left\{\begin{array}{ll}
       \displaystyle a_k(\xi;s)\beta\big(2^{-k+ 2\ell}\bu(\xi)\big)   & \textrm{if $0 \leq \ell < \floor{k/3}$}  \\[8pt]
       \displaystyle a_{k}(\xi;s)\eta\big(2^{-k + 2\floor{k/3}}\bu(\xi)\big) & \textrm{if $\ell = \floor{k/3}$}
    \end{array}\right. .
    \end{equation}

The $J=3$ case of Theorem~\ref{Sobolev theorem} is a consequence of the following bound for the localised pieces of the multiplier. 

\begin{proposition}\label{J3 Lp proposition} Let $4 \leq p \leq  6$, $k \in \N$ and $\varepsilon > 0$. For all $0 \leq \ell \leq \floor{k/3}$,
    \begin{equation*}
        \|m[a_{k,\ell}]\|_{M^p(\R^4)} \lesssim_{\varepsilon, p} 2^{-k/p - \ell(1/2 - 2/p - \varepsilon)}.
    \end{equation*}
\end{proposition}

\begin{proof}[Proof of $J=3$ case of Theorem~\ref{Sobolev theorem}, assuming Proposition~\ref{J3 Lp proposition}] Let $4 < p \leq 6$ and define $\varepsilon_p := \tfrac{1}{2}\big(\tfrac{1}{2} - \tfrac{2}{p}\big) > 0$. Apply the decomposition \eqref{J3 dec} and Proposition~\ref{J3 Lp proposition} to deduce that 
\begin{equation*}
    \|m[a_k]\|_{M^p(\R^4)} \leq \sum_{\ell = 0}^{\floor{k/3}} \|m[a_{k,\ell}]\|_{M^p(\R^4)} \lesssim_p 2^{-k/p} \sum_{\ell = 0}^{\infty} 2^{-\ell(1/2 - 2/p - \varepsilon_p)} \lesssim_p 2^{-k/p}. 
\end{equation*}
This establishes the desired result for $4 < p \leq 6$. The remaining range $6< p \leq \infty$ follows by interpolation with a trivial $L^\infty$ estimate. 
\end{proof}

The rest of \S\ref{sec:J=3} is devoted to establishing Proposition~\ref{J3 Lp proposition}. Before proceeding, it is instructive to reflect on the rationale behind the decomposition \eqref{J3 dec}. A lower bound on $|u(\xi)|$ ensures that the functions $s \mapsto \inn{\gamma'(s)}{\xi}$ and $s \mapsto \inn{\gamma''(s)}{\xi}$ do not vanish simultaneously. Quantifying this observation, one obtains, via the van der Corput lemma, the decay estimate
\begin{equation}\label{J3 decay a}
    |m[a_{k,\ell}](\xi)| \lesssim 2^{-k/2+\ell/2};
\end{equation}
see Lemma~\ref{J3 L2 lemma} below. This improves upon the trivial decay rate \eqref{J3 trivial decay rate} since $\ell$ varies over the range $0 \leq \ell \leq \floor{k/3}$. Note that $\ell = k/3$ corresponds to the critical value where \eqref{J3 trivial decay rate} and \eqref{J3 decay a} agree. 

By Plancherel's theorem, \eqref{J3 decay a} implies
\begin{equation*}
    \|m[a_{k,\ell}]\|_{M^2(\R^4)} \lesssim 2^{-k/2 + \ell/2}. 
\end{equation*}
As $\ell$ increases this estimate becomes weaker. To compensate for this, we attempt to establish stronger estimates for the $M^{\infty}(\R^4)$ norm. This is not possible, however, for the entire multiplier and a further decomposition is required. The $u(\xi)$ localisation means that $m[a_{k,\ell}]$ is supported in a neighbourhood of the cone $\Gamma$. Consequently, one may apply a decoupling theorem for this cone (in particular, an instance of Theorem~\ref{Frenet decoupling theorem}) to \textit{radially} decompose the multipliers. It transpires that each radially localised piece is automatically localised along the curve in the physical space, and this leads to favourable $M^{\infty}(\R^4)$ bounds: see Lemma~\ref{J3 localisation lemma} and Lemma~\ref{J3 Linfty lemma} below.



    
\subsection{Fourier localisation and decoupling}

The first step towards Proposition~\ref{J3 Lp proposition} is to radially decompose the symbols in terms of $\theta(\xi)$. Fix a smooth cutoff $\zeta \in C^{\infty}(\R)$ with $\supp \zeta \subseteq [-1,1]$ such that $\sum_{k \in \Z} \zeta(\,\cdot\, - k) \equiv 1$ and write 
\begin{equation}\label{J=3 s theta loc}
    a_{k,\ell}(\xi;s)  = \sum_{\mu \in \Z} a_{k,\ell}^{\mu}(\xi; s)
\qquad \textrm{where} \qquad
    a_{k,\ell}^{\mu}(\xi;s) := a_{k,\ell}(\xi;s) \zeta(2^{\ell}\theta(\xi) - \mu).
\end{equation}

Given $0 < r \leq 1$ and $s \in I$, recall the definition of the $(1,r)$-\textit{Frenet boxes} $\pi_{1}(s;r)$ introduced in Definition~\ref{def Frenet box}:
\begin{equation}\label{J=3 symbol supp}
    \pi_1(s;r) := \big\{\xi \in \hat{\R}^4: |\inn{\be_j(s)}{\xi}| \lesssim r^{3-j} \quad \textrm{for $j = 1$, $2$,} \quad |\inn{\be_{3}(s)}{\xi}| \sim 1,  \quad |\inn{\be_{4}(s)}{\xi}| \lesssim 1\big\}.
\end{equation}
Here $(\be_j)_{j=1}^4$ denotes the Frenet frame, as introduced in \S\ref{Frenet subsection}. The multipliers $a_{k,\ell}^\mu$ satisfy the following support properties.

\begin{lemma}\label{J=3 supp lem} With the above definitions,
\begin{equation*}
    \xisupp a_{k, \ell}^{\mu} \subseteq 2^k \cdot \pi_1(s_{\mu}; 2^{-\ell})
\end{equation*}
for all $0 \leq \ell \leq \floor{k/3}$ and $\mu \in \Z$, where $s_{\mu}:=2^{-\ell }\mu$.
\end{lemma}

\begin{proof} For $\xi \in \xisupp a_{k, \ell}^{\mu}$ observe that 
\begin{equation*}
    |\inn{\gamma^{(i)}\circ \theta(\xi)}{\xi}| \lesssim 2^{k - (3 - i)\ell \vee 0 } \qquad \textrm{for $1 \leq i \leq 4$,} \qquad  |\inn{\gamma^{(3)}\circ \theta(\xi)}{\xi}| \sim 2^k  .
\end{equation*}
Since the Frenet vectors $\be_i \circ \theta(\xi)$ are obtained from the $\gamma^{(i)}\circ \theta(\xi)$ via the Gram--Schmidt process, the matrix corresponding to change of basis from $\big(\be_i \circ \theta(\xi)\big)_{i=1}^4$ to $\big(\gamma^{(i)} \circ \theta(\xi)\big)_{i=1}^4$ is lower triangular. Furthermore, the initial localisation implies that this matrix is an $O(\delta)$ perturbation of the identity. Consequently, provided $\delta >0$ is chosen sufficiently small,
\begin{equation*}
   |\inn{\be_i\circ \theta(\xi)}{\xi}| \lesssim 2^{k - (3 - i)\ell \vee 0} \qquad \textrm{for $1 \leq i \leq 4$}, \qquad  |\inn{\be_3\circ \theta(\xi)}{\xi}| \sim 2^k.
   \end{equation*}
On the other hand, by \eqref{J=3 s theta loc} we also have $|\theta(\xi) - s_{\mu}| \lesssim 2^{-\ell}$ and so \eqref{Frenet bound} implies that
\begin{equation*}
    |\inn{\be_{i}\circ \theta(\xi)}{\be_{j}(s_\mu)}| \lesssim |\theta(\xi)-s_{\mu}|^{|i-j|} \lesssim 2^{-(i-j)\ell}.
\end{equation*}
Writing $\xi$ with respect to the orthonormal basis $\big(\be_j \circ \theta(\xi)\big)_{j=1}^4$, it follows that
\begin{equation*}
    |\inn{\be_j(s_{\mu})}{\xi}| \leq \sum_{i=1}^4 |\inn{\be_i\circ \theta(\xi)}{\xi}| |\inn{\be_i\circ\theta(\xi)}{\be_j (s_{\mu})}| \lesssim 2^{k - (3-j)\ell \vee 0}.
\end{equation*}
Thus, $\xi$ satisfies all the required upper bounds appearing in \eqref{J=3 symbol supp}.  Provided the parameter $\delta > 0$ is sufficiently small, the argument can easily be adapted to prove the remaining lower bound for $\inn{\be_3(s_\mu)}{\xi}$.
\end{proof}

In view of the Fourier localisation described above, we have the following decoupling inequality. 
  
\begin{proposition}\label{J=3 dec prop}
For all $2 \leq p \leq 6$ and $\varepsilon > 0$ one has
   \begin{equation*}
       \Big\|\sum_{\mu \in \Z} m[a_{k,\ell}^{\mu}](D)f\Big\|_{L^p(\R^4)} \lesssim_{\varepsilon} 2^{\ell(1/2 - 1/p) + \varepsilon \ell} \Big(\sum_{\mu \in \Z} \|m[a_{k,\ell}^{\mu}](D)f\|_{L^p(\R^4)}^p\Big)^{1/p}.
   \end{equation*}
\end{proposition}

\begin{proof}
In view of the support conditions from Lemma~\ref{J=3 supp lem}, after a simple rescaling, the desired result follows from Theorem~\ref{Frenet decoupling theorem} with $d - 1 = 1$, $n=4$ and $r= 2^{-\ell}$.
\end{proof}



    
\subsection{Localisation along the curve}

The $\theta(\xi)$ localisation introduced in the previous subsection induces a corresponding localisation along the curve in the physical space. In particular, the main contribution to $m[a_{k,\ell}^{\mu}]$ arises from the portion of the curve defined over the interval $|s - s_{\mu}| \leq  2^{-\ell}$. This is made precise in Lemma~\ref{J3 localisation lemma} below.

Here it is convenient to introduce a `fine tuning' constant $\rho > 0$. This is a small (but absolute) constant which plays a minor technical r\^ole in the forthcoming arguments: taking $\rho := 10^{-6}$ more than suffices for our purposes.

For $0 \leq \ell \leq \floor{k/3}$, $\mu \in \Z$ and $\varepsilon>0$, define
\begin{equation}\label{eq:dec J=3 along curve}
    a_{k,\ell}^{\mu, (\varepsilon)}(\xi;s) := a_{k,\ell}(\xi;s) \zeta(2^{\ell}\theta(\xi) - \mu)\eta(\rho 2^{\ell(1-\varepsilon)}(s - s_\mu)).
\end{equation}
The key contribution to the multiplier comes from the symbol $a_{k,\ell}^{\mu, (\varepsilon)}$.

 \begin{lemma}\label{J3 localisation lemma} Let $2 \le p < \infty$ and $\varepsilon > 0$. For all $0 \leq  \ell \leq  \floor{k/3}$,
    \begin{equation*}
    \|m[a_{k,\ell}^{\mu}-a_{k,\ell}^{\mu,(\varepsilon)}]\|_{M^p(\R^4)} \lesssim_{N, \varepsilon, p} 2^{-kN} \qquad \textrm{for all $N \in \N$.}
    \end{equation*}

  \end{lemma}  
  
 \begin{proof}  It is clear that the multipliers satisfy a trivial $L^{\infty}$-estimate with operator norm $O(2^{Ck})$ for some absolute constant $C \geq 1$. Thus, by interpolation, it suffices to prove the rapid decay estimate for $p = 2$ only. This amounts to showing that, under the hypotheses of the lemma, \begin{equation*}
     \|m[a_{k,\ell}^{\mu}-a_{k,\ell}^{\mu,(\varepsilon)}]\|_{L^{\infty}(\hat{\R}^4)} \lesssim_{N,\varepsilon} 2^{-kN} \qquad \textrm{for all $N \in \N$}.
 \end{equation*}
 
 Here the localisation of the $a_{k,\ell}$ symbols ensures that
\begin{equation}\label{J=3 curve loc 1}
     |\bu(\xi)| \lesssim 2^{k-2\ell}  \qquad \textrm{for all $(\xi;s) \in \supp (a_{k,\ell}^{\mu}-a_{k,\ell}^{\mu, (\varepsilon)})$,}
\end{equation}
where $u$ is the function introduced in \eqref{J=3 u function}. On the other hand, provided $\rho$ is sufficiently small, the additional localisation in \eqref{J=3 s theta loc} and \eqref{eq:dec J=3 along curve} implies, via the triangle inequality,
\begin{equation}\label{J=3 curve loc 2}
|s-\theta(\xi)| \gtrsim \rho^{-1} 2^{-\ell(1 - \varepsilon)} \quad \textrm{for all $(\xi;s) \in \supp (a_{k,\ell}^{\mu}-a_{k,\ell}^{\mu,(\varepsilon)})$}.
\end{equation}

Fix $\xi \in \xisupp (a_{k,\ell}^{\mu}-a_{k,\ell}^{\mu,(\varepsilon)})$ and consider the oscillatory integral $m[a_{k,\ell}^{\mu}-a_{k,\ell}^{\mu,(\varepsilon)}](\xi)$, which has phase $s \mapsto \inn{\gamma(s)}{\xi}$. Taylor expansion around $\theta(\xi)$ yields
\begin{align}\label{J=3 curve loc 3}
    \inn{\gamma'(s)}{\xi} &= \bu(\xi) + \omega_1(\xi;s) \cdot (s-\theta(\xi))^2,  \\
\label{J=3 curve loc 4}
    \inn{\gamma''(s)}{\xi} &=\omega_2(\xi;s) \cdot (s-\theta(\xi))
\end{align}
where the $\omega_i$ arise from the remainder terms and satisfy $|\omega_i(\xi;s)| \sim  2^k$. Provided $\rho$ is sufficiently small, \eqref{J=3 curve loc 1} and \eqref{J=3 curve loc 2} imply that the $\omega_1 (\xi;s) \cdot (s-\theta(\xi))^2$ term dominates the right-hand side of \eqref{J=3 curve loc 3} and therefore 
 \begin{equation}\label{J=3 curve loc 5}
     |\inn{\gamma'(s)}{\xi}| \gtrsim  2^{k} |s-\theta(\xi)|^2\qquad  \textrm{for all $(\xi;s) \in \supp (a_{k,\ell}^{\mu}-a_{k,\ell}^{\mu,(\varepsilon)})$.}
 \end{equation}
Furthermore, \eqref{J=3 curve loc 4}, \eqref{J=3 curve loc 5} and the localisation \eqref{J=3 curve loc 2} immediately imply
\begin{align*}
    |\inn{\gamma''(s)}{\xi}| &  \lesssim  2^{-k+3\ell(1-\varepsilon)}|\inn{\gamma'(s)}{\xi}|^2, \\
    |\inn{\gamma^{(j)}(s)}{\xi}| & \lesssim 2^{k} \lesssim_j 2^{-(k-3\ell(1-\varepsilon))(j-1)} |\inn{\gamma'(s)}{\xi}|^j
     \qquad \text{for all $j \geq 3$}
\end{align*}
for all $(\xi;s) \in \supp (a_{k,\ell}^{\mu}-a_{k,\ell}^{\mu,(\varepsilon)})$.

 On the other hand, by the definition of the symbols, \eqref{J=3 curve loc 5} and the localisation \eqref{J=3 curve loc 2},
  \begin{equation*}
 |\partial_s^N (a_{k,\ell}^{\mu}-a_{k,\ell}^{\mu,(\varepsilon)})(\xi;s)| \lesssim_N 2^{\ell N} \lesssim 2^{-(k-3\ell)N - 2\varepsilon \ell N}|\inn{\gamma'(s)}{\xi}|^N \qquad \textrm{for all $N \in \N$.}
 \end{equation*}
 Thus, by repeated integration-by-parts (via Lemma \ref{non-stationary lem}, with $R=2^{k-3\ell+ 2\varepsilon \ell } \geq 1$), 
\begin{equation*}
    |m[a_{k,\ell}^{\mu}-a_{k,\ell}^{\mu,(\varepsilon)}](\xi)| \lesssim_N 2^{-(k-3\ell)N - 2\varepsilon \ell N} \qquad \textrm{for all $N \in \N$.}
\end{equation*}
 Since $0 \leq \ell \leq \floor{k/3} \leq k/3$, the desired bound follows. 
 \end{proof}



    
\subsection{Estimating the localised pieces}

The multiplier operators $m[a_{k,\ell}^{\mu, (\varepsilon)}](D)$ satisfy favourable $L^2$ and $L^{\infty}$ bounds, owing to the $u(\xi)$ and $s$ localisation, respectively. 

 \begin{lemma}\label{J3 L2 lemma} For all $0 \leq \ell \leq \floor{k/3}$, $\mu \in \Z$ and $\varepsilon > 0$,  we have
    \begin{equation*}
    \|m[a_{k,\ell}^{\mu, (\varepsilon)}]\|_{M^2(\R^4)} \lesssim  2^{-k/2 + \ell/2}.  
    \end{equation*}
  \end{lemma}      
  
\begin{proof}  
If $\ell = \floor{k/3}$, then the desired estimate follows from Plancherel's theorem and van der Corput lemma with third order derivatives, as the localisation \eqref{J3 hyp} implies
\begin{equation*}
|\inn{\gamma^{(3)}(s)}{\xi}| \gtrsim 2^k \qquad \textrm{for all $(\xi, s) \in \supp a_{k,\ell}$.}
\end{equation*}
For the remaining case, it suffices to show that
\begin{equation}\label{J3 L2 1}
    |\inn{\gamma'(s)}{\xi}| + 2^{-\ell}|\inn{\gamma''(s)}{\xi}| \gtrsim 2^{k-2\ell} \qquad \textrm{for all $(\xi;s) \in \supp a_{k,\ell}^{\mu, (\varepsilon)}$.}
\end{equation}
Here the localisation of the symbol ensures the key property
\begin{equation}\label{J3 L2 2}
|u(\xi)| \sim 2^{k -2\ell}  \qquad \text{for all $(\xi;s) \in \supp a_{k,\ell}^{\mu, (\varepsilon)}$.}
\end{equation}
Indeed, this follows from \eqref{J3 akell def} together with the hypothesis $0 \leq \ell < \floor{k/3}$.

By Taylor expansion around $\theta(\xi)$, one has
\begin{align}
\label{J3 L2 3}
    \inn{\gamma'(s)}{\xi} &= \bu(\xi) + \omega_1(\xi;s) \cdot (s - \theta(\xi))^2, \\
    \label{J3 L2 4}
    \inn{\gamma''(s)}{\xi} &= \omega_2(\xi;s) \cdot (s - \theta(\xi)),
\end{align}
where the functions $\omega_i$ arise from the remainder terms and satisfy $|\omega_i(\xi;s)| \sim  2^k$ for $i= 1$, $2$.

 The analysis now splits into two cases.
 \medskip
 
 \noindent\textbf{Case 1:} $|s - \theta(\xi)| < \rho 2^{-\ell}$. Provided $\rho$ is sufficiently small, \eqref{J3 L2 2} implies that the $u(\xi)$ term dominates in the right-hand side of \eqref{J3 L2 3} and therefore $|\inn{\gamma'(s)}{\xi}| \gtrsim  2^{k-2\ell}$.

\medskip

\noindent\textbf{Case 2:} $|s - \theta(\xi)| \geq \rho 2^{-\ell}$. In this case, \eqref{J3 L2 4} implies that $|\inn{\gamma''(s)}{\xi}| \gtrsim \rho 2^{k-\ell}$.
\medskip

In either case, the desired bound \eqref{J3 L2 1} holds.
\end{proof}

\begin{lemma}\label{J3 Linfty lemma}
For all $0 \leq \ell \leq \floor{k/3}$, $\mu \in \Z$ and $\varepsilon > 0$,  
    \begin{equation*}
   \|m[a_{k,\ell}^{\mu, (\varepsilon)}]\|_{M^{\infty}(\R^4)} \lesssim  2^{-\ell(1 - \varepsilon)}.
    \end{equation*}
  \end{lemma}      
  
  \begin{proof} By Lemma~\ref{J=3 supp lem}, we have $\xisupp a_{k, \ell}^{\mu} \subseteq 2^k \cdot \pi_1(s_{\mu}; 2^{-\ell})$. Consequently, an integration-by-parts argument (see Lemma~\ref{gen Linfty lem}) reduces the problem to showing
\begin{equation}\label{J3 Linfty 1}
    |\nabla_{\bm{v}_j}^{N} a_{k,\ell}^{\mu}(\xi)| \lesssim_N 2^{-(k - (3-j)\ell \vee 0) N} \qquad \textrm{for all $1 \leq j \leq 4$ and all $N \in \N_0$,}
\end{equation}
 where $\nabla_{\bm{v}_j}$ denotes the directional derivative in the direction of the vector $\bm{v}_j := \be_j(s_\mu)$.\medskip 

Given $\xi \in \xisupp a_{k, \ell}^{\mu}$, we claim that
\begin{equation}\label{J3 Linfty 2}
    2^{\ell}|\nabla_{\bm{v}_j}^N \theta(\xi)| \lesssim_N 2^{-(k-(3-j)\ell \vee 0)N}\qquad \text{ and } \qquad  2^{-k+2\ell}|\nabla_{\bm{v}_j}^N u(\xi)| \lesssim_N 2^{-(k-(3-j)\ell \vee 0)N}
\end{equation}
for all $N \in \N$. Assuming that this is so, the derivative bounds \eqref{J3 Linfty 1} follow directly from the chain and Leibniz rule, applying \eqref{J3 Linfty 2}. 

The claimed bounds in \eqref{J3 Linfty 2} follow from repeated application of the chain rule, provided 
\begin{subequations}
\begin{align}\label{J3 Linfty 3a}
    |\inn{\gamma^{(3)}\circ \theta(\xi)}{\xi}| &\gtrsim 2^k, \\
    \label{J3 Linfty 3b}
    |\inn{\gamma^{(K)}\circ \theta(\xi)}{\xi}| &\lesssim_K 2^{k +\ell(K-3)},  \\
    \label{J3 Linfty 3c}
    |\inn{\gamma^{(K)}\circ \theta(\xi)}{\bm{v}_j}| &\lesssim_K 2^{(3-j)\ell \vee 0 +\ell(K-3)}
\end{align}
\end{subequations}
hold for all $K \geq 2$ and all $\xi \in \xisupp a_{k,\ell}^{\mu}$. In particular, assuming \eqref{J3 Linfty 3a}, \eqref{J3 Linfty 3b} and \eqref{J3 Linfty 3c}, the bounds in \eqref{J3 Linfty 2} are then a consequence of Lemma~\ref{imp deriv lem} in the appendix: \eqref{J3 Linfty 2} corresponds to \eqref{multi imp der bound} and \eqref{multi Faa di Bruno eq 2} whilst the hypotheses in the above display correspond to \eqref{multi imp deriv 2} and \eqref{multi imp deriv 3}. Here the parameters featured in the appendix are chosen as follows: 
\begin{center}
    \begin{tabular}{|c|c|c|c|c|c|c|} 
\hline
 & & & & & & \\[-0.8em]
$g$ & $h$ & $A$ & $B$ & $M_1$ & $M_2$ & $\be$ \\
 & & & & & & \\[-0.8em]
\hline
  & & & & & & \\[-0.8em]
$\gamma''$ & $\gamma'$ & 
$2^{k-\ell}$ & $2^{k-2\ell}$ & $2^{-k+(3-j)\ell \vee 0}$  &  $2^{\ell}$  & $\bm{v}_j$ \\
 & & & & & & \\[-0.8em]
 \hline
\end{tabular}
\end{center}
See Example~\ref{deriv ex}.

The conditions \eqref{J3 Linfty 3a}, \eqref{J3 Linfty 3b} and \eqref{J3 Linfty 3c} are direct consequences of the support properties of the $a_{k,\ell}^{\mu}$. 
Indeed, \eqref{J3 Linfty 3a} and the $K \geq 3$ case of \eqref{J3 Linfty 3b} are trivial consequences of the localisation of the symbol $a_k$. The remaining $K = 2$ case of \eqref{J3 Linfty 3b} follows immediately since $\inn{\gamma'' \circ \theta(\xi)}{\xi}=0$. Finally, \eqref{Frenet bound alt 1} together with the $\theta$ localisation imply
\begin{equation*}
      |\inn{\gamma^{(K)}\circ \theta(\xi)}{\bm{v}_j}| \lesssim_K |\theta(\xi) - s_{\mu}|^{(j-K) \vee 0} \lesssim 2^{-((j-K) \vee 0)\ell}
\end{equation*}
and this is easily seen to imply \eqref{J3 Linfty 3c}.
\end{proof}

Lemma~\ref{J3 L2 lemma} and Lemma~\ref{J3 Linfty lemma} can be combined to obtain the following $L^p$ bounds.

\begin{corollary}\label{J3 corollary} Let $0 \leq \ell \leq \floor{k/3}$ and $\varepsilon > 0$. For all $2 \leq p \leq \infty$,
    \begin{equation*}
   \Big(\sum_{\mu \in \Z} \|m[a_{k,\ell}^{\mu, (\varepsilon)}](D)f\|_{L^{p}(\R^4)}^p \Big)^{1/p} \lesssim 2^{-k/p - \ell(1-3/p) + \varepsilon\ell} \|f\|_{L^p(\R^4)}.
    \end{equation*}
    When $p = \infty$ the left-hand $\ell^p$-sum is interpreted as a supremum in the usual manner. 
  \end{corollary}
  
  \begin{proof} For $p = 2$ the estimate follows by combining the $L^2$ bounds from Lemma~\ref{J3 L2 lemma} with a simple orthogonality argument. For $p = \infty$ the estimate is a restatement of the $L^{\infty}$ bound from Lemma~\ref{J3 Linfty lemma}. Interpolating these two endpoint cases, using mixed norm interpolation (see, for instance, \cite[\S 1.18.4]{Triebel1978}), concludes the proof. 
  \end{proof}
  


    
\subsection{Putting everything together} We are now ready to combine the ingredients to conclude the proof of Proposition~\ref{J3 Lp proposition}. 

\begin{proof}[Proof of Proposition~\ref{J3 Lp proposition}]  
  By Proposition~\ref{J=3 dec prop}, for all $2 \leq p \leq 6$ and all $\varepsilon > 0$ one has
   \begin{equation*}
       \| m[a_{k,\ell}](D)f \|_{L^p(\R^4)}=\Big\|\sum_{\mu \in \Z} m[a_{k,\ell}^{\mu}](D)f\Big\|_{L^p(\R^4)} \lesssim_{\varepsilon} 2^{\ell(1/2 - 1/p) + \varepsilon \ell} \Big(\sum_{\mu \in \Z} \|m[a_{k,\ell}^{\mu}](D)f\|_{L^p(\R^4)}^p\Big)^{1/p}.
   \end{equation*}
Moreover, for all $2 \leq p < \infty$, $\mu \in \Z$ and all $\varepsilon > 0$, Lemma~\ref{J3 localisation lemma} implies that
  \begin{equation*}
    \|m[a_{k,\ell}^\mu]\|_{M^p(\R^4)}   \lesssim_{N, \varepsilon, p} \Big\|\sum_{\mu \in \Z} m[a_{k,\ell}^{\mu, (\varepsilon)}]\Big\|_{M^p(\R^4)} + 2^{-kN} \quad \text{ for all $N \in \N$}.
  \end{equation*}
  Combining the above, we obtain that for all $2 \leq p \leq 6$ and all $\varepsilon>0$,
  \begin{equation*}
      \| m[a_{k,\ell}](D)f \|_{L^p(\R^4)} \lesssim_{\varepsilon, p} 2^{\ell(1/2 - 1/p) + \varepsilon \ell} \Big(\sum_{\mu \in \Z} \|m[a_{k,\ell}^{\mu}](D)f\|_{L^p(\R^4)}^p\Big)^{1/p} + 2^{-kN} \| f \|_{L^p(\R^4)}
  \end{equation*}
which together with Corollary~\ref{J3 corollary} yields
     \begin{equation*}
       \| m[a_{k,\ell}](D)f \|_{L^p(\R^4)} \lesssim_{\varepsilon} 2^{-k/p-\ell(1/2 - 2/p- 2\varepsilon) }  \|f\|_{L^p(\R^4)}.
   \end{equation*}
   Since $\varepsilon > 0$ was chosen arbitrarily, this is the required bound.
  \end{proof}

We have established Proposition~\ref{J3 Lp proposition} and therefore completed the proof of the $J=3$ case of Theorem~\ref{Sobolev theorem}.




\section{The proof of Theorem~\ref{Sobolev theorem}: The  \texorpdfstring{$J=4$}{} case} \label{sec:J=4}

The analysis used to prove the $J=4$ case of Theorem~\ref{Sobolev theorem} is much more involved than that for $J=3$. This case constitutes to the main content of Theorem~\ref{Sobolev theorem}. 




\subsection{Preliminaries} Suppose $\gamma \in \mathfrak{G}_4(\delta_0)$ and $a \in C^{\infty}(\hat{\R}^4\setminus \{0\} \times \R)$ satisfies the hypotheses of Theorem~\ref{Sobolev theorem} for $J = 4$ with $\delta_1 :=: \delta_3 := \delta_0$, $\delta_2 := \delta_0^3$ and $\delta_4 := 9/10$.\footnote{The choice $\delta_2 := \delta_0^3$ is not relevant to this part of the argument (we may simply take $\delta_2 := \delta_0$) but is used for consistency with the previous section.} Note, in particular, that 
\begin{equation}\label{4 derivative bound_old}
 \left\{\begin{array}{ll}
    |\inn{\gamma^{(4)}(s)}{\xi}| \ge \frac{9}{10}\cdot |\xi| \\[5pt]
|\inn{\gamma^{(j)}(s)}{\xi}| \le 8\delta_0 |\xi| & \textrm{for $j = 1,2,3$}
\end{array}
\right. \qquad \textrm{for all $(\xi; s)\in \xisupp a \times I_0$,} 
\end{equation}
as a consequence of \eqref{e200727e3.4}. We note two further consequences of this technical reduction:
\begin{itemize}
    \item Recall $\gamma^{(j)}(0)=\vec{e}_j$ for $1 \leq j \leq 4$ and so \eqref{4 derivative bound_old} immediately implies that 
    \begin{equation*}
      |\xi_4|\ge \tfrac{9}{10} \cdot |\xi|    \quad \textrm{and} \quad |\xi_j|\le 8 \delta_0 |\xi| \quad \textrm{for $j=1, 2, 3$,}  \qquad \textrm{for all $\xi \in \xisupp a$.}
    \end{equation*} 
    \item Since $\gamma \in \mathfrak{G}_4(\delta_0)$, we have $\|\gamma^{(5)}\|_{\infty} \leq \delta_0$. Thus, provided $\delta_0$ is sufficiently small,
\begin{equation}\label{4 derivative bound}
    |\inn{\gamma^{(4)}(s)}{\xi}| \geq \tfrac{1}{2} \cdot |\xi| \qquad \textrm{for all $(\xi; s ) \in \xisupp a \times [-1,1]$}.
\end{equation}
Observe that this inequality holds on the large interval $[-1,1]$, rather than just $I_0$. 
\end{itemize}
Henceforth, we also assume that $\xi_4>0$ for all $\xi \in \xisupp a$. In particular,
\begin{equation}\label{convex}
    \inn{\gamma^{(4)}(s)}{\xi} > 0 \qquad \textrm{for all $(\xi; s ) \in \xisupp a \times [-1,1]$}
\end{equation}
and thus, for each $\xi \in \xisupp a$, the function $s \mapsto \inn{\gamma''(s)}{\xi}$ is strictly convex on $[-1,1]$. The analysis for the portion of the symbol supported on the set $\{\xi_4<0\}$ follows by symmetry.\medskip

 If $a_k:= a \cdot \beta^k$, as introduced in \S\ref{sec:bandlimited}, the derivative bound \eqref{4 derivative bound_old} implies, via the van der Corput lemma, that 
\begin{equation}\label{J4 trivial decay rate}
    |m[a_k](\xi)| \lesssim 2^{-k/4}.
\end{equation}
Thus, Plancherel's theorem and interpolation with a trivial $L^\infty$ estimate, as in the $J=2$ case, yields
\begin{equation*}
    \| m[a_k] \|_{M^p(\R^4)}\lesssim 2^{-k/2p} \qquad \textrm{for all $2 \leq p \leq \infty$.}
\end{equation*}
As in the $J=3$ case, to obtain the improved bound $\| m[a_k]\|_{M^p(\R^4)} \lesssim 2^{-k/p}$, we decompose the symbol $a_k$ into localised pieces which admit more refined decay rates than \eqref{J4 trivial decay rate}. This decomposition is, however, significantly more involved than that used in the previous section.

\subsection{Geometry of the slow decay cones} The first step is to isolate regions of the frequency space where the multiplier $m[a]$ decays relatively slowly. Owing to stationary phase considerations, this corresponds to the regions around the conic varieties
\begin{equation*}
    \Gamma_{d-1}:=\{\xi \in \xisupp a:  \inn{\gamma^{(j)}(s)}{\xi}=0, \,\, 1 \leq j \leq d, \text{ for some } s\in I_0 \}, \qquad 2 \leq d \leq 3.
\end{equation*}
Note that $\Gamma_{d-1}$ has codimension $d-1$, which motivates the choice of index. Since $\Gamma_2 \subseteq \Gamma_1$, the decay rate for the multiplier $m[a]$ depends on the relative position with respect to both cones. To analyse this, we begin with the following observation, which helps us to understand the geometry of $\Gamma_2$.

\begin{lemma}\label{lemma:2cone}
If $\xi\in \xisupp a$, then the equation $\inn{\gamma^{(3)}(s)}{\xi}= 0$ has a unique solution in $s  \in [-1,1]$, which corresponds to the unique global minimum of the function $s \mapsto \inn{\gamma''(s)}{\xi}$. Furthermore, the solution has absolute value $O(\delta_0)$.
\end{lemma}

The above lemma quickly follows from \eqref{convex} and the localisation of the symbol via the mean value theorem. A detailed proof (of a very similar result) can be found in \cite[Lemma 6.1]{BGHS-helical}.
\medskip

By Lemma~\ref{lemma:2cone}, there exists a unique smooth mapping $\theta_2: \xisupp a  \to [-1,1]$ such that
\begin{equation*}
    \inn{\gamma^{(3)} \circ \theta_2 (\xi)}{\xi}= 0 \quad \text{ for all $\xi \in \xisupp a$}.
\end{equation*}
It is easy to see that $\theta_2$ is homogeneous of degree 0. Define the quantities 
\begin{equation*}
   u_{1,2}(\xi):=\inn{\gamma' \circ \theta_2(\xi)}{\xi} \quad \textrm{and} \quad u_2(\xi):=\inn{\gamma'' \circ \theta_2(\xi)}{\xi} \quad  \textrm{ for all $\xi \in \xisupp a$.}
\end{equation*}
 Note that $\xi \in \Gamma_{2}$ if and only if $u_{1,2}(\xi)=u_2(\xi)=0$ and thus, roughly speaking, together the quantities $|u_2(\xi)|$ and $|u_{1,2}(\xi)|$ measure the distance of $\xi$ to $\Gamma_2$.\smallskip

The next observation helps us to understand the geometry of the cone $\Gamma_1$.

\begin{lemma}\label{lemma:1cone}
Let $\xi \in \xisupp a$ and consider the equation \begin{equation}\label{0404e3.29}
    \inn{\gamma''(s)}{\xi}=0.
\end{equation}
\begin{enumerate}[i)]
    \item If $u_2(\xi)>0$, then the equation \eqref{0404e3.29} has no solution on $[-1, 1]$.
\item If $u_2(\xi)=0$, then the equation \eqref{0404e3.29} has only the solution $s=\theta_2(\xi)$ on $[-1,1]$.
\item If $u_2(\xi)<0$, then the equation \eqref{0404e3.29} has precisely two solutions on $[-1,1]$. Both solutions have absolute value $O(\delta_0^{1/2})$.
\end{enumerate}
\end{lemma}

Again, this lemma quickly follows using the information in Lemma \ref{lemma:2cone}, the localisation of the symbol and Taylor expansion. The relevant details can be found in \cite[Lemma 6.2]{BGHS-helical}.
\medskip

Using Lemma~\ref{lemma:1cone}, we construct a (unique) pair of smooth mappings
\begin{equation*}
    \theta_1^{\pm} \colon \{ \xi \in \xisupp a :  u_2(\xi) <0 \} \to [-1, 1]
\end{equation*}
with $\theta_1^-(\xi)\le \theta_1^+(\xi)$ which satisfies
\begin{equation*}
    \inn{\gamma'' \circ \theta_1^{\pm}(\xi)}{\xi}= 0 \quad \text{ for all $\xi \in \xisupp a$ with  $u_2(\xi) <0$.}
\end{equation*}
Define the functions 
\begin{equation*}
        u_{1}^{\pm}(\xi):=\inn{\gamma' \circ \theta_1^\pm(\xi)}{\xi} \quad \textrm{and} \quad
    u_{3,1}^{\pm}(\xi):=\inn{\gamma^{(3)} \circ \theta_1^{\pm}(\xi)}{\xi} \qquad \text{ for all $\xi \in \xisupp a$ with  $u_2(\xi) <0$}
\end{equation*}
and note that $\xi \in \Gamma_{1}$ if and only if $u_{1}^{+}(\xi)=0$ or $u_1^-(\xi) = 0$. For this reason, we introduce 
\begin{equation*}
    u_1(\xi):=
    \begin{cases}
    u_{1}^{+}(\xi) &  \text{ if }  |u_{1}^{+}(\xi)|= \displaystyle\min_{\pm} |u_{1}^{\pm}(\xi)|\\
    u_{1}^{-}(\xi) &  \text{ if }  |u_{1}^{-}(\xi)|= \displaystyle \min_{\pm} |u_{1}^{\pm}(\xi)|
    \end{cases}
\quad \textrm{and}\quad 
 \theta_1(\xi):=
     \begin{cases}
     \theta_1^+(\xi) &  \text{ if } u_{1}(\xi)=u_{1}^{+}(\xi)\\
     \theta_1^-(\xi) &  \text{ if } u_{1}(\xi)=u_{1}^{-}(\xi)
     \end{cases},
 \end{equation*}
 which clearly satisfy
 \begin{equation*}
     u_1(\xi) = \inn{\gamma'\circ \theta_1(\xi)}{\xi}.
 \end{equation*}
Roughly speaking, the quantity $|u_1(\xi)|$ measures the distance of $\xi$ from $\Gamma_1$. Furthermore, if $\xi \in \Gamma_1$ satisfies $u_{3,1}(\xi)=0$ where
\begin{equation*}
    \uthree(\xi) := \inn{\gamma^{(3)} \circ \theta_{1}(\xi)}{\xi},
\end{equation*}
then $\xi \in \Gamma_2$. Thus, again, $|u_{3,1}(\xi)|$ may be interpreted as measuring the distance of $ \xi \in \Gamma_1$ to $\Gamma_2$.\medskip

The following lemma relates important information regarding the functions $\theta_2(\xi)$, $\theta_1^{\pm}(\xi)$ and the associated quantities $u_2(\xi)$, $u_{1,2}(\xi)$, $u_{1}^{\pm}(\xi)$, $u_{3,1}^{\pm}(\xi)$. 

\begin{lemma}\label{lemma:size of quantities}
Let $\xi \in \xisupp a$ with $u_2(\xi)<0$. Then the following hold:
\begin{enumerate}[i)]
    \item $\big|u_{3,1}^{\pm}\big(\xip\big)\big| \sim |\theta_1^{\pm}(\xi) - \theta_2(\xi)| \sim  |\theta_1^+(\xi)-\theta_1^-(\xi)|\sim \big|u_2\big(\xip\big)\big|^{1/2}$,\smallskip
    \item $\big|u_{1,2}\big(\xip\big) - u_{1}^{\pm}\big(\xip\big)\big| \lesssim \big|u_2\big(\xip\big)\big|^{3/2}$, \smallskip
    \item $\big|u_{1}^{+}\big(\xip\big) - u_{1}^{-}\big(\xip\big)\big| \sim \big|u_2\big(\xip\big)\big|^{3/2}$.
\end{enumerate}
\end{lemma}

\begin{proof}
i) This is almost immediate from Taylor expansion around $\theta_2(\xi)$, and around $\theta_1^-(\xi)$ in the last display. The interested reader is referred to \cite[Lemma 6.3]{BGHS-helical} for details of a closely related calculation.
\medskip

\noindent ii) By Taylor expansion around $\theta_1^\pm(\xi)$,
\begin{equation*}
    u_{1,2}(\xi)=u_{1}^{\pm}(\xi) + u_{3,1}^{\pm}(\xi) \cdot \frac{(\theta_2(\xi) - \theta_1^{\pm}(\xi))^2}{2} + \omega_4(\xi) \cdot (\theta_2(\xi) - \theta_1^{\pm}(\xi))^3,
\end{equation*}
where $|\omega_4(\xi)| \sim |\xi|$. The desired estimate follows from the above expansion and part i).
\medskip

\noindent iii) By part i), it suffices to show
\begin{equation}\label{eq:distance u1 proof}
    |u_{1}^{+}(\xi)-u_{1}^{-}(\xi)| \sim |u_2(\xi)| | \theta_1^+(\xi) - \theta_1^-(\xi)|.
\end{equation}
To this end, note that
\begin{equation*}
    u_{1}^{+}(\xi)-u_{1}^{-}(\xi) = \int_{\theta_1^-(\xi)}^{\theta_1^+(\xi)} \inn{\gamma''(s)}{\xi}\,  \ud s. 
\end{equation*}
By Lemma~\ref{lemma:2cone}, $u_2(\xi) \leq \inn{\gamma''(s)}{\xi}\leq 0$ for $\theta_1^-(\xi) \leq s \leq \theta_1^+(\xi)$. Thus, the upper bound in \eqref{eq:distance u1 proof} immediately follows from the above identity and the triangle inequality. To see the lower bound in \eqref{eq:distance u1 proof}, recall from \eqref{convex}  that the function $s \mapsto \phi(s):=\inn{\gamma''(s)}{\xi}$ is strictly convex in $[-1,1]$ and that $\phi \circ \theta_1^+(\xi)=\phi \circ \theta_1^-(\xi)=0$. 
As $\theta_1^-(\xi) \leq \theta_2(\xi) \leq \theta_1^+(\xi)$ and $\phi \circ \theta_2(\xi)= u_2(\xi)$, the convexity of $\phi$ implies
$$
\int_{\theta_1^-}^{\theta_1^+} |\inn{\gamma''(s)}{\xi}| \, \ud s \geq \frac{1}{2} |u_2(\xi)| | \theta_1^+(\xi) - \theta_1^-(\xi)|,
$$
and thus \eqref{eq:distance u1 proof} follows from the constant sign of $\phi(s)$ on $[\theta^-_1(\xi), \theta_1^+(\xi)]$.
\end{proof}

\subsection{Decomposition of the symbols}\label{J=4 decomp sec} For $k\ge 1$ consider the frequency localised symbols $a_k:=a \cdot \beta^k$ as defined in \S\ref{sec:bandlimited}. We begin by decomposing each $a_k$ in relation to the codimension $2$ cone $\Gamma_2$ corresponding to the directions of slowest decay for $\hat{\mu}$. In order to measure the distance to this cone, we consider the two quantities $u_{1,2}$ and $u_2$ introduced in the previous subsection and, in particular, form a simultaneous dyadic decomposition according to the relative sizes of each.\smallskip

Here it is convenient to introduce a `fine tuning' constant $\rho > 0$. This is a small (but absolute) constant which plays a minor technical r\^ole in the forthcoming arguments: taking $\rho := 10^{-6}$ more than suffices for our purposes.

\medskip

\noindent \underline{\textit{Decomposition with respect to $\Gamma_2$}}. Let $\beta, \eta \in C_c^\infty(\hat{\R})$ be the functions used to perform a Littlewood--Paley decomposition in \S\ref{sec:bandlimited}. Let $\beta_+$, $\beta_- \in C_c^\infty(\R)$ with $\supp \beta_+ \subset (0,\infty)$ and $\supp \beta_- \subset (-\infty, 0)$ be such that $\beta = \beta_+ +\beta_-$. For each $m \in \N$, write
\begin{equation*}
    \eta(2^{-1}r_1) \eta(r_2) = \sum_{\ell=0}^{m} \beta(2^{\ell-1}r_1)\eta(2^\ell r_2) + \sum_{\ell=0}^{m-1} \eta(2^\ell r_1)\big(\beta_+(2^{\ell} r_2) + \beta_-(2^{\ell}r_2) \big) + \eta(2^m r_1) \eta(2^m r_2).
\end{equation*}
The above formula corresponds to a smooth decomposition of $[-2,2] \times [-1,1]$ into  axis-parallel dyadic rectangles: see Figure~\ref{figure:dyad 2d}. 
We apply this decomposition\footnote{Here the $\beta$ function should be defined slightly differently compared with \eqref{beta def}. In particular, when acting on $r_1$ we have $\beta(r_1) := \eta(2^{-2}r_1) - \eta(r_1)$ and when acting on $r_2$ we have $\beta(r_2) := \eta(2^{-3}r_2) - \eta(r_2)$. Such minor changes are ignored in the notation.} with $r_1=2^{-k}u_{1,2}(\xi)$ and  $r_2=\rho^{-1}2^{-k}u_{2}(\xi)$. This is then used to split the symbol $a_k$ as a sum
\begin{equation*}
    a_k = \sum_{\ell=0}^{\floor{k/4}} a_{k,\ell,1} + a_{k,\ell,2} + \sum_{\ell = 0}^{\floor{k/4}-1} b_{k,\ell}
\end{equation*}
where $\floor{k/4}$ denotes the greatest integer less than or equal to $k/4$ and
\begin{align*}
    a_{k,\ell,1}(\xi; s)&:= a_{k}(\xi; s) \beta(2^{-k+3\ell} u_{1,2}(\xi))\eta(\rho^{-1}2^{-k+2\ell}u_2(\xi)) \quad \text{ $0 \leq \ell \leq \floor{k/4}$}, \\[5pt]
    a_{k,\ell,2}(\xi; s)&:=\begin{cases} a_{k}(\xi; s)\eta(2^{-k+3\ell} u_{1,2}(\xi))
    \beta_+(\rho^{-1}2^{-k+2\ell} u_2(\xi)) & \text{ if } 0 \leq \ell< \floor{k/4}\\
    a_{k}(\xi; s) \eta(2^{-k + 3\floor{k/4}} u_{1,2}(\xi)) \eta(\rho^{-1}2^{-k+2\floor{k/4}} u_2(\xi))  & \text{ if } \ell= \floor{k/4}
    \end{cases},
\\[5pt]
    b_{k,\ell}(\xi; s)&:= a_{k}(\xi; s) \eta(2^{-k+3\ell} u_{1,2}(\xi))\beta_-(\rho^{-1}2^{-k+2\ell}u_2(\xi)) \quad \text{ $0 \leq \ell < \floor{k/4}$}.
\end{align*}

\begin{figure}
\begin{tikzpicture}[scale=1.5] 

\begin{scope}[scale=2]
\draw[thick,->] (-0.05,0) -- (3,0);
\draw (3,-0.05) node[below] {$r_1$};
\draw[thick,->] (0,-0.05) -- (0,1.6);
\draw (-0.05,1.55) node[left] {$r_2$};

\draw (0.3,-0.05) -- (0.3, 0.3) -- (-0.05, 0.3);
\draw (0.3,-0.05) node[below] {$2^{-m}$};
\draw (-0.05,0.3) node[left] {$2^{-m}$};

\draw (0.3, 0.3) -- (0.6, 0.3) -- (0.6,-0.05);
\draw (0.6,-0.05) node[below] {$\quad 2^{-m+1}$};

\draw (-0.05, 0.6) -- (0.6,0.6) -- (0.6, 0.3);
\draw (-0.05,0.6) node[left] {$2^{-m+1}$};

\draw (0.6, 0.6) -- (1.2, 0.6) -- (1.2,-0.05);
\draw (1.2,-0.05) node[below] {$2^{-m+2}$};

\draw (-0.05, 1.2) -- (1.2,1.2) -- (1.2, 0.6);
\draw (-0.05,1.2) node[left] {$2^{-m+2}$};

\draw (1.2, 1.2) -- (2.4, 1.2) -- (2.4,-0.05);
\draw (2.4,-0.05) node[below] {$2^{-m+3}$};

\end{scope}

\end{tikzpicture}
\caption{Two parameter dyadic decomposition in the upper-left quadrant.}
\label{figure:dyad 2d}
\end{figure}
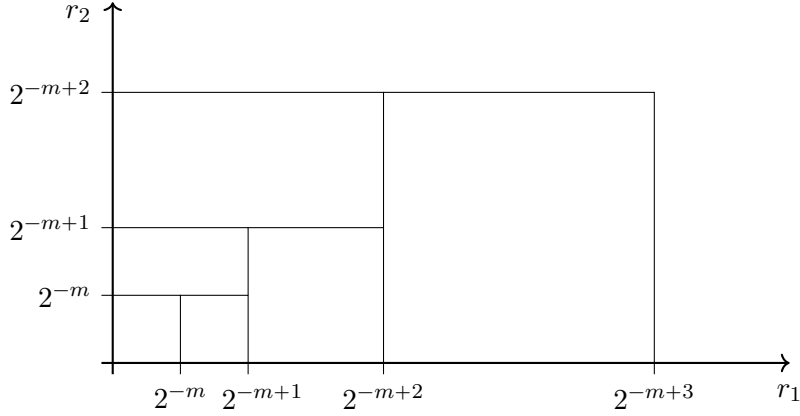

The following remarks help to motive the above decomposition:\medskip

For $\xi \in \xisupp a_{k,\ell,1}$, the functions $s \mapsto \inn{\gamma'(s)}{\xi}$ and $s \mapsto  \inn{\gamma''(s)}{\xi}$ do not vanish simultaneously. This is due, in part, to the lower bound on $|u_{2,1}(\xi)|$. On the other hand, for $\xi \in \xisupp a_{k,\ell, 2}$ we have $u_2(\xi) > 0$ and therefore $s \mapsto \inn{\gamma''(s)}{\xi}$ is non-vanishing by Lemma~\ref{lemma:1cone}. Quantifying these observations, one obtains the decay estimate
\begin{equation}\label{J4 decay a}
    |m[a_{k,\ell, \iota}](\xi)| \lesssim 2^{-k/2+\ell} \qquad \textrm{for $\iota = 1$, $2$}
\end{equation}
via the van der Corput lemma. See Lemma~\ref{L2 bounds J=4 lem} a) for details. This improves upon the trivial decay rate \eqref{J4 trivial decay rate} since $\ell$ varies over the range $0 \leq \ell \leq \floor{k/4}$. Note that $\ell = k/4$ corresponds to the critical value where \eqref{J4 trivial decay rate} and \eqref{J4 decay a} agree. \medskip

For $\xi \in \xisupp b_{k,\ell}$, as $u_2(\xi)<0$, the function $s \mapsto  \inn{\gamma''(s)}{\xi}$ vanishes at $s=\theta_1^\pm(\xi)$ by Lemma~\ref{lemma:1cone}. Moreover, the lack of a lower bound for $|u_{1,2}(\xi)|$ allows for simultaneous vanishing of $s \mapsto \inn{\gamma'(s)}{\xi}$ and $s \mapsto  \inn{\gamma''(s)}{\xi}$, in contrast with the situation considered above. However, the lower bound on $|u_2(\xi)|$ implies that the functions $s \mapsto \inn{\gamma''(s)}{\xi}$ and $s \mapsto  \inn{\gamma^{(3)}(s)}{\xi}$ do not vanish simultaneously. Quantifying these observations, one obtains, via the van der Corput lemma, the decay estimate
\begin{equation}\label{J4 trivial decay b's}
    |m[b_{k,\ell}](\xi)| \lesssim 2^{-k/3+\ell/3}.
\end{equation}
Again, this improves upon the trivial decay rate \eqref{J4 trivial decay rate} since $0 \leq \ell < \floor{k/4}$ and, furthermore, $\ell = k/4$ corresponds to the critical value where \eqref{J4 trivial decay rate} and \eqref{J4 trivial decay b's} agree. However, the estimate \eqref{J4 trivial decay b's} can be further improved by decomposing each $b_{k,\ell}$ with respect to the codimension 1 cone $\Gamma_1$. Recall that this cone corresponds to directions of slow (but not necessarily minimal) decay for $\hat{\mu}$. We proceed by performing a secondary dyadic decomposition with respect to the function $u_1$, which measures the distance to $\Gamma_1$.

\medskip

\noindent\textit{\underline{Decomposition with respect to $\Gamma_1$}}. If $\xi \in \xisupp b_{k,\ell}$, then $u_2(\xi) < 0$ and therefore the roots $\theta_1^{\pm}(\xi) \in [-1, 1]$ are well-defined by Lemma~\ref{lemma:1cone}. Observe that
\begin{equation*}
    |u_2(\xi)|\sim \rho 2^{k-2\ell}  \qquad \text{and} \qquad |u_{1,2}(\xi)|\lesssim 2^{k-3\ell} \quad \text{ for all $\xi \in \xisupp b_{k,\ell}$},
\end{equation*}
and so it follows from Lemma~\ref{lemma:size of quantities} ii) that
\begin{equation*}
    |u_{1}(\xi)|\lesssim 2^{k-3\ell} \quad \text{ for all $\xi \in \xisupp b_{k,\ell}$}.
\end{equation*}
Consequently, provided $\rho$ is chosen sufficiently small,
\begin{equation}\label{Gamma1 dec 1}
    b_{k,\ell_2}(\xi;s)= b_{k,\ell_2}(\xi;s) \eta(\rho 2^{-k+3\ell_2}u_1(\xi)).
\end{equation}

For every $k \in \N$ define the indexing set
\begin{equation*}
    \Lambda(k) := \Big\{ \bm{\ell}=(\ell_1,\ell_2) \in \Z^2 : 0 \leq \ell_2 < \floor{k/4}, \,\, \ell_2 \leq \ell_1 \leq \big\lfloor\tfrac{2k+\ell_2}{9}\big\rfloor \Big\}
\end{equation*}
and, for each $0 \leq \ell_2 < \floor{k/4}$, consider the fibre associated to its projection in the $\ell_2$-variable,
\begin{equation*}
    \Lambda(k, \ell_2):= \big\{ \bm{\ell} \in \Lambda(k) : \bm{\ell} = (\ell_1, \ell_2) \textrm{ for some $\ell_1 \in \Z$} \big\}.
\end{equation*}
In view of \eqref{Gamma1 dec 1}, we may decompose
\begin{equation*}
    b_{k,\ell_2}(\xi;s)= b_{k,\ell_2}(\xi;s) \eta(\rho 2^{-k+3\ell_2} u_{1}(\xi))= a_{k,\ell_2,3}(\xi;s) + a_{k,\ell_2,4}(\xi;s) +  \sum_{\bm{\ell} \in \Lambda(k, \ell_2)} b_{k,\bm{\ell}\,}(\xi;s)
\end{equation*}
where 
\begin{align*}
    a_{k,\ell_2,3}(\xi;s)&:=b_{k,\ell_2}(\xi;s)   \big(\eta(\rho 2^{-k+3\ell_2} u_{1}(\xi)) - \eta(\rho^{-4}2^{-k+3\ell_2} u_{1}(\xi))\big),
\\
    a_{k,\ell_2,4}(\xi;s)&:=b_{k,\ell_2}(\xi;s)   \eta(\rho^{-4}2^{-k+3\ell_2} u_{1}(\xi))\big( 1 - \eta(\rho^{-1} 2^{\ell_2}(s - \theta_1(\xi))) \big)
\end{align*}
and
\begin{equation*}
    b_{k, \bm{\ell}}(\xi;s):=
    \begin{cases}
    b_{k, \ell_2}(\xi;s) \beta(\rho^{-4}2^{-k+3\ell_1} u_{1}(\xi))\eta(\rho^{-1} 2^{\ell_2}(s - \theta_1(\xi))) & \text{ if }   \ell_1< \floor{(2k+\ell_2)/9} \\
    b_{k, \ell_2}(\xi;s) \eta(\rho^{-4}2^{-k+3\ell_1} u_{1}(\xi))\eta(\rho^{-1} 2^{\ell_2}(s - \theta_1(\xi))) & \text{ if } \ell_1= \floor{(2k+\ell_2)/9}
    \end{cases}
\end{equation*}
for $\bm{\ell} = (\ell_1, \ell_2) \in \Lambda(k)$.\medskip

\noindent\textit{\underline{The final decomposition}}. Combining the preceding definitions, we have 
\begin{equation}\label{J=4 symbol dec}
    a_{k} =   \sum_{\ell=0}^{\floor{k/4}} \sum_{\iota=1}^4 a_{k,\ell,\iota} + \sum_{\bm{\ell} \in \Lambda(k)} b_{k,\bm{\ell}}
\end{equation}
where for $\iota=3,4$ it is understood that $a_{k,\ell,\iota} \equiv 0$ for $\ell = \floor{k/4}$. This concludes the initial frequency decomposition.\medskip

The following remarks help to motive the above decomposition:\medskip

For $\xi  \in \xisupp a_{k,\ell, 3}$ or $\xi \in \xisupp a_{k, \ell, 4}$ it transpires that the functions $s \mapsto \inn{\gamma'(s)}{\xi}$ and $s \mapsto  \inn{\gamma''(s)}{\xi}$ do not vanish simultaneously. Quantifying these observations, one obtains the decay estimate
\begin{equation*}
    |m[a_{k,\ell, \iota}](\xi)| \lesssim 2^{-k/2+\ell} \qquad \textrm{for $\iota = 3$, $4$},
\end{equation*}
exactly as in \eqref{J4 decay a}. See Lemma~\ref{L2 bounds J=4 lem} a) for details. Here, however, the attendant stationary phase arguments are a little more delicate than those used to prove \eqref{J4 decay a} and, in particular, they rely on a careful analysis involving both $\Gamma_1$ and $\Gamma_2$. The lower bounds on $|u_1(\xi)|$ and $|s-\theta_1(\xi)|$ are fundamental in each case. \medskip

Turning to the $b_{k,\bm{\ell}}$ symbols, the localisation $|s-\theta_1(\xi)|\lesssim \rho 2^{-\ell_2}$ leads to the following key observation.

\begin{lemma}\label{J4 lower bound lemma}
Let $k \in \N$ and $\bm{\ell}=(\ell_1,\ell_2) \in \Lambda(k)$. Then
\begin{equation}\label{J=4 lower bound 3rd order bs}
    |\inn{\gamma^{(3)}(s)}{\xi}| \sim \rho^{1/2} 2^{k-\ell_2} \qquad \textrm{for all $(\xi;s) \in \supp b_{k,\bm{\ell}}$.}
\end{equation}
\end{lemma}

\begin{proof}
The localisation of the symbol ensures the key properties 
\begin{equation}\label{J=4 lower bound b 1}
    |u_2(\xi)| \sim \rho 2^{k-2\ell_2}, \quad |s - \theta_1(\xi)| \lesssim \rho 2^{-\ell_2} \qquad \textrm{for all $(\xi;s) \in \supp b_{k,\bm{\ell}}$}.
\end{equation}

By the mean value theorem we obtain
\begin{equation}\label{J=4 lower bound b 2}
    \inn{\gamma^{(3)}(s)}{\xi} = u_{3,1}(\xi) + \omega(\xi;s) \cdot (s-\theta_1(\xi)),
\end{equation}
where $\omega$ satisfies $|\omega(\xi;s)| \sim 2^k$. 
Observe that \eqref{J=4 lower bound b 1} and Lemma~\ref{lemma:size of quantities} i) imply $|u_{3,1}(\xi)| \sim \rho^{1/2} 2^{k - \ell_2}$. Consequently, provided $\rho$ is sufficiently small, the second inequality in \eqref{J=4 lower bound b 1} implies that the $u_{3,1}$ term dominates the right-hand side of \eqref{J=4 lower bound b 2} and therefore the desired bound \eqref{J=4 lower bound 3rd order bs} holds.
\end{proof}

The condition \eqref{J=4 lower bound 3rd order bs} reveals that the symbol $b_{k,\bm{\ell}}$ essentially corresponds to a scaled version of the multiplier $a_{k,\ell}$ from the $J=3$ case, for a suitable choice of $\ell$ and $k$. Of course, the condition \eqref{J=4 lower bound 3rd order bs} immediately implies
\begin{equation}\label{J4 trivial decay b's revisited}
    |m[b_{k,\bm{\ell}}(\xi)]| \lesssim 2^{-k/3 + \ell_2/3},
\end{equation}
as in \eqref{J4 trivial decay b's}. However, arguing as in Lemma~\ref{J3 L2 lemma}, one may improve the decay rate to
\begin{equation}\label{J4 decay b's}
    |m[b_{k,\bm{\ell}}(\xi)]| \lesssim 2^{-k/2 + (3\ell_1+\ell_2)/4};
\end{equation}
see Lemma~\ref{L2 bounds J=4 lem} b). Indeed, for each $0 \leq \ell_2 < \floor{k/4}$, the decomposition of the $a_{k,\ell}$ for the $J=3$ case in \S\ref{subsec:J3 dec symbols} matches that of the $b_{k,\bm{\ell}}$ above, with the identification 
\begin{equation*}
 k \longleftrightarrow k-\ell_2 \qquad \textrm{and} \qquad  \ell \longleftrightarrow \tfrac{3\ell_1-\ell_2}{2}.
\end{equation*}
The bound \eqref{J4 decay b's} corresponds to the conclusion of Lemma~\ref{J3 L2 lemma} once we substitute in these indices. Observe that \eqref{J4 decay b's} is indeed an improvement over the trivial decay rate \eqref{J4 trivial decay b's revisited} since, for $\ell_2$ fixed, $\ell_1$ varies over the range $0 \leq \ell_1 \leq \floor{\tfrac{2k + \ell_2}{9}}$. Note that $\ell_1 = \tfrac{2k + \ell_2}{9}$ corresponds to the critical index where \eqref{J4 trivial decay b's revisited} and \eqref{J4 decay b's} agree. 


\begin{figure}
\begin{tikzpicture}[scale=1.5] 

\begin{scope}[scale=1.5]
\draw[thick,->] (-.1,0) -- (2.4,0) node[below] {$ \ell_2$};
\draw[thick,->] (0,-.1) -- (0,2.4) node[left] {$ \ell_1$};

\draw (2,-0.025) -- (2,.025) node[below= 0.25cm] {$ k/4$}; 
\draw (-0.025,2) -- (.025,2) node[left= 0.25cm] {$ k/4$}; 

\draw[dashed] (0,0) -- (2,2);

\draw (2,2) -- (2,0);

\draw (0,1.75) -- (2,2);
\draw (-0.025,1.75) -- (.025,1.75) node[left= 0.25cm] {$ 2k/9$}; 

\draw (0.1,0) -- (2,1.9);

\draw (0,0.1) -- (1.87, 1.97);

\end{scope}
\end{tikzpicture}
\caption{Setting $\ell=\ell_1$ for $a_{k,\ell,1}$ and $\ell=\ell_2$ for $a_{k,\ell,\iota}$, $2 \leq \iota \leq 4$, one can interpret the decomposition \eqref{J=4 symbol dec} in the $(\ell_2,\ell_1)$-plane as follows. The symbols $a_{k,\ell,1}$ correspond to horizontal lines in the lower triangle, whilst the symbols $a_{k,\ell,2}$ correspond to vertical lines in the diagonal and upper triangle whenever $u_2(\xi)>0$. If $u_2(\xi)<0$, the symbols $a_{k,\ell,3}$ correspond to vertical lines in the fattened diagonal, the symbols $a_{k,\ell,4}$ correspond to vertical lines in the upper triangle (under the additional condition that $|s-\theta_1(\xi)|\gtrsim 2^{-\ell_2}$) and the symbols $b_{k,\bm{\ell}}$ correspond to integer points in the upper triangle (under the additional condition that $|s-\theta_1(\xi)|\lesssim  2^{-\ell_2}$).}
\end{figure}
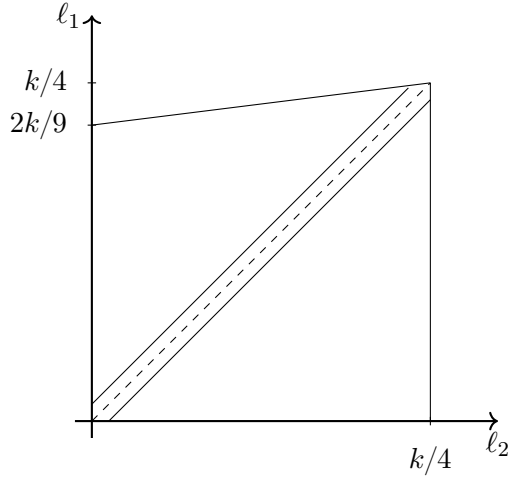


\begin{remark} The symbols in the above decomposition are in fact smooth. This is not entirely obvious, since the function $u_1$ is defined pointwise as the minimum of $|u_1^-|$ and $|u_1^+|$. Thus, $u_1$ fails to be smooth whenever $u_1^-(\xi) = \pm u_1^+(\xi)$. However, the decomposition ensures that $|u_2(\xi)| \sim \rho 2^{k-2\ell_2}$ and $|u_1(\xi)| \lesssim \rho^4 2^{k-3\ell_2}$ for all $\xi \in \xisupp a_{k,\ell_2, 4}$ or $\xi \in \xisupp b_{k,\bm{\ell}}$. Combining these facts with Lemma~\ref{lemma:size of quantities}, one easily deduces that
\begin{equation*}
 |u_1^-(\xi) \pm  u_1^+(\xi)| \gtrsim \rho^{3/2} 2^{k-3\ell_2} 
\end{equation*}
and so $u_1$ is smooth on the $\xi$-support of either $a_{k,\ell_2, 4}$ or $b_{k,\ell_2}$. Furthermore, these observations also imply that the function $\theta_1(\xi)$ is smooth on the supports. The symbol $a_{k, \ell_2, 3}$ can be treated in a similar manner, by writing it as a difference of the symbols \begin{equation*}
    b_{k,\ell_2}(\xi;s) \eta(\rho 2^{-k+3\ell_2}u_1(\xi))=b_{k,\ell_2}(\xi;s) \qquad \textrm{and} \qquad b_{k,\ell_2}(\xi;s)\eta(\rho^{-4} 2^{-k+3\ell_2}u_1(\xi)) 
\end{equation*}
and showing that both are smooth. 
\end{remark}

Given the above decomposition, in order to prove the $J=4$ case of Theorem~\ref{Sobolev theorem}, it suffices to establish the following.

\begin{proposition}\label{J4 Lp proposition} 
Let $6 \leq p \leq  12$, $k \in \N$ and $\varepsilon > 0$.
\begin{enumerate}[a)]
\item For all $0 \leq \ell \leq \floor{k/4}$ and $1 \leq \iota \leq 4$,
\begin{equation*}
    \|   m[a_{k,\ell,\iota}]\|_{M^p(\R^4)} \lesssim_{p, \varepsilon} 2^{-k/p-\ell(1/2 - 3/p - \varepsilon)}.
\end{equation*}
    \item For all $\bm{\ell}=(\ell_1,\ell_2) \in \Lambda(k)$, 
\begin{equation*}
    \| m[b_{k,\bm{\ell}}]\|_{M^p(\R^4)} \lesssim_{p, \varepsilon} 2^{-3(\ell_1 - \ell_2)(1/2p - \varepsilon) -\ell_2(1/2 - 3/p - \varepsilon)}.
\end{equation*}
\end{enumerate}
\end{proposition}

\begin{proof}[Proof of $J=4$ case of Theorem~\ref{Sobolev theorem}, assuming Proposition~\ref{J4 Lp proposition}] Let $6 < p \leq 12$ and define
\begin{equation*}
    \varepsilon_p := \frac{1}{2} \cdot \min\Big\{ \frac{1}{2} - \frac{3}{p}, \frac{1}{2p}\Big\} > 0.
\end{equation*}
Apply the decomposition \eqref{J=4 symbol dec} to deduce that
\begin{equation*}
        \|m[a_{k}]\|_{M^p(\R^4)} \leq   \sum_{\iota=1}^4 \sum_{\ell=0}^{\floor{k/4}}  \|m[a_{k,\ell,\iota}]\|_{M^p(\R^4)} + \sum_{\bm{\ell} \in \Lambda(k)} \|m[b_{k,\bm{\ell}}]\|_{M^p(\R^4)}.
\end{equation*}

By Proposition~\ref{J4 Lp proposition} a), we have
\begin{equation*}
    \sum_{\iota=1}^4 \sum_{\ell=0}^{\floor{k/4}}  \|m[a_{k,\ell,\iota}]\|_{M^p(\R^4)} \lesssim_p 2^{-k/p}\sum_{\ell = 0}^{\infty} 2^{-\ell(1/2 - 3/p - \varepsilon_p)} \lesssim_p 2^{-k/p},
\end{equation*}
Similarly, by Proposition~\ref{J4 Lp proposition} b), we have
\begin{equation*}
    \sum_{\bm{\ell} \in \Lambda(k)} \|m[b_{k,\bm{\ell}}]\|_{M^p(\R^4)} \lesssim_p 2^{-k/p}\sum_{\ell_2 = 0}^{\infty} 2^{-\ell_2(1/2 - 3/p - \varepsilon_p)} \sum_{\ell_1 = \ell_2}^{\infty} 2^{-3(\ell_1-\ell_2)(1/2p - \varepsilon_p)} \lesssim_p 2^{-k/p}.
\end{equation*}
Combining these observations establishes the desired result for $6 < p \leq 12$. The remaining range $12< p \leq \infty$ follows by interpolation with a trivial $L^\infty$ estimate.
\end{proof}

The rest of \S\ref{sec:J=4} is devoted to establishing Proposition \ref{J4 Lp proposition}. Before proceeding, it is instructive to describe the general strategy.\medskip

By Plancherel's theorem, \eqref{J4 decay a} and \eqref{J4 decay b's} imply
\begin{equation*}
    \|m[a_{k,\ell, \iota}]\|_{M^2(\R^4)} \lesssim 2^{-k/2 + \ell/2} \qquad \quad \textrm{and} \qquad 
    \|m[b_{k,\bm{\ell}}]\|_{M^2(\R^4)} \lesssim 2^{-k/2 + (3\ell_1+\ell_2)/4}.
\end{equation*}
As $\ell$, $\ell_1$ and $\ell_2$ increase, these estimates become weaker. To compensate for this, we attempt to establish stronger estimates for the $M^{\infty}(\R^4)$ norm. This is not possible, however, for the entire multipliers and a further decomposition is required. The $u_2(\xi)$ localisation means that $m[a_{k,\ell,\iota}]$ and $m[b_{k,\bm{\ell}}]$ are supported in a neighbourhood of the cone $\Gamma_2$. Furthermore, the $u_1(\xi)$ localisation means that $m[b_{k,\bm{\ell}}]$ is localised in a neighbourhood of the cone $\Gamma_1$. Consequently, one may apply a decoupling theorem for such cones (in particular, instances of Theorem~\ref{Frenet decoupling theorem}) to \textit{radially} decompose the multipliers. In the case of the $m[b_{k,\bm{\ell}}]$, we first decouple with respect to the cone $\Gamma_2$. After rescaling, the localised pieces can be treated in a similar manner to the multipliers from the $J=3$ case. In particular, we apply a second decoupling to each rescaled piece with respect to the cone $\Gamma_1$ to further decompose into smaller pieces. For both the $a_{k,\ell,\iota}$ and $b_{k,\bm{\ell}}$,  it transpires that each radially localised piece is automatically localised along the curve in the physical space, and this leads to favourable $M^{\infty}(\R^4)$ bounds: see Lemma~\ref{J=4 curve loc lem} and Lemma~\ref{J4 Linfty bounds lem} below.


\subsection{Fourier localisation and decoupling} \label{subsec:Fourier loc J=4}


The first step towards proving Proposition~\ref{J4 Lp proposition} is to further decompose the symbols $a_{k,\ell,\iota}$ and $b_{k,\bm{\ell}}$ in terms of $\theta_2(\xi)$ and $\theta_1(\xi)$ respectively. Fix $\zeta \in C^{\infty}(\R)$ with $\supp \zeta \subseteq [-1,1]$ such that $\sum_{l \in \Z} \zeta(\,\cdot\, - l) \equiv 1$. For $0 \leq \ell \leq \floor{k/4}$, $1 \leq \iota \leq 4$ and $\bm{\ell}=(\ell_1,\ell_2) \in \Lambda(k)$, write
\begin{equation*}
    a_{k, \ell,\iota}=\sum_{\mu\in \Z}a_{k, \ell, \iota}^{\mu} \qquad \text{ and } \qquad b_{k, \bm{\ell}}=\sum_{\nu\in \Z}b_{k, \bm{\ell}}^{\nu}
\end{equation*}
where
\begin{align}\label{eq:dec theta_2 lower triangle}
    a_{k,\ell,\iota}^{\mu}(\xi;s) &:= a_{k,\ell,\iota}(\xi;s) \zeta(2^{\ell}\theta_2(\xi) - \mu),\\
\label{eq:b freq dec theta}
    b_{k,\bm{\ell}}^{\nu}(\xi;s) &:= b_{k,\bm{\ell}}(\xi;s) \zeta(2^{(3\ell_1-\ell_2)/2}\theta_1(\xi) - \nu).
\end{align}
In the case of the $b_{k,\bm{\ell}}$, we also consider symbols formed by grouping the $b_{k,\bm{\ell}}^{\nu}$ into pieces at the larger scale $2^{-\ell_2}$. Given $\bm{\ell}=(\ell_1,\ell_2) \in \Lambda(k)$ we write $\Z = \bigcup_{\mu \in \Z} \mathfrak{N}_{\bm{\ell}}(\mu)$, where the sets $\mathfrak{N}_{\bm{\ell}}(\mu)$ are disjoint and satisfy
\begin{equation*}
    \mathfrak{N}_{\bm{\ell}}(\mu)\subseteq \{\nu \in \Z: |\nu-2^{3(\ell_1-\ell_2)/2} \mu| \leq 2^{3(\ell_1-\ell_2)/2} \}.
\end{equation*}
For each $\mu \in \Z$, we then define
\begin{equation*}
    b_{k,\bm{\ell}}^{*,\mu}:=\sum_{\nu \in \mathfrak{N}_{\bm{\ell}}(\mu)} b_{k,\bm{\ell}}^\nu
\end{equation*}
and note that $|\theta_1(\xi)-s_{\mu}|\lesssim 2^{-\ell_2}$ on $\supp b_{k,\bm{\ell}}^{*,\mu}$, where $s_{\mu}:=2^{-\ell_2}\mu$. Of course, by the definition of the sets $\mathfrak{N}_{\bm{\ell}}(\mu)$,
\begin{equation*}
    b_{k,\bm{\ell}}=\sum_{\mu \in \Z} b_{k,\bm{\ell}}^{*,\mu} = \sum_{\mu \in \Z} \sum_{\nu \in \mathfrak{N}_{\bm{\ell}}(\mu)} b_{k,\bm{\ell}}^\nu.
\end{equation*}


Given $0 < r \leq 1$ and $s \in I$, recall the definition of the $(2,r)$-\textit{Frenet boxes} $\pi_{2}(s;r)$ introduced in Definition~\ref{def Frenet box}:
\begin{equation}\label{J=4 symbol supp}
    \pi_2(s;r):= \big\{ \xi \in \hat{\R}^4: |\inn{\be_j(s)}{\xi}| \lesssim r^{4-j} \,\, \textrm{for $1 \leq j \leq 3$}, \quad 
    |\inn{\be_{4}(s)}{\xi}| \sim 1\big\}. 
\end{equation}
The symbols $a_{k,\ell,\iota}^\mu$ and $b_{k,\bm{\ell}}^{\ast, \mu}$ satisfy the following support properties.

\begin{lemma}\label{J=4 supp lem} With the above definitions,
\begin{enumerate}[a)]
    \item $\xisupp a_{k, \ell,\iota}^{\mu} \subseteq 2^k \cdot \pi_2(s_{\mu}; 2^{-\ell})$ for all $0 \leq \ell \leq \floor{k/4}$, $1 \leq \iota \leq 4$ and $\mu \in \Z$, where $s_{\mu}:=2^{-\ell }\mu$;\smallskip
    \item $\xisupp b_{k,\bm{\ell}}^{*,\mu} \subseteq 2^k \cdot \pi_2(s_{\mu}; 2^{-\ell_2})$ for all $\bm{\ell} = (\ell_1, \ell_2) \in \Lambda(k)$ and $\mu \in \Z$, where $s_{\mu} := 2^{-\ell_2}\mu$.
\end{enumerate}
\end{lemma}

It is convenient to set up a unified framework in order to treat parts a) and b) of Lemma~\ref{J=4 supp lem} simultaneously. Given $n$, $s \in \R$, let $\Xi_2(k, n; s)$ denote the set of all $\xi \in \xisupp a_k$ which lie in the domain of $\theta_2$ and satisfy
\begin{equation}\label{Xi_2 def}
    |\theta_2(\xi) - s| \lesssim 2^{-n}, \qquad |u_{1,2}(\xi)| \lesssim 2^{k-3n}, \qquad |u_2(\xi)| \lesssim 2^{k-2n}.
\end{equation}
Note in particular that:
\begin{enumerate}[a)]
    \item $\xisupp a_{k,\ell, \iota}^{\mu} \subseteq \Xi_2(k, \ell ;s_{\mu})$ for $1 \leq \iota \leq 4$.
    \item $\xisupp b_{k,\bm{\ell}}^{\nu} \subseteq \Xi_2(k, \ell_2 ;s_{\nu})$ and $\xisupp b_{k,\bm{\ell}}^{\nu} \subseteq \Xi_2(k, \ell_2 ;s_{\mu})$ for all $\nu \in \mathfrak{N}_{\bm{\ell}}(\mu)$.
 \end{enumerate}   
Indeed, for the respective parameter values, all the desired properties stated in \eqref{Xi_2 def} hold as an immediate consequence of the definition of the symbols, with the exception of the bounds $|\theta_2(\xi) - s_{\nu}| \lesssim 2^{-\ell_2}$ for $\xi \in \xisupp b_{k,\ell}^{\nu}$ and $|\theta_2(\xi) - s_{\mu}| \lesssim 2^{-\ell_2}$ for $\xi \in \xisupp b_{k,\ell}^{\nu}$ and $\nu \in \mathfrak{N}_{\bm{\ell}}(\mu)$. However, by Lemma~\ref{lemma:size of quantities}, it follows from the localisation of the symbol that
    \begin{equation*}
        |\theta_2(\xi) - s_{\nu}| \lesssim \big|u_2\big(\tfrac{\xi}{|\xi|}\big)\big|^{1/2} + |\theta_1(\xi) - s_{\nu}| \lesssim 2^{-\ell_2} \quad \textrm{for all $\xi \in \xisupp b_{k,\ell}^{\nu}$}, 
    \end{equation*}
    which further implies
    \begin{equation*}
        |\theta_2(\xi) - s_{\mu}| \lesssim |\theta_2(\xi) - s_{\nu}| + |s_\nu-s_\mu| \lesssim 2^{-\ell_2} \quad \textrm{for all $\xi \in \xisupp b_{k,\ell}^{\nu}$, $\,\,\nu \in \mathfrak{N}_{\bm{\ell}}(\mu)$,}
    \end{equation*}
    by the condition $|s_{\mu} - s_{\nu}| \lesssim 2^{-\ell_2}$ for $s_{\nu} := 2^{-(3\ell_1 - \ell_2)/2}\nu$. Thus, all the required bounds hold.
    
    Note that the support property in b) immediately implies that $\xisupp b_{k,\bm{\ell}}^{\ast,\mu} \subseteq \Xi_2(k,\ell_2;s_\mu)$.

\begin{proof}[Proof of Lemma~\ref{J=4 supp lem}] Let $n$, $s \in \R$. As a consequence of the preceding discussion, it suffices to show that 
\begin{equation*}
    \Xi_2(k, n; s) \subseteq 2^k \cdot \pi_2(s; 2^{-n}).
\end{equation*} 
Let $\xi \in \Xi_2(k, n;s)$ and observe that the localisation of $a_k$, the implicit definition of $\theta_2$ and the latter two conditions in \eqref{Xi_2 def} imply
\begin{equation*}
    |\inn{\gamma^{(i)}\circ \theta_2(\xi)}{\xi}| \lesssim 2^{k - (4 - i)n} \qquad \textrm{for $1 \leq i \leq 4$.} 
\end{equation*}
Since the Frenet vectors $\be_i \circ \theta_2(\xi)$ are obtained from the $\gamma^{(i)}\circ \theta_2(\xi)$ via the Gram--Schmidt process, 
\begin{equation*}
   |\inn{\be_i\circ \theta_2(\xi)}{\xi}| \lesssim 2^{k - (4 - i)n} \qquad \textrm{for $1 \leq i \leq 4$.}  
\end{equation*}
On the other hand, the first condition in \eqref{Xi_2 def}, together with \eqref{Frenet bound}, implies
\begin{equation*}
    |\inn{\be_{i}\circ \theta_2(\xi)}{\be_{j}(s)}| \lesssim |\theta_2(\xi)-s|^{|i-j|} \lesssim 2^{-(i-j)n}.
\end{equation*}
Writing $\xi$ with respect to the orthonormal basis $\big(\be_j \circ \theta_2(\xi)\big)_{j=1}^4$, it follows that
\begin{equation*}
    |\inn{\be_j(s)}{\xi}| \leq \sum_{i=1}^4 |\inn{\be_i\circ \theta_2(\xi)}{\xi}| |\inn{\be_i\circ\theta_2(\xi)}{\be_j (s_{\mu})}| \lesssim 2^{k - (4-j)n}.
\end{equation*}
Thus, $\xi$ satisfies all the required upper bounds arising from \eqref{J=4 symbol supp}. The remaining condition $|\inn{\be_4(s)}{\xi}| \gtrsim 2^k$ holds as an immediate consequence of the initial localisation of $a_k$.
\end{proof}

The argument used in the proof of Lemma~\ref{J=4 supp lem} can be applied to analyse the support properties of the individual $b_{k,\bm{\ell}}^{\nu}$, although in this case the geometric significance of the supporting set is only apparent after rescaling (see Lemma~\ref{J=4 resc b_nu supp lem} below). Given $0 < r_1, r_2 \leq 1$ and $s \in I$, define the set
\begin{equation}\label{J=4 b_nu sym supp}
     \pi_1(s;r_1, r_2)\!:=\! \big\{\xi \in \hat{\R}^4: |\inn{\be_j(s)}{\xi}| \lesssim r_1^{3-j} \,\, \textrm{for $j = 1$, $2$}, \,\,\, |\inn{\be_3(s)}{\xi}| \sim 1, \,\,
    |\inn{\be_4(s)}{\xi}| \lesssim r_2\big\}.
\end{equation}
 The multipliers $b_{k,\bm{\ell}}^\nu$ satisfy the following support property.
\begin{lemma}\label{J=4 b_nu supp lem} With the above definitions, 
\begin{equation*}
    \xisupp b_{k,\bm{\ell}}^{\nu} \subseteq 2^{k-\ell_2} \cdot \pi_1(s_{\nu}; 2^{-(3\ell_1 - \ell_2)/2}, 2^{\ell_2})
\end{equation*}
for all $\bm{\ell} = (\ell_1, \ell_2) \in \Lambda(k)$ and $\nu \in \Z$, where $s_{\nu} := 2^{-(3\ell_1-\ell_2)/2 } \nu$. 
\end{lemma}

As with the proof of Lemma~\ref{J=4 supp lem}, here and in Lemma~\ref{J4 Linfty bounds lem}, we will work with a more general setup. This abstraction is not particularly useful at this stage, but it will help to unify some of the later arguments. Given $\bm{n} = (n_1, n_2) \in (0,\infty)^2$ and $s \in \R$, let $\Xi_1(k, \bm{n}; s)$ denote the set of all $\xi \in \xisupp a_k$ which lie in the domain of $\theta_1$ and satisfy 
\begin{equation}\label{Xi_1 def}
    |\theta_1(\xi) - s| \lesssim 2^{-n_1}, \qquad |u_1(\xi)| \lesssim 2^{k- 2n_1 - n_2}, \qquad |u_{3,1}(\xi)| \sim 2^{k-n_2}. 
\end{equation}
Note in particular that:
\begin{enumerate}[a)]
    \item $\xisupp a_{k,\ell, \iota}^{\mu} \subseteq \Xi_1(k, \ell, \ell ; s_{\mu})$ for $\iota = 3$ or $\iota = 4$.
    \item $\xisupp b_{k,\bm{\ell}}^{\nu} \subseteq \Xi_1\big(k,\tfrac{3\ell_1 - \ell_2}{2}, \ell_2; s_{\nu}\big)$.
 \end{enumerate}   
 Indeed, the definition of the symbols implies $|u_2(\xi)| \sim 2^{k-2\ell}$, $|u_1(\xi)| \lesssim 2^{k-3\ell}$ and $|\theta_2(\xi) - s_{\mu}| \lesssim 2^{-\ell}$ for all $\xi \in \xisupp a_{k,\ell, \iota}^{\mu}$ when $\iota \in \{3,4\}$. Consequently, by Lemma~\ref{lemma:size of quantities}, it follows that
    \begin{equation*}
        |\theta_1(\xi) - s_{\mu}| \lesssim \big|u_2\big(\tfrac{\xi}{|\xi|}\big)\big|^{1/2} + |\theta_2(\xi) - s_{\mu}| \lesssim 2^{-\ell}, \qquad \big|u_{3,1}\big(\tfrac{\xi}{|\xi|}\big)\big| \sim \big|u_2\big(\tfrac{\xi}{|\xi|}\big)\big|^{1/2} \sim 2^{\ell}
    \end{equation*}
    for all $\xi \in \xisupp a_{k,\ell, \iota}^{\mu}$, which covers the required bounds for a). Turning to b), all the desired properties hold as an immediate consequence of the definition of the symbols, with the exception of the bound $|u_{3,1}(\xi)| \sim 2^{k - \ell_2}$. However, as in a), the function $u_{3,1}$ can be estimated via Lemma~\ref{lemma:size of quantities} using the $u_2$ localisation.

\begin{proof}[Proof of Lemma~\ref{J=4 b_nu supp lem}] Let $\bm{n} = (n_1, n_2) \in (0,\infty)^2$ and $s \in I_0$.  As a consequence of the preceding discussion, it suffices to show that \begin{equation*}
    \Xi_1(k,\bm{n}; s) \subseteq 2^{k - n_2} \cdot \pi_1(s;2^{-n_1}, 2^{-n_2}).
\end{equation*}
The argument in fact depends on the implicit constants in \eqref{Xi_1 def} satisfying certain size relations, but we shall ignore this minor technicality. In the case in question (namely, on the support of $b_{k,\ell}^{\nu}$), the required size relations follow provided $\rho$ is chosen sufficiently small.

Let $\xi \in \Xi_1(k,\bm{n};s)$ and observe that the localisation of $a_k$, the implicit definition of $\theta_1$ and the latter two conditions in \eqref{Xi_1 def} imply
\begin{equation*}
    |\inn{\gamma^{(i)}\circ \theta_1(\xi)}{\xi}| \lesssim 2^{k - (3 - i)n_1 - n_2} \quad \textrm{for $i = 1$, $2$}, \quad |\inn{\gamma^{(3)}\circ \theta_1(\xi)}{\xi}| \sim 2^{k - n_2}, \quad |\inn{\gamma^{(4)}\circ \theta_1(\xi)}{\xi}| \sim 2^k.
\end{equation*}
Since the Frenet vectors $\be_i \circ \theta_2(\xi)$ are obtained from the $\gamma^{(i)}\circ \theta_2(\xi)$ via the Gram--Schmidt process, the matrix corresponding to change of basis from $\big(\be_i \circ \theta_1(\xi)\big)_{i=1}^4$ to $\big(\gamma^{(i)} \circ \theta_1(\xi)\big)_{i=1}^4$ is lower triangular. Furthermore, the initial localisations imply that this matrix is an $O(\delta)$ perturbation of the identity. Consequently, provided $\delta >0$ is chosen sufficiently small, 
\begin{gather*}
   |\inn{\be_i\circ \theta_1(\xi)}{\xi}| \lesssim 2^{k - (3 - i)n_1 - n_2} \quad \textrm{for $i = 1$, $2$}, \quad
   |\inn{\be_3\circ \theta_1(\xi)}{\xi}| \sim 2^{k - n_2}, \quad |\inn{\be_4\circ \theta_1(\xi)}{\xi}| \sim 2^{k}.
\end{gather*}
On the other hand, the first condition in \eqref{Xi_1 def} together with \eqref{Frenet bound} imply
\begin{equation*}
    |\inn{\be_{i}\circ \theta_1(\xi)}{\be_{j}(s)}| \lesssim |s - \theta_1(\xi)|^{|i-j|} \lesssim 2^{-(i-j)n_1}.
\end{equation*}
Writing $\xi$ with respect to the orthonormal basis $\big(\be_j \circ \theta_1(\xi)\big)_{j=1}^4$, it follows that
\begin{equation*}
    |\inn{\xi}{\be_j(s)}| \leq \sum_{i=1}^4 |\inn{\be_i\circ \theta_1(\xi)}{\xi}| |\inn{\be_i\circ\theta_1(\xi)}{\be_j (s)}| \lesssim 2^{k - ((3-j)n_1 + n_2)\vee 0}.
\end{equation*}
Thus, $\xi$ satisfies all the required upper bounds arising from \eqref{J=4 b_nu sym supp}. The above argument can easily be adapted to give the required lower bounds, provided the implied constant in the the hypothesis $|u_{3,1}(\xi)| \sim 2^{k-n_2}$ is large compared to that in the hypothesis $|\theta_1(\xi) - s| \lesssim 2^{-n_1}$. 
\end{proof}

Fix some $\bm{\ell} = (\ell_1, \ell_2) \in \Lambda(k)$ and $\mu \in \Z$ with $s_{\mu} := 2^{-\ell_2} \mu \in [-1,1]$. To simplify notation, let $\sigma := s_{\mu}$, $\lambda := 2^{-\ell_2}$ and let $\tilde{\gamma} := \gamma_{\sigma, \lambda}$ denote the rescaled curve, as defined in Definition~\ref{rescaled curve def}, so that
\begin{equation}\label{J4 gamma resc}
    \tilde{\gamma}(s) := \big([\gamma]_{\sigma, \lambda}\big)^{-1} \big( \gamma(\sigma + \lambda s) - \gamma(\sigma)\big). 
\end{equation}
Given a symbol $b \in C^{\infty}_c(\hat{\R}^4 \times I_0)$, let $\tilde{b}$ be the rescaled symbol defined by the relation
\begin{equation}\label{J4 b resc}
    \tilde{b}(\tilde{\xi};\tilde{s}) = b(\xi;s) \qquad \textrm{for} \quad \tilde{\xi} := \big([\gamma]_{\sigma, \lambda}\big)^{\top} \xi \quad \textrm{and} \quad \tilde{s} := \lambda^{-1}(s - \sigma). 
\end{equation}
Given $f \in \mathscr{S}(\R^4)$, it follows by a simple changes of the variables that 
\begin{equation}\label{J4 m resc}
    m[b](D)f(x) = \lambda \cdot \tilde{m}[\tilde{b}](D)\tilde{f}(\tilde{x}) 
\end{equation}
where:
\begin{itemize}
    \item The multiplier $\tilde{m}[\tilde{b}]$ is defined in the same manner as $m[\tilde{b}]$ but with the curve $\gamma$ replaced with $\tilde{\gamma}$ and the cut-off $\chi_{\circ}$ replaced with $\chi_{\circ}(\sigma + \lambda \,\cdot\,)$;
    \item $\tilde{f} := f \circ [\gamma]_{\sigma, \lambda}$;
    \item $\tilde{x} := \big([\gamma]_{\sigma, \lambda}\big)^{-1} \big( x - \gamma(\sigma)\big)$. 
\end{itemize}

Let $(\tilde{\be}_j)_{j=1}^4$ denote the Frenet frame defined with respect to $\tilde{\gamma}$. Given $0 < r \leq 1$ and $s \in I$, recall the definition of the $(1,r)$-\textit{Frenet boxes} (with respect to $(\tilde{\be}_j)_{j=1}^4$) introduced in Definition~\ref{def Frenet box}: 
\begin{equation*}
     \tilde{\pi}_1(s;r):= \big\{\xi \in \hat{\R}^4: |\inn{\tilde{\be}_j(s)}{\xi}| \lesssim r^{3-j} \quad \textrm{for $j = 1$, $2$,}\quad |\inn{\tilde{\be}_{3}(s)}{\xi}| \sim 1,  \quad |\inn{\tilde{\be}_{4}(s)}{\xi}| \lesssim 1\big\}.
\end{equation*}
 Note that all these definitions depend of the choice of $\mu$ and $\bm{\ell}$, but it is typographically convenient to suppress this dependence. 

The rescaled symbols $\tilde{b}_{k,\bm{\ell}}^\nu$ satisfy the following support properties.

\begin{lemma}\label{J=4 resc b_nu supp lem} With the above definitions,
\begin{equation*}
    \xisupp \tilde{b}_{k,\bm{\ell}}^{\nu} \subseteq 2^{k-4\ell_2} \cdot \tilde{\pi}_1(\tilde{s}_{\nu}; 2^{-3(\ell_1 - \ell_2)/2})
\end{equation*}
for all $\bm{\ell} = (\ell_1, \ell_2) \in \Lambda(k)$ and $\nu \in \mathfrak{N}_{\bm{\ell}}(\mu)$, where $\tilde{s}_{\nu} := 2^{\ell_2} (s_{\nu} - s_{\mu})$ for $s_{\nu} := 2^{-(3\ell_1 - \ell_2)/2}\nu$. 
\end{lemma}

\begin{proof} For $\tilde{\xi} \in \xisupp \tilde{b}_{k,\bm{\ell}}^{\nu}$, it follows from Lemma~\ref{J=4 b_nu supp lem} and the definition of the rescaling in \eqref{J4 b resc} that $\xi := \big([\gamma]_{\sigma, \lambda}\big)^{-\top} \tilde{\xi}$ satisfies 
\begin{equation*}
     |\inn{\be_j(s_{\nu})}{\xi}| \lesssim 2^{k-(3-j)(3 \ell_1 - \ell_2)/2-\ell_2} \quad \textrm{for $j = 1$, $2$}, \quad |\inn{\be_3(s_{\nu})}{\xi}| \sim 2^{k-\ell_2}, \quad
    |\inn{\be_4(s_{\nu})}{\xi}| \sim 2^k.
\end{equation*}
Since the matrix corresponding to the change of basis from $\big(\be_j(s_{\nu})\big)_{j=1}^4$ to $\big(\gamma^{(j)}(s_{\nu})\big)_{j=1}^4$ is lower triangular and an $O(\delta_0)$ perturbation of the identity, provided $\delta_0$ is sufficiently small,
\begin{equation*}
     |\inn{\gamma^{(j)}(s_{\nu})}{\xi}| \lesssim 2^{k-(3-j)(3 \ell_1 - \ell_2)/2-\ell_2} \quad \textrm{for $j = 1$, $2$}, \quad |\inn{\gamma^{(3)}(s_{\nu})}{\xi}| \sim 2^{k-\ell_2}, \quad
    |\inn{\gamma^{(4)}(s_{\nu})}{\xi}| \sim 2^k.
\end{equation*}
On the other hand, recalling that $\lambda := 2^{-\ell_2}$, it follows from the definition of $\tilde{\gamma}$ from \eqref{J4 gamma resc} that
\begin{equation*}
   \inn{\tilde{\gamma}^{(j)}(\tilde{s}_{\nu})}{\tilde{\xi}\,} = 2^{-j\ell_2}\inn{\gamma^{(j)}(s_{\nu})}{\xi}  \qquad \textrm{for $j \geq 1$}.
\end{equation*}
Combining the above observations,
\begin{gather*}
     |\inn{\tilde{\gamma}^{(j)}(\tilde{s}_{\nu})}{\tilde{\xi}\,}| \lesssim 2^{k-(3-j)(3 \ell_1 - \ell_2)/2-(j+1)\ell_2} \quad \textrm{for $j = 1$, $2$}, \\ |\inn{\tilde{\gamma}^{(3)}(\tilde{s}_{\nu})}{\tilde{\xi}\,}| \sim 2^{k-4\ell_2}, \quad
    |\inn{\tilde{\gamma}^{(4)}(\tilde{s}_{\nu})}{\tilde{\xi}\,}| \sim 2^{k-4\ell_2}.
\end{gather*}
Provided $\delta_0$ is sufficiently small, the desired result now follows since the matrix corresponding to the change of basis from $\big(\tilde{\be}_i(\tilde{s}_{\nu})\big)_{i=1}^4$ to $\big(\tilde{\gamma}^{(i)}(\tilde{s}_{\nu})\big)_{i=1}^4$ is also lower triangular and an $O(\delta_0)$ perturbation of the identity. 
\end{proof}


 In view of the support conditions from Lemma~\ref{J=4 supp lem} and Lemma~\ref{J=4 resc b_nu supp lem}, the multipliers can be effectively decoupled using Theorem~\ref{Frenet decoupling theorem}.

\begin{proposition}\label{J4 dec prop}
For all $2\le p\le 12$ and all $\varepsilon>0$, the following inequalities hold:
\begin{enumerate}[a)]
    \item For all $0\le \ell \leq  \floor{k/4}$, $1 \leq \iota \leq 4$,
\begin{equation*}
     \Big \|\sum_{\mu \in \Z} m[a^{ \mu}_{k,\ell, \iota}](D)f\Big\|_{L^p(\R^4)}  \lesim_{\varepsilon} 2^{\ell(1/2-1/p) + \varepsilon\ell} \Big(\sum_{\mu\in \Z} \| m[a^{ \mu}_{k,\ell,\iota}](D)f\|_{L^p(\R^4)}^p\Big)^{1/p}.
\end{equation*}
\item For all $\bm{\ell} = (\ell_1, \ell_2) \in \Lambda(k)$,
\begin{equation*}
    \Big\| \sum_{\mu \in \Z} m[ b_{k, \bm{\ell}}^{*,\mu}](D) f \Big\|_{L^p(\R^4)}  \lesssim_{\varepsilon} 2^{\ell_2(1/2-1/p) + \varepsilon\ell_2} \Big(\sum_{\mu \in \Z} \|  m[ b_{k,\bm{\ell}}^{*,\mu}](D) f\|_{L^p(\R^4)}^p\Big)^{1/p}.
\end{equation*}
\end{enumerate}
\end{proposition}

\begin{proof}
In view of the support conditions from Lemma~\ref{J=4 supp lem}, after a simple rescaling, the desired result follows from Theorem~\ref{Frenet decoupling theorem} with $d - 1 = 2$, $n=4$ and $r= 2^{-\ell}$, $2^{-\ell_2}$ for parts a) and b), respectively. 
\end{proof}

\begin{proposition}\label{J=4 b_nu dec prop}
For all $\bm{\ell} = (\ell_1,\ell_2) \in \Lambda(k)$, $6 \leq p \leq 12$ and $\varepsilon>0$, 
\begin{equation*}
    \Big\|\sum_{\nu \in \Z} m[b_{k, \bm{\ell}}^{\nu}](D)f\Big\|_{L^p(\R^4)} \lesim_{\varepsilon}2^{\ell_2(1/2-1/p + \varepsilon)} 2^{3(\ell_1-\ell_2)(1-4/p + \varepsilon)/2} \Big(\sum_{\nu\in \Z} \| m[b_{k, \bm{\ell}}^{\nu}](D)f \|_{L^p(\R^4)}^p \Big)^{1/p}.
\end{equation*}
\end{proposition}

\begin{proof} It suffices to show that, under the hypotheses of the proposition, for all $\mu \in \Z$ one has
\begin{equation}\label{J=4 b_nu dec 1}
    \Big\| \sum_{\nu \in \mathfrak{N}_{\bm{\ell}} (\mu)} m[b_{k,\bm{\ell}}^{\nu}] (D)f  \Big\|_{L^p(\R^4)} \\
    \lesssim_{\varepsilon} 2^{3(\ell_1-\ell_2)(1-4/p+\varepsilon)/2} \Big( \sum_{\nu \in \mathfrak{N}_{\bm{\ell}} (\mu)}   \| m[b_{k,\bm{\ell}}^{\nu}] (D)f \|_{L^p(\R^4)}^p \Big)^{1/p}.
\end{equation}
Indeed, one may then combine the above inequality with Proposition~\ref{J4 dec prop} b) to obtain the desired decoupling result. However, by applying a linear change of variables, \eqref{J=4 b_nu dec 1} is equivalent to the same inequality but with each $m[b_{k,\bm{\ell}}^\nu]$ replaced with the rescaled multiplier $\tilde{m}[\tilde{b}_{k,\bm{\ell}}^\nu]$ as defined in \eqref{J4 m resc}. In view of the support conditions from Lemma~\ref{J=4 resc b_nu supp lem}, after a simple rescaling, the desired result follows from Theorem~\ref{Frenet decoupling theorem} with $d-1 = 1$, $n=4$ and $r= 2^{-3(\ell_1 - \ell_2)/2}$.
\end{proof}


\subsection{Localisation along the curve} The localisation in $\theta_2(\xi)$ and $\theta_1(\xi)$ introduced in the previous subsection induces a corresponding localisation along the curve in the physical space. In particular, the main contribution to $m[a_{k,\ell,\iota}^{\mu}]$ and $m[b_{k,\bm{\ell}}^{\nu}]$ arises from the portion of the curve defined over the interval $|s - s_{\mu}| \leq 2^{-\ell}$ and $|s-s_{\nu}| \leq 2^{-(3\ell_1-\ell_2)/2}$, respectively. This is made precise by Lemma~\ref{J=4 curve loc lem} below.

For each $\mu, \nu \in \Z$, let $s_\mu := 2^{-\ell}\mu$ and $s_\nu := 2^{-(3\ell_1-\ell_2)/2}\nu$. Given $\varepsilon>0$ and for the fine tuning parameter $\rho$ as introduced in \S\ref{J=4 decomp sec}, define
\begin{align}\label{eq:dec lower triangle along curve}
    a_{k,\ell,\iota}^{\mu, (\varepsilon)}(\xi;s) & := a_{k,\ell,\iota}^\mu(\xi) \eta(\rho 2^{\ell(1-\varepsilon)}(s - s_\mu)),
\\
\label{eq:b dec along curve}
b_{k,\bm{\ell}}^{\nu, (\varepsilon)}(\xi;s) & := b_{k,\bm{\ell}}^\nu(\xi)  \eta(\rho 2^{(1-\varepsilon)(3\ell_1-\ell_2)/2}(s - s_\nu)).
\end{align}
The key contribution to the multipliers comes from the symbols $a_{k,\ell,\iota}^{\mu, (\varepsilon)}$ and $b_{k,\bm{\ell}}^{\nu, (\varepsilon)}$ respectively. 

\begin{lemma}\label{J=4 curve loc lem} Let $2 \le p < \infty$ and $\varepsilon > 0$.
\begin{enumerate}[a)]
    \item For all $0 \leq  \ell \leq  \floor{k/4}$, $\mu \in \Z$ and $1 \leq \iota \leq 4$,
\begin{equation*}
    \|m[a_{k,\ell,\iota}^{\mu}-a_{k,\ell,\iota}^{\mu, (\varepsilon)}]\|_{M^p(\R^4)}\lesim_{N, \varepsilon, p} 2^{-kN} \qquad \textrm{for all $N\in \N$.}
\end{equation*}
\item For all $\bm{\ell}=(\ell_1,\ell_2) \in \Lambda(k)$ and $\nu \in \Z$,
\begin{equation*}
    \|m[b_{k,\bm{\ell}}^{\nu}-b_{k,\bm{\ell}}^{\nu, (\varepsilon)}]\|_{M^p(\R^4)}\lesim_{N, \varepsilon, p} 2^{-kN} \qquad \textrm{ for all $N\in \N$.}
\end{equation*}
\end{enumerate}
\end{lemma}

\begin{proof} In both part a) and b) it is clear that the multipliers satisfy a trivial $L^{\infty}$-estimate with operator norm $O(2^{Ck})$ for some absolute constant $C \geq 1$. Thus, by interpolation, it suffices to prove the rapid decay for $p = 2$ only. This amounts to showing that, under the hypotheses of the lemma, 
\begin{equation}
 \label{J=4 curve loc 1}
     \|m[a_{k,\ell,\iota}^{\mu}-a_{k,\ell,\iota}^{\mu, (\varepsilon)}]\|_{L^{\infty}(\hat{\R}^4)} \lesssim_{N,\varepsilon} 2^{-kN} \quad \textrm{and} \quad 
     \|m[b_{k,\bm{\ell}}^{\nu}-b_{k,\bm{\ell}}^{\nu, (\varepsilon)}]\|_{L^{\infty}(\hat{\R}^4)} \lesim_{N,\varepsilon} 2^{-kN} \quad \textrm{for all $N \in \N$.}   
\end{equation}
This is achieved via a simple (non)-stationary phase analysis.\medskip

\noindent a) Here the localisation of the $a_{k,\ell, \iota}$ symbols ensures that
\begin{equation}\label{J=4 curve loc a 1}
     |u_{1,2}(\xi)| \lesssim 2^{k-3\ell}, \quad |u_2(\xi)| \lesssim \rho 2^{k-2\ell}  \qquad \textrm{for all $(\xi;s) \in \supp (a_{k,\ell,\iota}^{\mu}-a_{k,\ell,\iota}^{\mu, (\varepsilon)})$.}
\end{equation}
On the other hand, provided $\rho$ is sufficiently small, the additional localisation in \eqref{eq:dec theta_2 lower triangle} and \eqref{eq:dec lower triangle along curve} implies
\begin{equation}\label{J=4 curve loc a 2}
|s-\theta_2(\xi)| \gtrsim \rho^{-1} 2^{-\ell(1 - \varepsilon)} \quad \text{for all } (\xi;s) \in \supp (a_{k,\ell,\iota}^{\mu}-a_{k,\ell,\iota}^{\mu, (\varepsilon)}).
\end{equation}

Fix $\xi \in \xisupp (a_{k,\ell,\iota}^{\mu}-a_{k,\ell,\iota}^{\mu, (\varepsilon)})$ and consider the oscillatory integral $m[a_{k,\ell,\iota}^{\mu}-a_{k,\ell,\iota}^{\mu, (\varepsilon)}](\xi)$, which has phase $s \mapsto \inn{\gamma(s)}{\xi}$. Taylor expansion around $\theta_2(\xi)$ yields
\begin{align}\label{J=4 curve loc a 3}
        \inn{\gamma'(s)}{\xi} &=u_{1,2}(\xi)+ \big( u_2(\xi) + \om_1(\xi;s)\cdot (s-\theta_{2}(\xi))^2 \big) \cdot (s-\theta_{2}(\xi)), \\ \label{J=4 curve loc a 4}
        \inn{\gamma''(s)}{\xi} &=u_{2}(\xi)+  \om_2(\xi;s)\cdot (s-\theta_{2}(\xi))^2,  \\ \label{J=4 curve loc a 5}
        \inn{\gamma^{(3)}(s)}{\xi} &=\om_3(\xi;s)\cdot (s-\theta_{2}(\xi)),
\end{align}
where the $\omega_i$ arise from the remainder terms and satisfy $|\om_i(\xi;s)| \sim 2^k$. Provided $\rho$ is sufficiently small, \eqref{J=4 curve loc a 1} and \eqref{J=4 curve loc a 2} imply that the $\omega_1(\xi;s) \cdot (s- \theta_2(\xi))^3$ term dominates the right-hand side of \eqref{J=4 curve loc a 3} and therefore
\begin{equation}\label{J=4 curve loc a 6}
    |\inn{\gamma'(s)}{\xi}| \gtrsim  2^{k} |s-\theta_2(\xi)|^3 \quad \textrm{for all $(\xi; s) \in \supp (a_{k,\ell,\iota}^{\mu}-a_{k,\ell, \iota}^{\mu, (\varepsilon)})$.}
\end{equation}
Furthermore, by \eqref{J=4 curve loc a 1} and \eqref{J=4 curve loc a 2}, the term $\omega_2(\xi;s) \cdot (s-\theta_2(\xi))^2$ dominates in \eqref{J=4 curve loc a 4}. This, \eqref{J=4 curve loc a 5}, \eqref{J=4 curve loc a 6} and the localisation \eqref{J=4 curve loc a 2} immediately imply
\begin{align*}
    |\inn{\gamma''(s)}{\xi}| &  \lesssim   2^{-k+4\ell(1-\varepsilon)}|\inn{\gamma'(s)}{\xi}|^2, \\
    |\inn{\gamma^{(3)}(s)}{\xi}| &  \lesssim   2^{-(k-4\ell(1-\varepsilon))2}|\inn{\gamma'(s)}{\xi}|^3, \\
    |\inn{\gamma^{(j)}(s)}{\xi}| & \lesssim 2^{k} \lesssim_j  2^{-(k-4\ell(1-\varepsilon))(j-1)} |\inn{\gamma'(s)}{\xi}|^j
     \qquad \text{for all $j \geq 4$}
\end{align*}
for all $(\xi;s) \in \supp (a_{k,\ell,\iota}^{\mu}-a_{k,\ell,\iota}^{\mu,(\varepsilon)})$.

On the other hand, by the definition of the symbols, \eqref{J=4 curve loc a 6} and the localisation \eqref{J=4 curve loc a 2},
  \begin{equation*}
 |\partial_s^N (a_{k,\ell,\iota}^{\mu}-a_{k,\ell,\iota}^{\mu,(\varepsilon)})(\xi;s)| \lesssim_N 2^{\ell N} \lesssim 2^{-(k-4\ell)N - 3\varepsilon \ell N}|\inn{\gamma'(s)}{\xi}|^N \qquad \textrm{for all $N \in \N$.}
 \end{equation*}
Thus, by repeated integration-by-parts (via Lemma \ref{non-stationary lem}, with $R=2^{k-4\ell+ 3\varepsilon \ell } \geq 1$),
\begin{equation*}
    |m[a_{k,\ell,\iota}^{\mu}-a_{k,\ell,\iota}^{\mu, (\varepsilon)}](\xi)| \lesssim_N  2^{-(k-4\ell) N} 2^{-3\varepsilon \ell N} \quad \text{ for all }  N \in \N.
\end{equation*}
Since $0 \leq \ell \leq  \floor{k/4} \leq  k/4$, the first bound in \eqref{J=4 curve loc 1} follows.
\medskip

\noindent b)  Here the localisation of the $b_{k,\bm{\ell}}$ symbols ensures that
\begin{equation}\label{J=4 curve loc bb 1}
    |u_1(\xi)| \lesssim \rho^4 2^{k-3\ell_1},\qquad |u_2(\xi)|\sim \rho 2^{k-2\ell_2}
    ,\qquad |s - \theta_1(\xi)| \lesssim \rho 2^{-\ell_2} 
\end{equation}
hold for all $(\xi;s) \in \supp (b_{k,\bm{\ell}}^{\nu}-b_{k,\bm{\ell}}^{\nu, (\varepsilon)})$. Furthermore, by Lemma~\ref{J4 lower bound lemma},
\begin{equation}\label{J=4 curve loc bb 2}
    |\inn{\gamma^{(3)}(s)}{\xi}| \sim \rho^{1/2} 2^{k-\ell_2} \qquad \textrm{for all $(\xi;s) \in \supp b_{k,\bm{\ell}}$,}
\end{equation}
whilst, provided $\rho$ is sufficiently small, the additional localisation in \eqref{eq:b freq dec theta} and \eqref{eq:b dec along curve} implies
\begin{equation}\label{J=4 curve loc bb 3}
    |s-\theta_1(\xi)|\gtrsim \rho^{-1} 2^{-(1-\varepsilon)(3\ell_1-\ell_2)/2}  \quad \text{for all } (\xi;s) \in \supp (b_{k,\bm{\ell}}^{\nu}-b_{k,\bm{\ell}}^{\nu, (\varepsilon)}).
\end{equation}

Fix $\xi \in \xisupp (b_{k,\bm{\ell}}^{\nu}-b_{k,\bm{\ell}}^{\nu, (\varepsilon)})$ and consider the oscillatory integral $m[b_{k,\bm{\ell}}^{\nu}-b_{k,\bm{\ell}}^{\nu, (\varepsilon)}](\xi)$, which has phase $s \mapsto \inn{\gamma(s)}{\xi}$. Taylor expansion around $\theta_1(\xi)$ yields
\begin{align}\label{J=4 curve loc b 3}
    \inn{\gamma'(s)}{\xi} &=u_{1}(\xi)
        + \omega_1(\xi; s)\cdot (s-\theta_1(\xi))^2, \\ \label{J=4 curve loc b 4}
    \inn{\gamma''(s)}{\xi} &=\omega_2(\xi; s)\cdot (s-\theta_1(\xi)),
\end{align}
where the $\omega_i$ arise from the remainder terms and satisfy $|\omega_i(\xi;s)| \sim \rho^{1/2} 2^{k-\ell_2}$ by \eqref{J=4 curve loc bb 2}. Provided $\rho>0$ is sufficiently small, \eqref{J=4 curve loc bb 1} and \eqref{J=4 curve loc bb 3} imply that the second term dominates the right-hand side of \eqref{J=4 curve loc b 3} and therefore
\begin{equation}\label{J=4 curve loc b 5}
    |\inn{\gamma'(s)}{\xi}|\gtrsim \rho^{1/2} 2^{k-\ell_2}|s-\theta_1(\xi)|^2 \quad \text{ for all } \, (\xi; s) \in \supp (b_{k,\bm{\ell}}^{\nu}-b_{k,\bm{\ell}}^{\nu, (\varepsilon)}).
\end{equation}
Furthermore, \eqref{J=4 curve loc b 4}, \eqref{J=4 curve loc bb 2}, \eqref{J=4 curve loc b 5} and the localisation \eqref{J=4 curve loc bb 3} imply
\begin{align*}
    |\inn{\gamma''(s)}{\xi}| &  \lesssim  2^{-k+\ell_2+3(1-\varepsilon)(3\ell_1-\ell_2)/2}|\inn{\gamma'(s)}{\xi}|^2, \\
    |\inn{\gamma^{(3)}(s)}{\xi}| &  \lesssim 2^{k-\ell_2} \lesssim 2^{-2(k-\ell_2-3(1-\varepsilon)(3\ell_1-\ell_2)/2) }|\inn{\gamma'(s)}{\xi}|^3, \\
    |\inn{\gamma^{(j)}(s)}{\xi}| & \lesssim 2^{k} \lesssim_j 2^{-(k-\ell_2-3(1-\varepsilon)(3\ell_1-\ell_2)/2) (j-1)} |\inn{\gamma'(s)}{\xi}|^j
     \qquad \text{for all $j \geq 4$}
\end{align*}
for all $(\xi;s) \in \supp (b_{k,\bm{\ell}}^{\nu}-b_{k,\bm{\ell}}^{\nu,(\varepsilon)})$.

 On the other hand, by the definition of the symbols, \eqref{J=4 curve loc b 5} and the localisation \eqref{J=4 curve loc bb 3},
  \begin{align*}
 |\partial_s^N (b_{k,\bm{\ell}}^{\nu}-b_{k,\bm{\ell}}^{\nu, (\varepsilon)})(\xi;s)| & \lesssim_N \max\big\{ \rho^{-N} 2^{\ell_2 N}\,,\, 2^{(1-\varepsilon)(3\ell_1-\ell_2) N/2}\big\} \\
 &\lesssim_{N, \rho} 2^{-(k-\ell_2-3(1-\varepsilon)(3\ell_1-\ell_2)/2)N}|\inn{\gamma'(s)}{\xi}|^N \qquad \textrm{for all $N \in \N$}
 \end{align*}
 and  all  $(\xi; s) \in \supp (b_{k,\bm{\ell}}^{\nu}-b_{k,\bm{\ell}}^{\nu, (\varepsilon)})$, using that $0 \leq \ell_2 \leq \ell_1$ for $\bm{\ell} \in \Lambda(k)$. Thus, by repeated integration-by-parts (via Lemma \ref{non-stationary lem} with $R=2^{k-\ell_2-3(1-\varepsilon)(3\ell_1-\ell_2)/2}\geq 1$),
\begin{equation*}
        |m[ b_{k,\bm{\ell}}^{\nu}-b_{k,\bm{\ell}}^{\nu, (\varepsilon)}](\xi;s)|  \lesim_{N, \rho} 2^{-(k-\ell_2-3(3\ell_1-\ell_2)/2)N - 3\varepsilon (3\ell_1-\ell_2) N/2}
    \quad \text{ for all }  N \in \N.
\end{equation*}
Since $\ell_2 \leq \ell_1 \leq (2k+\ell_2)/9$ and $0 \leq \ell_2 < k/4$ for $\bm{\ell} \in \Lambda(k)$, the second bound in \eqref{J=4 curve loc 1} follows.
\end{proof}


\subsection{Estimating the localised pieces}

Each piece of the multipliers $m[a_{k,\ell,\iota}^{\mu, (\varepsilon)}]$ and $m[b_{k,\bm{\ell}}^{\nu, (\varepsilon)}]$ arising from the preceding decomposition satisfies favourable $L^2$ and $L^{\infty}$ bounds.


\begin{lemma}\label{L2 bounds J=4 lem}
\begin{enumerate}[a)]
\item For $0 \leq \ell \leq  \floor{k/4}$, $\mu \in \Z$, $1 \leq \iota\leq 4$ and  $\varepsilon>0$, we have
\begin{equation*}
    \|m[a_{k,\ell,\iota}^{\mu, (\varepsilon)}]\|_{M^2(\R^4)} \lesim  2^{-k/2+\ell}.
\end{equation*}
\item 
For $\bm{\ell}=(\ell_1,\ell_2) \in \Lambda(k)$, $\nu \in \Z$ and  $\varepsilon>0$, we have
\begin{equation*}
    \| m[b_{k,\bm{\ell}}^{\nu, (\varepsilon)}] \|_{M^2(\R^4)} \lesim 2^{-k/2+(3\ell_1+\ell_2)/4}.
\end{equation*}
\end{enumerate}
\end{lemma}

\begin{proof} 

a) If $\ell = \floor{k/4}$, then the desired bounds follow from Plancherel's theorem and the van der Corput lemma with fourth order derivatives. For the remaining cases, it suffices to show that
\begin{equation}\label{J=2 L2 1}
    |\inn{\gamma'(s)}{\xi}| + 2^{-\ell}|\inn{\gamma''(s)}{\xi}| \gtrsim 2^{k-3\ell} \qquad \textrm{for all $(\xi;s) \in \supp a_{k,\ell,\iota}^{\mu, (\varepsilon)}$.}
\end{equation}
We treat each class of symbol, as index by the parameter $\iota$, individually.\medskip

\noindent\underline{$\iota = 1$}. Here the localisation of the symbol ensures the key properties
\begin{equation}\label{J=4 L2 1 1}
|u_{1,2}(\xi)| \sim 2^{k-3\ell}, \qquad |u_2(\xi)| \lesssim \rho 2^{k-2\ell} \qquad \text{ for all $(\xi;s) \in \supp a_{k,\ell,1}^{\mu,(\varepsilon)}$.}
\end{equation}
By Taylor expansion around $\theta_2(\xi)$, one has
\begin{align}\label{J=4 L2 1 2}
    \inn{\gamma'(s)}{\xi} & = u_{1,2}(\xi) + u_2(\xi) \cdot (s-\theta_2(\xi))  + \omega_1(\xi;s) \cdot (s-\theta_2(\xi))^3,  
\\
\label{J=4 L2 1 3}
    \inn{\gamma''(s)}{\xi} & =  u_2 (\xi) + \omega_2(\xi;s) \cdot (s-\theta_2(\xi))^2
\end{align}
where the functions $\omega_i$ arise from the remainder terms and satisfy $|\omega_i(\xi;s)| \sim 2^k$ for $i = 1$, $2$. The argument splits into two cases: \medskip 

\noindent\textbf{Case 1:}  $|s-\theta_2(\xi)| \leq \rho^{1/4} 2^{-\ell}$. Provided $\rho >0$ is chosen sufficiently small, \eqref{J=4 L2 1 1} implies that the $u_{1,2}(\xi)$ term dominates the right-hand side of \eqref{J=4 L2 1 2} and therefore $|\inn{\gamma'(s)}{\xi}| \gtrsim 2^{k-3\ell}$.\medskip

\noindent\textbf{Case 2:}  $|s-\theta_2(\xi)| \geq \rho^{1/4} 2^{-\ell}$. Again provided $\rho > 0$ is sufficiently small,  \eqref{J=4 L2 1 1} implies that the second term dominates the right-hand side of \eqref{J=4 L2 1 3} and therefore $|\inn{\gamma''(s)}{\xi}| \gtrsim \rho^{1/2} 2^{k-2\ell}$. \medskip

\noindent Thus, in either case the desired bound \eqref{J=2 L2 1} holds.\medskip

\noindent\underline{$\iota = 2$}.  Suppose  $0 \leq \ell < \floor{k/4}$ and $\xi \in \supp a_{k,\ell,2}^{\mu,(\varepsilon)}$. Recall, by Lemma~\ref{lemma:2cone}, that $\theta_2(\xi)$ is the unique global minimum of the function $s \mapsto \inn{\gamma''(s)}{\xi}$ on $[-1,1]$. Thus, $   \inn{\gamma''(s)}{\xi} \geq u_2(\xi) \sim \rho 2^{k-2\ell}$, as required. 

\medskip

\noindent\underline{$\iota = 3$}. Here the localisation of the symbol ensures the key properties
\begin{equation}\label{J=4 L2 3 1}
|u_1(\xi)| \sim \rho^4 2^{k-3\ell}, \qquad |u_2(\xi)| \sim \rho 2^{k-2\ell} \qquad \text{ for all $(\xi;s) \in \supp a_{k,\ell,3}^{\mu,(\varepsilon)}$.}
\end{equation}
The argument splits into two cases:\medskip

\noindent \textbf{Case 1:}  $\min_{\pm}|s-\theta_1^{\pm}(\xi)| \leq \rho^2 2^{-\ell}$. By Taylor expansion around $\theta_1^{\pm}(\xi)$, one has 
\begin{equation}\label{J=4 L2 3 2}
     \inn{\gamma'(s)}{\xi}  = u_1^{\pm}(\xi) + u_{3,1}^{\pm}(\xi) \cdot \frac{(s-\theta_1^{\pm}(\xi))^2}{2} + \omega^{\pm}(\xi;s) \cdot (s-\theta_1^{\pm}(\xi))^3,
\end{equation}
where the functions $\omega^\pm$ arise from the third order remainder term and satisfy $|\omega^{\pm}(\xi; s)|\sim 2^k$. Moreover, \eqref{J=4 L2 3 1} and Lemma \ref{lemma:size of quantities} i) imply $|u_{3,1}^\pm(\xi)| \sim \rho^{1/2} 2^{k-\ell}$. Provided $\rho$ is sufficiently small, \eqref{J=4 L2 3 1} implies that the $u_{1}^\pm(\xi)$ term dominates the right-hand side of \eqref{J=4 L2 3 2}  and therefore $|\inn{\gamma'(s)}{\xi}| \gtrsim \rho^4 2^{k-3\ell}$.\medskip 

\noindent \textbf{Case 2:}  $\min_{\pm}|s-\theta_1^{\pm}(\xi)| \geq \rho^2 2^{-\ell}$. In this case, rather than analysing Taylor expansions, we use a convexity argument. Fix $\xi \in \supp a_{k,\ell,3}^{\mu, (\varepsilon)}$ and let 
\begin{equation*}
    \phi \colon [-1,1] \to \R, \quad  \phi \colon s \mapsto \inn{\gamma''(s)}{\xi};
\end{equation*}
 by \eqref{convex}, this function is strictly convex. Thus, given $t \in [-1,1]$, the auxiliary function 
\begin{equation*}
q_t \colon [-1,1] \to \R, \quad q_t \colon s \mapsto \frac{\phi(s)-\phi(t)}{s-t} \quad \textrm{for $s \neq t$} \quad \textrm{and} \quad q_t \colon t \mapsto \phi'(t)  
\end{equation*}
 is increasing. Setting $t := \theta_1^-(\xi)$ and noting that $\phi\circ\theta_1^-(\xi)=0$, it follows that
\begin{equation*}
    \frac{\phi(s)}{s-\theta_1^-(\xi)}  \leq \frac{\phi\circ \theta_2(\xi)}{\theta_2(\xi)-\theta_1^-(\xi)} = \frac{u_2(\xi)}{\theta_2(\xi)-\theta_1^-(\xi)} < 0 \qquad \text{ for all } -1 \leq s \leq  \theta_2(\xi),
\end{equation*}
where we have used the fact that $u_2(\xi) < 0$ on the support of $a_{k,\ell,3}$. If $s \in [\theta_2(\xi), 1]$, then we can carry out the same argument with respect to $t = \theta_1^+(\xi)$ to obtain a similar inequality. From this, we deduce the bound 
\begin{equation}\label{J=4 L2 3 3}
    |\inn{\gamma''(s)}{\xi}| \geq \min_{\pm} \frac{|u_2(\xi)| |s-\theta_1^{\pm}(\xi)|} {|\theta_2(\xi)- \theta_1^{\pm}(\xi)|}  \qquad \text{ for all } -1 \leq s \leq 1.
\end{equation}
Recall from \eqref{J=4 L2 3 1} that $|u_2(\xi)| \sim \rho 2^{k-2\ell}$ and  therefore $|\theta_2(\xi)- \theta_1^{\pm}(\xi)| \sim \rho^{1/2} 2^{-\ell}$ by  Lemma~\ref{lemma:size of quantities} i). Substituting these bounds and the hypothesis $\min_{\pm}  |s-\theta_1^{\pm}(\xi)| \geq \rho^2 2^{-\ell}$ into \eqref{J=4 L2 3 3}, we conclude that $|\inn{\gamma''(s)}{\xi}| \gtrsim \rho^{5/2}2^{k-2\ell}$.\medskip 

\noindent Thus, in either case the desired bound \eqref{J=2 L2 1} holds.\medskip

\noindent\underline{$\iota = 4$}. Here the localisation of the symbol ensures the key properties \begin{equation}\label{J=4 L2 4 1}
    |u_{1}(\xi)| \lesssim \rho^4 2^{k-3\ell}, \qquad |s - \theta_1(\xi)| \gtrsim \rho 2^{-\ell} \qquad \textrm{ for all $(\xi;s) \in \supp a_{k,\ell,4}^{\mu,(\varepsilon)}$}.
\end{equation} 
By Taylor expansion around $\theta_1(\xi)$, we obtain 
 \begin{align}\label{J=4 L2 4 2}
     \inn{\gamma'(s)}{\xi} & = u_1(\xi) + u_{3,1}(\xi) \cdot \frac{(s-\theta_1(\xi))^2}{2} + \omega_1(\xi;s) \cdot (s-\theta_1(\xi))^3,  
 \\
 \label{J=4 L2 4 3}
     \inn{\gamma''(s)}{\xi} & = u_{3,1} (\xi) \cdot (s-\theta_1(\xi)) + \omega_2(\xi;s) \cdot (s-\theta_1(\xi))^2
 \end{align}
 where the functions $\om_1$ and $\om_2$ arise from the remainder terms and satisfy $|\omega_i(\xi; s)|\sim 2^k$ for $i = 1$, $2$. It is convenient to define the functions
 \begin{align*}
 &\alpha (\xi; s):=u_{3,1} (\xi) + \omega_2(\xi; s) \cdot (s- \theta_1(\xi)),\\
 & \beta(\xi; s) := 2\omega_1(\xi;s) - \omega_2(\xi;s),
 \end{align*}
 so that \eqref{J=4 L2 4 2} and \eqref{J=4 L2 4 3} can be rewritten as
 \begin{align}
 \label{J=4 L2 4 4}
     \inn{\gamma'(s)}{\xi} &= u_1(\xi) + \big(\alpha (\xi;s) +   \beta (\xi;s)\cdot (s-\theta_1(\xi)) \big) \cdot \frac{(s- \theta_1(\xi))^2}{2}, \\
 \label{J=4 L2 4 5}
     \inn{\gamma''(s)}{\xi} &=  \alpha(\xi;s) \cdot (s-\theta_1(\xi)).
 \end{align}
 The argument splits into two cases:\medskip

\noindent \textbf{Case 1:} $|\alpha(\xi;s)|\leq  \rho^2 2^{k-\ell}$. By the integral form of the remainder, 
 \begin{equation*}
     \beta(\xi;s)\cdot (s-\theta_1(\xi))^3 = - \int_{\theta_1(\xi)}^s \inn{\gamma^{(4)}(t)}{\xi} \cdot (s - t)\cdot (t - \theta_1(\xi))\,\ud t.
 \end{equation*}
Recall from \eqref{convex} that $\inn{\gamma^{(4)}(t)}{\xi}>0$ for all $t \in [-1,1]$. Thus, the integrand in the above display has constant sign. Furthermore, \eqref{4 derivative bound} also guarantees that $|\inn{\gamma^{(4)}(t)}{\xi}| \sim 2^k$. Combining these observations, 
 \begin{equation*}
     |\beta(\xi;s)| \sim 2^k \qquad \textrm{for all $(\xi;s) \in \supp a_{k, \ell, 4}$.}
 \end{equation*} 
 Thus, provided $\rho$ is chosen sufficiently small, the hypothesis $|\alpha(\xi;s)|\leq  \rho^2 2^{k-\ell}$ and together with the bound $|s-\theta_1 (\xi)| \geq \rho2^{-\ell}$ from \eqref{J=4 L2 1 1} imply
 \begin{equation*}
    |\beta(\xi;s)| |s-\theta_1 (\xi)| - |\alpha(\xi;s)| \gtrsim \rho2^{k-\ell}.  
 \end{equation*}
Consequently, \eqref{J=4 L2 4 1} implies that the second term dominates the right-hand side of \eqref{J=4 L2 4 4} and therefore $|\inn{\gamma'(s)}{\xi}| \gtrsim \rho^3 2^{k-3\ell}$.\medskip 

\noindent \textbf{Case 2:} $|\alpha(\xi;s)| \geq  \rho^2 2^{k-\ell}$. Here \eqref{J=4 L2 4 1} and \eqref{J=4 L2 4 5} immediately imply  $|\inn{\gamma''(s)}{\xi}| \gtrsim \rho^3 2^{k-2\ell}$.\medskip

\noindent Thus, in either case the desired bound \eqref{J=2 L2 1} holds.\bigskip

\noindent b) If $\ell_1=\floor{(2k+\ell_2)/9}$, then the desired bound follows from Plancherel's theorem and the van der Corput lemma with third order derivatives. Indeed, by Lemma~\ref{J4 lower bound lemma},
\begin{equation}\label{J=4 L2 bb 1}
    |\inn{\gamma^{(3)}(s)}{\xi}| \sim \rho^{1/2} 2^{k-\ell_2} \qquad \textrm{for all $(\xi;s) \in \supp b_{k,\bm{\ell}}^{\nu,(\varepsilon)}$.}
\end{equation}

For the remaining cases, it suffices to show that
\begin{equation}\label{J=4 L2 bb 2}
    |\inn{\gamma'(s)}{\xi}| + 2^{-(3\ell_1-\ell_2)/2}|\inn{\gamma''(s)}{\xi}| \gtrsim 2^{k-3\ell_1} \qquad \textrm{for all $(\xi;s) \in \supp b_{k,\bm{\ell}}^{\nu, (\varepsilon)}$.}
\end{equation}
Here the localisation of the symbol ensures the key properties 
\begin{equation}\label{J=4 L2 bb 3}
|u_1(\xi)| \sim \rho^4 2^{k-3\ell_1}, \quad |u_2(\xi)| \sim \rho 2^{k-2\ell_2}, \quad |s - \theta_1(\xi)| \lesssim \rho 2^{-\ell_2} \qquad \textrm{for all $(\xi;s) \in \supp b_{k,\bm{\ell}}^{\nu,(\varepsilon)}$}.
\end{equation}

By Taylor expansion around $\theta_1(\xi)$, we obtain 
 \begin{align}\label{J=4 L2 bb 4}
     \inn{\gamma'(s)}{\xi} & = u_1(\xi) + \omega_1(\xi; s) \cdot (s-\theta_1(\xi))^2,  
 \\
 \label{J=4 L2 bb 5}
     \inn{\gamma''(s)}{\xi} & =  \omega_2(\xi;s) \cdot (s - \theta_1(\xi)),
 \end{align}
 where the functions $\om_1$ and $\om_2$ arise from the remainder terms and satisfy $|\omega_i(\xi; s)|\sim \rho^{1/2} 2^{k-\ell_2}$ for $i = 1$,~$2$ by \eqref{J=4 L2 bb 1}.  The argument splits into two cases:\medskip

\noindent \textbf{Case 1:} $|\theta_1(\xi) - s| \leq \rho^2 2^{-(3\ell_1-\ell_2)/2}$. Provided $\rho >0$ is chosen sufficiently small, \eqref{J=4 L2 bb 3} and the bound $|\om_{1}(\xi; s)| \sim \rho^{1/2} 2^{k - \ell_2}$ imply that the $u_1(\xi)$ term dominates the right-hand side of \eqref{J=4 L2 bb 4} and therefore $|\inn{\gamma'(s)}{\xi}| \gtrsim \rho^4 2^{k-3\ell_1}$.\medskip

\noindent\textbf{Case 2:} $|\theta_1(\xi) - s| \geq \rho^2 2^{-(3\ell_1-\ell_2)/2}$. In this case, the bound $|\om_{2}(\xi)| \sim \rho^{1/2} 2^{k - \ell_2}$ and \eqref{J=4 L2 bb 5} immediately imply $|\inn{\gamma''(s)}{\xi}| \gtrsim  \rho^{5/2} 2^{k-\ell_2-(3\ell_1-\ell_2)/2}$.\medskip

\noindent Thus, in either case the desired bound \eqref{J=4 L2 bb 2} holds.
\end{proof}


\begin{lemma}\label{J4 Linfty bounds lem}
\begin{enumerate}[a)]
    \item For all $0 \leq \ell \leq  \floor{k/4}$, $\mu \in \Z$, $1 \leq \iota \leq 4$ and  $\varepsilon > 0$, we have
\begin{equation*}
    \|m[a_{k,\ell,\iota}^{\mu, (\varepsilon)}]\|_{M^{\infty}(\R^4)} \lesim 2^{-(1-\varepsilon)\ell}.
\end{equation*}
\item For $\bm{\ell} = (\ell_1, \ell_2) \in \Lambda(k)$, $\nu \in \Z$ and  $\varepsilon>0$, we have
\begin{equation*}
    \|m[b_{k,\bm{\ell}}^{\nu, (\varepsilon)}]\|_{M^{\infty}(\R^4)} \lesim 2^{-(1-\varepsilon)(3\ell_1-\ell_2)/2}.
\end{equation*}
\end{enumerate}
\end{lemma}

\begin{proof} In view of the support properties of the symbols (see Lemma~\ref{J=4 supp lem} and Lemma~\ref{J=4 b_nu supp lem}), by an integration-by-parts argument (see Lemma~\ref{gen Linfty lem}), the problem is reduced to showing
\begin{subequations}
\begin{align}\label{J4 Linfty 1a}
       |\nabla_{\be_j(s_\mu)}^{N} a_{k,\ell, \iota}^{\mu}(\xi;s)|  & \lesssim_N  2^{-(k - (4-j)\ell) N},  \\
       \label{J4 Linfty 1b}
       |\nabla_{\be_j(s_\nu)}^{N} b_{k,\bm{\ell}}^{\nu}(\xi;s)|  & \lesssim_N  2^{-(k - ((3-j)(3\ell_1-\ell_2)/2 + \ell_2)\vee 0) N}
\end{align}
\end{subequations}
for all $1 \leq j \leq 4$ and all $N \in \N_0$. 

For all $N \in \N$, we claim the following: 

\begin{itemize}
    \item  For all $\xi \in \xisupp a_{k,\ell,\iota}^\mu$, $1\leq \iota \leq 4$,
    \begin{equation}\label{J4 Linfty 2a}
    2^{\ell}|\nabla_{{\be_j(s_\mu)}}^N \theta_2(\xi)|, \quad  2^{-k+2\ell}|\nabla_{{\be_j(s_\mu)}}^N u_2(\xi)|, \quad  2^{-k+3\ell}|\nabla_{{\be_j(s_\mu)}}^N u_{1,2}(\xi)| \lesssim_N 2^{-(k-(4-j)\ell)N};
\end{equation}
\item For all $\xi \in \xisupp a_{k,\ell,\iota}^\mu$, $3\leq \iota \leq 4$,
    \begin{equation}\label{J4 Linfty 3a}
    2^{\ell}|\nabla_{{\be_j(s_\mu)}}^N \theta_1(\xi)|, \quad 2^{-k+3\ell}|\nabla_{{\be_j(s_\mu)}}^N u_1(\xi)| \lesssim_N 2^{-(k-(4-j)\ell)N};
\end{equation}
\item For all $\xi \in \xisupp b_{k,\bm{\ell}}^\nu$, 
\begin{equation}\label{J4 Linfty 2b}
    2^{\ell_2}|\nabla_{{\be_j(s_\nu)}}^N \theta_2(\xi)|, \quad  2^{-k+2\ell_2}|\nabla_{{\be_j(s_\nu)}}^N u_2(\xi)|, \quad  2^{-k+3\ell_2}|\nabla_{{\be_j(s_\nu)}}^N u_{1,2}(\xi)| \lesssim_N 2^{-(k-(4-j)\ell_2)N};
\end{equation}
\item  For all $\xi \in \xisupp b_{k,\bm{\ell}}^\nu$, 
\begin{equation}\label{J4 Linfty 3b}
    2^{(3\ell_1-\ell_2)/2}|\nabla_{\be_j(s_\nu)}^N \theta_1(\xi)|, \quad 2^{-k+3\ell_1}|\nabla_{\be_j(s_\nu)}^N u_1(\xi)| \lesssim_N 2^{-(k-((3-j)(3\ell_1-\ell_2)/2 + \ell_2)\vee 0)N}.
\end{equation}
\end{itemize}

Once the above claims are established, the derivative bounds \eqref{J4 Linfty 1a} and \eqref{J4 Linfty 1b} follow directly from the chain and Leibniz rule.

\medskip

In order to prove \eqref{J4 Linfty 2a}-\eqref{J4 Linfty 3b} we work with the unified framework introduced in \S\ref{subsec:Fourier loc J=4}. \medskip

We start with \eqref{J4 Linfty 2a} and \eqref{J4 Linfty 2b}. Given $n$, $s \in \R$, recall the set $\Xi_2(k,n; s)$ introduced in \eqref{Xi_2 def}. In particular, if $\xi \in \Xi_2(k,n; s)$, then $\xi \in \xisupp a_k$ and $\xi$ lies in the domain of $\theta_2$ and satisfies
\begin{equation}\label{J4 Linfty 2}
    |\theta_2(\xi) - s| \lesssim 2^{-n} \qquad \textrm{and} \qquad |u_2(\xi)| \lesssim 2^{k-2n}.
\end{equation}
From the discussion following \eqref{Xi_2 def}, we know that
\begin{equation*}
    \xisupp a_{k,\ell, \iota}^{\mu} \subseteq \Xi_2(k, \ell ;s_{\mu}) \quad \textrm{for $1 \leq \iota \leq 4$} \quad \textrm{and} \quad \xisupp b_{k,\bm{\ell}}^{\nu} \subseteq \Xi_2\big(k, \ell_2 ;s_{\nu}\big).
\end{equation*}

Let $\xi \in \Xi_2(k, n; s)$ and for $1 \leq j \leq 4$ define $\bm{v}_j := \be_j(s)$. The bounds \eqref{J4 Linfty 2a} and \eqref{J4 Linfty 2b} amount to proving that
\begin{equation}\label{J4 Linfty 4}
    2^{n}|\nabla_{\bm{v}_j}^N \theta_2(\xi)|, \quad  2^{-k+2n}|\nabla_{\bm{v}_j}^N u_{1,2}(\xi)|, \quad  2^{-k+3n}|\nabla_{\bm{v}_j}^N u_2(\xi)| \lesssim_N 2^{-(k-(4-j)n)N}
\end{equation}
hold for all $N \in \N$. These bounds follow from repeated application of the chain rule, provided
\begin{subequations}
\begin{align}\label{J4 Linfty 5a}
    |\inn{\gamma^{(4)}\circ \theta_2(\xi)}{\xi}| &\gtrsim 2^k, \\
    \label{J4 Linfty 5b}
    |\inn{\gamma^{(K)}\circ \theta_2(\xi)}{\xi}| &\lesssim_{K} 2^{k + n(K-4)},  \\
    \label{J4 Linfty 5c}
    |\inn{\gamma^{(K)}\circ \theta_2(\xi)}{\bm{v}_j}| &\lesssim_{K} 2^{(K-j)n}
\end{align}
\end{subequations}
hold for all $K \geq 2$. In particular, assuming \eqref{J4 Linfty 5a}, \eqref{J4 Linfty 5b} and \eqref{J4 Linfty 5c}, the bounds in \eqref{J4 Linfty 4} are then a consequence of Lemma~\ref{imp deriv lem} in the appendix. More precisely, the desired estimates in \eqref{J4 Linfty 4} correspond to \eqref{multi imp der bound} and two separate instances of \eqref{multi Faa di Bruno eq 2} whilst the hypotheses in the above display correspond to \eqref{multi imp deriv 2} and \eqref{multi imp deriv 3}. Here the parameters featured in the appendix are chosen as follows: 

 \begin{center}
    \begin{tabular}{|c|c|c|c|c|c|c|} 
\hline
 & & & & & & \\[-0.8em]
 $g$ & $h$ & $A$ & $B$ & $M_1$ & $M_2$ & $\be$ \\
 & & & & & & \\[-0.8em]
\hline
  & & & & & & \\[-0.8em]
 $\gamma^{(3)}$ & 
 $\gamma''$ & $2^{k- n}$ & $2^{k-2n}$ & $2^{-k+(4-j)n}$ & $2^n$ & $\bm{v}_j$ \\
 & & & & & & \\[-0.8em]
 \hline
 & & & & & & \\[-0.8em]
 $\gamma^{(3)}$ & 
 $\gamma'$ & $2^{k- n}$ & $2^{k-3n}$ & $2^{-k+(4-j)n}$ & $2^n$ & $\bm{v}_j$  \\
 & & & & & & \\[-0.8em]
 \hline
\end{tabular}
\end{center}

The conditions \eqref{J4 Linfty 5a}, \eqref{J4 Linfty 5b} and \eqref{J4 Linfty 5c} follow directly from the definition of $\Xi_2(k, n; s)$. Indeed, \eqref{J4 Linfty 5a} and the $K \geq 4$ case of \eqref{J4 Linfty 5b} are trivial consequences of the localisation of the symbol $a_k$. The $K = 3$ case of \eqref{J4 Linfty 5b} follows immediately since $\inn{\gamma^{(3)} \circ \theta_2(\xi)}{\xi}=0$ and the $K = 2$ case of \eqref{J4 Linfty 5b} is just a restatement of the condition $|u_2(\xi)| \lesssim 2^{k-2n}$ from \eqref{J4 Linfty 2}. Finally, \eqref{Frenet bound alt 1} together with the $\theta_2$ localisation hypothesis from \eqref{J4 Linfty 2} imply that
\begin{equation*}
    |\inn{\gamma^{(K)}\circ \theta_2(\xi)}{\bm{v}_j}| \lesssim_K |\theta_2(\xi) - s|^{(j-K) \vee 0} \lesssim 2^{-((j-K) \vee 0)n}
\end{equation*}
which yields \eqref{J4 Linfty 5c}.\medskip

We next turn to \eqref{J4 Linfty 3a} and \eqref{J4 Linfty 3b}. Given $\bm{n} = (n_1, n_2) \in \R^2$ and $s \in \R$, recall the set $\Xi_1(k, \bm{n}; s)$ introduced in \eqref{Xi_1 def}. In particular, if $\xi \in \Xi_1(k, \bm{n}; s)$, then $\xi \in \xisupp a_k$ and $\xi$ lies in the domain of $\theta_1$ and satisfies
\begin{equation}\label{J4 Linfty 6}
    |\theta_1(\xi) - s| \lesssim 2^{-n_1} \quad \textrm{and} \qquad |u_{3,1}(\xi)| \sim 2^{k-n_2}. 
\end{equation}
From the discussion following \eqref{Xi_1 def}, we know that
\begin{equation*}
    \xisupp a_{k,\ell, \iota}^{\mu} \subseteq \Xi_1(k, \ell, \ell ;s_{\mu}) \quad \textrm{for $\iota = 3$, $4$} \quad \textrm{and} \quad \xisupp b_{k,\bm{\ell}}^{\nu} \subseteq \Xi_1\big(k, \tfrac{3\ell_1 - \ell_2}{2}, \ell_2 ;s_{\nu}\big).
\end{equation*}

Let $\xi \in \Xi_1(k, \bm{n};s)$ where $\bm{n} = (n_1, n_2)$ for some $0 < n_2 \leq n_1$ and for $1 \leq j \leq 4$ define $\bm{v}_j := \be_j(s)$. The bounds \eqref{J4 Linfty 3a} and \eqref{J4 Linfty 3b} amount to proving that
\begin{equation}\label{J4 Linfty 8}
    2^{n_1}|\nabla_{\bm{v}_j}^N \theta_1(\xi)|, \quad 2^{-k+2n_1 + n_2}|\nabla_{\bm{v}_j}^N u_1(\xi)| \lesssim_N 2^{-(k-((3-j)n_1 + n_2)\vee 0)N}
\end{equation}
hold for all $N \in \N$. These bounds follow from repeated application of the chain rule, provided
\begin{subequations}
\begin{align}\label{J4 Linfty 9a}
    |\inn{\gamma^{(3)}\circ \theta_1(\xi)}{\xi}| &\gtrsim 2^{k-n_2}, \\
    \label{J4 Linfty 9b}
    |\inn{\gamma^{(K)}\circ \theta_1(\xi)}{\xi}| &\lesssim_K 2^{k + n_1(K-3) - n_2},  \\
    \label{J4 Linfty 9c}
    |\inn{\gamma^{(K)}\circ \theta_1(\xi)}{\bm{v}_j}| &\lesssim_K 2^{n_1(K - 3) - n_2 + ((3-j)n_1 + n_2)\vee 0}
\end{align}
\end{subequations}
hold for all $K \geq 2$. In particular, assuming \eqref{J4 Linfty 9a}, \eqref{J4 Linfty 9b} and \eqref{J4 Linfty 9c}, the bounds in \eqref{J4 Linfty 8} are then a consequence of Lemma~\ref{imp deriv lem} in the appendix. More precisely, the desired estimates in \eqref{J4 Linfty 8} correspond to \eqref{multi imp der bound} and \eqref{multi Faa di Bruno eq 2} whilst the hypotheses in the above display correspond to \eqref{multi imp deriv 2} and \eqref{multi imp deriv 3}. Here the parameters featured in the appendix are chosen as follows: 
 \begin{center}
    \begin{tabular}{|c|c|c|c|c|c|c|} 
\hline
 & & & & & & \\[-0.8em]
 $g$ & $h$ & $A$ & $B$ & $M_1$ & $M_2$ & $\be$ \\
 & & & & & & \\[-0.8em]
\hline
  & & & & & & \\[-0.8em]
 $\gamma''$ & 
 $\gamma'$ & $2^{k- n_1 - n_2}$ & $2^{k - 2n_1 - n_2}$ & $2^{-k+((3-j)n_1 + n_2)\vee 0}$ & $2^{n_1}$ & $\bm{v}_j$ \\
 & & & & & & \\[-0.8em]
 \hline
\end{tabular}
\end{center}

The conditions \eqref{J4 Linfty 9a}, \eqref{J4 Linfty 9b} and \eqref{J4 Linfty 9c} follow directly from the definition of $\Xi_1(k, \bm{n};s)$. Indeed, \eqref{J4 Linfty 9a} and the $K = 3$ case of \eqref{J4 Linfty 9b} are just a restatement of the condition $|u_{3,1}(\xi)| \sim 2^{k-n_2}$ from \eqref{J4 Linfty 6}. The $K \geq 4$ case of \eqref{J4 Linfty 9b} is a trivial consequence of the localisation of the symbol $a_k$ whist the remaining $K = 2$ case of \eqref{J4 Linfty 9b} follows immediately since $\inn{\gamma'' \circ \theta_1(\xi)}{\xi}=0$.  Finally, \eqref{Frenet bound alt 1} together with the $\theta_1$ localisation hypothesis from \eqref{J4 Linfty 6} imply that
\begin{equation*}
    |\inn{\gamma^{(K)}\circ \theta_1(\xi)}{\bm{v}_j}| \lesssim_N |\theta_1(\xi) - s|^{(j-K) \vee 0} \lesssim 2^{((j-K) \vee 0)n_1}
\end{equation*}
which, by directly comparing exponents, yields \eqref{J4 Linfty 9c}. 
\end{proof}


Lemma~\ref{L2 bounds J=4 lem} and Lemma~\ref{J4 Linfty bounds lem} can be combined to obtain the following $L^p$ bounds.

\begin{corollary}\label{cor:J=4 p-sum}
    For all $2 \leq p \leq \infty$ and all $\varepsilon>0$, the following inequalities hold: 
    \begin{enumerate}[a)]
        \item For all $0 \leq \ell\le \floor{k/4}$ and $1 \leq \iota \leq 4$,
        \begin{equation*}
        \Big( \sum_{\mu \in \Z} \|m[a_{k,\ell,\iota}^{\mu, (\varepsilon)}](D)f \|_{L^p(\R^4)}^p\Big)^{1/p} \lesssim 2^{-k/p+\ell(4/p-1)+ \varepsilon \ell } \|f\|_{L^p(\R^4)}.
    \end{equation*}
    \item For all $\bm{\ell}=(\ell_1,\ell_2) \in \Lambda(k)$,
    \begin{equation*}
    \Big( \sum_{\nu \in \Z} \|m[b_{k,\bm{\ell}}^{\nu, (\varepsilon)}](D)f \|_{L^p(\R^4)}^p\Big)^{1/p} \lesim 2^{-k/p +(3\ell_1+\ell_2)/2p -(3\ell_1-\ell_2)(1/2-1/p-\varepsilon)} \|f\|_{L^p(\R^4)}.
    \end{equation*}
    \end{enumerate}
    When $p = \infty$ the left-hand $\ell^p$-sums are interpreted as suprema in the usual manner. 
\end{corollary}

\begin{proof} For $p = 2$ the estimate a) and b) follow by the combining $L^2$ bounds from Lemma~\ref{L2 bounds J=4 lem} with a simple orthogonality argument, as the supports of $m[a_{k,\ell,\iota}^{\mu, (\varepsilon)}]$ and $m[b_{k,\bm{\ell}}^{\nu, (\varepsilon)}]$ are essentially disjoint for different $\mu$ and $\nu$ respectively. For $p = \infty$ the estimate is a restatement of the $L^{\infty}$ bounds from Lemma~\ref{J4 Linfty bounds lem}. Interpolating these two endpoint cases, using mixed norm interpolation (see, for instance, \cite[\S 1.18.4]{Triebel1978}), concludes the proof. 
\end{proof}


\subsection{Putting everything together} We are now ready to combine the ingredients to conclude the proof of Proposition~\ref{J4 Lp proposition}.

\begin{proof}[Proof of Proposition~\ref{J4 Lp proposition}] a) Let $1 \leq \iota \leq 4$.
By Proposition~\ref{J4 dec prop} a), for all $2 \leq p \leq 12$ and all $\varepsilon > 0$ one has
    \begin{equation*}
       \|m[a_{k,\ell,\iota}](D)f\|_{L^p(\R^4)} = \Big\|\sum_{\mu \in \Z} m[a_{k,\ell,\iota}^{\mu}](D)f\Big\|_{L^p(\R^4)} \lesssim_{\varepsilon} 2^{\ell(1/2 - 1/p) + \varepsilon \ell} \Big(\sum_{\mu \in \Z} \|m[a_{k,\ell,\iota}^{\mu}](D)f\|_{L^p(\R^4)}^p\Big)^{1/p}.
    \end{equation*}
Moreover, for all $\mu \in \Z$, Lemma~\ref{J=4 curve loc lem} a) implies that
  \begin{equation*}
    \|m[a_{k,\ell,\iota}^{\mu}]\|_{M^p(\R^4)}  \lesssim_{N, \varepsilon, p} \| m[a_{k,\ell,\iota}^{\mu, (\varepsilon)}]\|_{M^p(\R^4)} + 2^{-k} \quad \text{ for all $N \in \N$.}
  \end{equation*}
Combining the above, we obtain
\begin{equation*}
    \|m[a_{k,\ell,\iota}](D)f\|_{L^p(\R^4)} \lesssim_{ \varepsilon, p} 2^{\ell(1/2 - 1/p) + \varepsilon \ell} \Big(\sum_{\mu \in \Z} \|m[a_{k,\ell,\iota}^{\mu, (\varepsilon)}](D)f\|_{L^p(\R^4)}^p\Big)^{1/p} + 2^{-k} \| f \|_{L^p(\R^4)},
\end{equation*}
which, together with Corollary~\ref{cor:J=4 p-sum} a), yields
\begin{equation*}
     \|m[a_{k,\ell, \iota}](D)f\|_{L^p(\R^4)}  \lesssim_{\varepsilon,p} 2^{-k/p-\ell(1/2 - 3/p- 2\varepsilon) }  \|f\|_{L^p(\R^4)}.
\end{equation*}
Since $\varepsilon > 0$ was chosen arbitrarily, this is the required bound.\medskip

\noindent b) By Proposition~\ref{J=4 b_nu dec prop}, for all $6 \leq p \leq 12$ and all $\varepsilon > 0$ one has
    \begin{equation*}
       \|m[b_{k,\bm{\ell}}](D)f\|_{L^p(\R^4)} 
       \lesssim_{\varepsilon} 2^{\ell_2(1/2-1/p + \varepsilon)} 2^{3(\ell_1-\ell_2)(1-4/p + \varepsilon)/2} \Big(\sum_{\nu \in \Z} \|m[b_{k,\bm{\ell}}^{\nu}](D)f\|_{L^p(\R^4)}^p\Big)^{1/p}.
    \end{equation*}
Moreover, for all $\nu \in \Z$, Proposition~\ref{J=4 curve loc lem} b) implies that
  \begin{equation*}
    \|m[b_{k,\bm{\ell}}^{\nu}]\|_{M^p(\R^4)}  \lesssim_{N, \varepsilon, p} \| m[b_{k,\bm{\ell}}^{\nu, (\varepsilon)}]\|_{M^p(\R^4)} + 2^{-kN} \quad \text{ for all $N \in \N$.}
  \end{equation*}
Combining the above, we obtain
\begin{align*}
    \|m[b_{k,\bm{\ell}}](D)f\|_{L^p(\R^4)} &\lesssim_{ \varepsilon, p} 2^{\ell_2(1/2-1/p + \varepsilon)} 2^{3(\ell_1-\ell_2)(1-4/p + \varepsilon)/2} \Big(\sum_{\nu \in \Z} \|m[b_{k,\bm{\ell}}^{\nu, (\varepsilon)}](D)f\|_{L^p(\R^4)}^p\Big)^{1/p} \\ 
    & \qquad + 2^{-k} \| f \|_{L^p(\R^4)},
\end{align*}
which, together with Corollary~\ref{cor:J=4 p-sum} b), yields 
     \begin{equation*}
      \|m[b_{k,\bm{\ell}}](D)f\|_{L^p(\R^4)}  \lesssim_{\varepsilon,p} 2^{-3(\ell_1 - \ell_2)(1/2p - 2\varepsilon) -\ell_2(1/2 - 3/p - 2\varepsilon)}\|f \|_{L^p(\R^4)}.
   \end{equation*}
Since $\varepsilon > 0$ was chosen arbitrarily, this is the required bound.
\end{proof}

We have established Proposition~\ref{J4 Lp proposition} and therefore completed the proof of the $J=4$ case of Theorem~\ref{Sobolev theorem}.




\section{Proof of the decoupling inequalities}\label{sec:decoupling}

This section is devoted to the proof of Theorem~\ref{Frenet decoupling theorem}.




\subsection{Decoupling inequalities for non-degenerate curves}\label{BDG subsec}

The central ingredient in the proof of Theorem~\ref{Frenet decoupling theorem} is the decoupling theorem of Bourgain--Demeter--Guth \cite{BDG2016}. We begin by recalling the statement of (one formulation of) this result. Given a non-degenerate curve $g \in C^{d+1}(I;\hat{\R}^d)$ and $0 < r  \leq 1$, an `anisotropic $r$-neighbourhood' of the curve is constructed as follows.
 
 \begin{definition}\label{slab def}
 For each $s \in I$ define the parallelepiped 
\begin{equation*}
\alpha(s;r):= \big \{ \xi \in \hat{\R}^d : \xi = g(s) + \sum_{j=1}^d \lambda_j r^j g^{(j)}(s) \quad \textrm{for some $\lambda_j \in [-2,2]$, $1 \leq j \leq d$}\big\};
\end{equation*}
such sets are referred to as $r$-\textit{slabs}.
  \end{definition} 
  
  In some cases it is useful to highlight the choice of function $g$ by writing $\alpha(g;s;r)$ for a $r$-slab $\alpha(s;r)$. Note that the formula for the parallelepiped $\alpha(s;r)$ can be expressed succinctly in terms of the matrix $[g]_{s,r}$ introduced in~\eqref{gamma transformation}. In particular,
\begin{equation}\label{slab matrix form}
    \alpha(s;r) = g(s) + [g]_{s,r}\big([-2,2]^d\big).
\end{equation}
  
An anisotropic $r$-neighbourhood of the curve $g$ is formed by taking the union of all the $r$-slabs as $s$ varies over $I$. 

\begin{definition} A collection $\mathcal{A}(r)$ of $r$-slabs is a \textit{slab decomposition for $g$} if it consists of precisely the $r$-slabs $\alpha(g;s;r)$ for $s$ varying over a $r$-net in $I$. 
\end{definition}

With the above definitions, the decoupling theorem may be stated as follows.  

\begin{theorem}[Bourgain--Demeter--Guth \cite{BDG2016}]\label{BDG theorem} Let $g \in \mathfrak{G}_d(\delta)$ for some $0 < \delta \ll 1$, $0 < r \leq 1$ and $\mathcal{A}(r)$ be a $r$-slab decomposition for $g$. For all $2 \leq p \leq d(d+1)$ and $\varepsilon > 0$ the inequality
\begin{equation}\label{BDG ineq}
    \Big\|\sum_{\alpha \in \mathcal{A}(r)} f_{\alpha} \Big\|_{L^p(\R^{d})} \lesssim_{\varepsilon} r^{-\varepsilon} \Big(\sum_{\alpha \in \mathcal{A}(r)} \|f_{\alpha}\|_{L^p(\R^{d})}^2\Big)^{1/2}
\end{equation}
holds for any tuple of functions $(f_{\alpha})_{\alpha \in \mathcal{A}(r)}$ satisfying $\supp \hat{f}_{\alpha} \subseteq \alpha$.
\end{theorem}

\begin{remark} This is a slight variant of the decoupling inequality of Bourgain--Demeter--Guth \cite{BDG2016} which can be found, for instance, in \cite{GLYZ}.\footnote{More precisely, the general version of the decoupling theorem here follows by combining Theorem 1.2 and Lemma 3.6 from \cite{GLYZ}.} It is also remarked that the result holds for general non-degenerate curves, although not in the uniform fashion described here. Note, in particular, that by restricting to the model curves $g \in \mathfrak{G}_d(\delta)$ for $0 < \delta \ll 1$, the decoupling inequality \eqref{BDG ineq} holds with a constant independent of both the choice of $g$ and $\delta$. 
\end{remark}




\subsection{Geometric observations}\label{geo obs sec} In order to relate Theorem~\ref{Frenet decoupling theorem} to the Bourgain--Demeter--Guth result from Theorem~\ref{BDG theorem}, we first relate the Frenet boxes $\pi_{d-1,\gamma}(s;r)$ to certain regions which are more similar in form to the slabs $\alpha(g;s,r)$ introduced above. The Frenet boxes $\pi_{d-1,\gamma}(s;r)$ do not correspond precisely to slabs but to related regions referred to as \textit{plates}. These plate regions are formed by extending $d$-dimensional slabs into $n$-dimensions by adjoining additional long directions. Moreover, the plates are naturally defined in relation to a cone generated over a family of non-degenerate curves $g_j \colon I \to \R^d$.  \medskip

\noindent \textit{A family of cones.} Let $\gamma \in \mathfrak{G}_{n}(\delta)$ for $0 < \delta \ll 1$ and $\be_j \colon[-1,1] \to S^{n-1}$ for $1 \leq j \leq n$ be the associated Frenet frame. Without loss of generality, in proving Theorem~\ref{Frenet decoupling theorem} we may always localise so that we only consider the portion of the curve lying over the interval $I = [-\delta, \delta]$. In this case
\begin{equation}\label{Frenet loc}
\be_j(s) = \vec{e}_j + O(\delta) \qquad \textrm{for $1 \leq j \leq n$}
\end{equation}
where, as in Definition~\ref{rescaled curve def}, the $\vec{e}_j$ denote the standard basis vectors. 

Here we introduce certain conic surfaces which are `generated' over the curves $s \mapsto \be_j(s)$. The following observations extend the analysis of \cite{PS2007}, where a cone in $\R^3$ generated by the binormal vector $\be_3$ features prominently in the proof of the 3-dimensional analogue of Theorem~\ref{non-degenerate theorem}. 

Let $2 \leq d \leq n-1$ and consider the map $\tilde{\Gamma} \colon \R^{n-d} \times I \to \R^n$ defined by
\begin{equation*}
    \tilde{\Gamma}(\vec{\lambda}, s) := \sum_{j=d+1}^n \lambda_j \be_j(s), \qquad \vec{\lambda} = (\lambda_{d+1},\dots, \lambda_n).
\end{equation*}
Restricting to $\lambda_{d+1}$ bounded away from zero, this is a regular parametrisation of a $(n-d+1)$-dimensional surface in $\R^n$, which is denoted $\Gamma_{n,d}$. Indeed, by the Frenet formul\ae,
\begin{align*}
\frac{\partial\tilde{\Gamma}}{\partial s}(\vec{\lambda}, s) &= -\lambda_{d+1}\tilde{\kappa}_d(s)\be_d(s) + E_d(\vec{\lambda},s), \\
\frac{\partial\tilde{\Gamma}}{\partial \lambda_j}(\vec{\lambda},s) &= \be_j(s), \qquad d+ 1 \leq j \leq n,
\end{align*}
where  $E_d(\vec{\lambda},s)$ lies in the subspace $\langle \be_{d+1}(s), \dots, \be_n(s)\rangle$. Thus, provided $\lambda_{d+1}$ is bounded away from zero, the non-vanishing of $\tilde{\kappa}_d$ ensures that these tangent vectors are linearly independent.\medskip

\noindent \textit{Reparametrisation.} It is convenient to reparametrise $\Gamma_{n,d}$  so that it is realised as a surface `generated' over an alternative family of curves which is formed by graphs. To this end, let $A \colon I \to \mathrm{GL}(n-d, \R)$ be given by 
\begin{equation*}
    A(s) :=
    \begin{bmatrix}
    \be_{d+1,d+1}(s) & \cdots & \be_{n,d+1}(s) \\
    \vdots & & \vdots \\
    \be_{d+1,n}(s) & \cdots & \be_{n,n}(s)
    \end{bmatrix}^{-1},
\end{equation*}
where $\be_{i,j}(s)$ denotes the $j$th component of $\be_i(s)$. Provided $\delta$ is chosen sufficiently small, \eqref{Frenet loc} ensures that the above matrix inverse is well-defined and, moreover, is a small perturbation of the identity matrix. Define the reparametrisation 
\begin{equation}\label{G repara 1}
    \Gamma(\vec{\lambda}, s) := \tilde{\Gamma}(A(s)\vec{\lambda}, s) \qquad \textrm{for all $(\vec{\lambda},s) \in \R^{n-d}\times I$.}
\end{equation}
Consider the restriction of this mapping to the set $\mathcal{R}_{n,d}' \subset \R^{n-d}$ consisting of all vectors $\vec{\lambda} = (\lambda_{d+1}, \dots, \lambda_n)$ satisfying 
\begin{equation}\label{G repara 2}
1/4 \leq \lambda_{d+1} \leq 2 \quad \textrm{and} \quad |\lambda_j|\leq 2 \quad \textrm{for $d+2 \leq j \leq n$};
\end{equation}
under this restriction, $\Gamma$ is a regular parametrisation by the preceding observations.  

The mapping \eqref{G repara 1} can be expressed in matrix form as
\begin{align}\label{G repara 3}
  \Gamma(\vec{\lambda}, s) &= 
  \begin{bmatrix}
  \be_{d+1}(s) & \cdots & \be_n(s)
  \end{bmatrix} \cdot A(s) \vec{\lambda}, \\
  \nonumber
 & = 
    \begin{bmatrix}
    G_{d+1}(s) & \cdots & G_n(s)
    \end{bmatrix} \vec{\lambda}, 
\end{align}
where the $G_j \colon I \to \R^n$ (which form the column vectors of the above matrix) are of the form 
\begin{equation*}
    G_j(s) =
    \begin{bmatrix}
    g_j(s) \\
    0
    \end{bmatrix} 
    + \vec{e}_j
\end{equation*}
for some smooth function $g_j \colon I \to \R^d$.\medskip 

\noindent \textit{Non-degeneracy conditions.} Given $\ba = (a_{d+1}, \dots, a_n) \in \mathcal{R}'_{n,d}$, define
\begin{equation}\label{g a non deg 1}
   G_{\ba} := \sum_{j=d+1}^n a_j\cdot G_j \qquad \textrm{and} \qquad g_{\ba}:= \sum_{j=d+1}^na_j \cdot g_j,
\end{equation}
noting $G_{\ba}(s) = \Gamma(\ba, s)$. The curve $g_{\ba} \colon I \to \R^d$ is non-degenerate. To see this, first note that $\frac{\partial^i \Gamma}{\partial s^i}(\vec{\lambda},s)$ can be expressed as a linear combination of vectors of the form
\begin{equation}\label{g a non deg 2}
    \begin{bmatrix}
  \be_{d+1}^{(\ell)}(s) & \cdots & \be_n^{(\ell)}(s)
  \end{bmatrix} \cdot A^{(i-\ell)}(s) \,\vec{\lambda}, \qquad 0 \leq \ell \leq i,
\end{equation}
where $A^{(k)}$ denotes the component-wise $k$th-derivative of $A$. Indeed, this follows simply by applying the Leibniz rule to \eqref{G repara 3}. Consequently, $\frac{\partial^i \Gamma}{\partial s^i}(\vec{\lambda},s)$ must lie in the subspace generated by the columns of the left-hand matrix in \eqref{g a non deg 2}, where $i$ is allowed to vary over the stated range. In particular, one concludes from the Frenet formul\ae\ that
\begin{equation}\label{g a non deg 3}
    \frac{\partial^i\Gamma}{\partial s^i} (\vec{\lambda},s) \in \langle \be_{d+1 -i}(s), \dots, \be_n(s)  \rangle \qquad \textrm{for $0 \leq i \leq d$.}
\end{equation}
On the other hand, the Frenet formul\ae\ also show that the $\be_{d+1 -i}(s)$ component of $\frac{\partial^i\Gamma}{\partial s^i} (\vec{\lambda},s)$ arises only from the term in \eqref{g a non deg 2} corresponding to $\ell = i$ and 
\begin{equation}\label{g a non deg 4}
   \inn{\frac{\partial^i\Gamma}{\partial s^i} (\vec{\lambda},s)}{\be_{d+1 -i}(s)} = (-1)^i\Big( \prod_{\ell = d+1 -i}^d \tilde{\kappa}_{\ell}(s) \Big) \inn{\vec{A}_1(s)}{\vec{\lambda}},
\end{equation}
where $\vec{A}_1(s)$ denotes the first row of $A(s)$. Recall $A$ is a small perturbation of the identity matrix. Thus, under the constraint $\vec{\lambda} \in \mathcal{R}'_{n,d}$ from \eqref{G repara 2}, if $\delta$ is chosen sufficiently small, then \eqref{g a non deg 4} implies that
\begin{equation}\label{g a non deg 5}
   |\inn{\frac{\partial^i\Gamma}{\partial s^i} (\vec{\lambda},s)}{\be_{d+1 -i}(s)} | \sim 1 \qquad \textrm{for all $1 \leq i \leq d$.}
\end{equation}
Thus, combining \eqref{g a non deg 3} and \eqref{g a non deg 5}, it follows that the vectors $\frac{\partial^i\Gamma}{\partial s^i} (\vec{\lambda},s)$, $1 \leq i \leq d$, are linearly independent. Moreover, fixing $\vec{\lambda} = \ba$ and noting that $G_{\ba}^{(i)}(s) = \frac{\partial^i\Gamma}{\partial s^i} (\ba,s) \in \R^d \times \{0\}^{n-d}$ for $i \geq 1$, one concludes that 
\begin{equation}\label{g a non deg 6}
    |\det [g_{\ba}]_{s}| \gtrsim 1
\end{equation}
for all $s \in I$, which is the claimed non-degeneracy condition. Note this holds uniformly over the choice of original curve $\gamma \in \mathfrak{G}_n(\delta)$ and over $\ba \in \mathcal{R}'_{n,d}$. 
\medskip

\noindent \textit{Frenet boxes revisited.} From the preceding observations, the vectors $G_{\ba}^{(i)}(s)$ for $1 \leq i \leq d$ form a basis of $\R^d \times \{0\}^{n-d}$. Fixing $\xi \in \hat{\R}^n$ and $r > 0$, one may write
\begin{equation}\label{Frenet box 1}
    \xi - \sum_{j=d+1}^n \xi_j G_j(s) =  \sum_{i=1}^d r^{i} \eta_i G_{\ba}^{(i)}(s)
\end{equation}
for some vector of coefficients $(\eta_1, \dots, \eta_d) \in \R^d$. The powers of $r$ appearing in the above expression play a normalising r\^ole below. For each $1 \leq k \leq d$ form the inner product of both sides of the above identity with the Frenet vector $\be_k(s)$. Combining the resulting expressions with the linear independence relations inherent in \eqref{g a non deg 3}, the coefficients $\eta_k$ can be related to the numbers $\inn{\xi}{\be_k(s)}$ via a lower anti-triangular transformation, viz. 
\begin{equation}\label{Frenet box 2}
    \begin{bmatrix}
    \inn{\xi}{\be_1(s)} \\
    \vdots \\
    \inn{\xi}{\be_d(s)}
    \end{bmatrix}
    =
    \begin{bmatrix}
    0 & \cdots & \inn{G_{\ba}^{(d)}(s)}{\be_1(s)} \\
    \vdots & \ddots & \vdots \\
    \inn{G_{\ba}^{(1)}(s)}{\be_d(s)} & \cdots & \inn{G_{\ba}^{(d)}(s)}{\be_d(s)}
    \end{bmatrix}
    \begin{bmatrix}
    r\eta_1 \\
    \vdots \\
    r^d \eta_d
    \end{bmatrix}.
\end{equation}
Thus, if $\xi \in \pi_{d-1,\gamma}(s;r)$, then it follows from combining \eqref{neighbourhood 1} and \eqref{g a non deg 5} with \eqref{Frenet box 2} that $|\eta_i| \lesssim_{\gamma} 1$ for $1 \leq i \leq d$, provided $\delta >0$ is sufficiently small. Similarly, the conditions \eqref{neighbourhood 2}, \eqref{neighbourhood 3} and the localisation \eqref{Frenet loc} imply that 
\begin{equation*}
 \pi_{d-1,\gamma}(s;r) \subseteq    \mathcal{R}_{n,d} := [-2,2]^d \times \mathcal{R}'_{n,d}
\end{equation*}

The identity \eqref{Frenet box 1} can be succinctly expressed using matrices. In particular, collect the functions $g_j$ together as an $(n-d)$-tuple $\bg := (g_{d+1}, \dots, g_n)$ and, for $s \in I$ and $r > 0$, define the $n \times n$ matrix
\begin{equation}\label{Frenet box 3}
    [\bg]_{\ba, s,r} :=
    \begin{pmatrix}
    [g_{\ba}]_{s,r} & \bg(s) \\
    0 & \mathrm{I}_{n-d}
    \end{pmatrix}.
\end{equation}
Here the block $[g_{\ba}]_{s,r}$ is the $d \times d$ matrix \eqref{gamma transformation} with $\gamma$ here taken to be $g_{\ba}$ as defined in \eqref{g a non deg 1}, whilst $\bg(s)$ is understood to be the $(n-d) \times d$ matrix with $j$th column equal to $g_j(s)$ and $\mathrm{I}_{n-d}$ is the $(n-d) \times (n-d)$ identity matrix. With this notation, the identity \eqref{Frenet box 1} may be written as
\begin{equation*}
    \xi = [\bg]_{\ba,s,r} \cdot \eta \qquad \textrm{where $\eta = (\eta_1, \dots, \eta_d, \xi_{d+1}, \dots, \xi_n)$.}
\end{equation*}
Moreover, if $\xi \in \pi_{d-1,\gamma}(s;r)$, then the preceding observations show that $\eta$ in the above equation may be taken to lie in a bounded region and so
\begin{equation}\label{Frenet box 4}
    \pi_{d-1,\gamma}(s;r) \subseteq \bigcap_{\ba \in \mathcal{R}'_{n,d}} [\bg]_{\ba,s,C r}\big([-2,2]^n\big) \cap \mathcal{R}_{n,d},
\end{equation}
where $C \geq 1$ is a suitably large dimensional constant. The right-hand side of \eqref{Frenet box 4} should be compared with the matrix definition of the slabs used in the Bourgain--Demeter--Guth theorem from \eqref{slab matrix form}.




\subsection{Decoupling inequalities for cones generated by non-degenerate curves}\label{cone dec subsec} Here the geometric setup described in \S\ref{geo obs sec} is abstracted. We first generalise the definition \eqref{G repara 1} to arbitrary cones generated over a tuple of curves $(g_{d+1}, \dots, g_n)$.

\begin{definition}\label{cone gen g def} Let $2 \leq d \leq n-1$, $\bg = (g_{d+1}, \dots, g_n)$ be an $(n-d)$-tuple of functions in $C^{d+1}(I;\R^d)$ and $\Gamma_{\bg} $ denote the codimension $d-1$ cone in $\R^n$  parametrised by
\begin{equation*}
 (\vec{\lambda} ,s) \mapsto  \sum_{j=d+1}^n \lambda_j \cdot \Big( \begin{bmatrix}
    g_j(s) \\
    0
    \end{bmatrix} +  \vec{e}_j \Big) \quad \textrm{for } \vec{\lambda} = (\lambda_{d+1}, \dots, \lambda_n) \in \mathcal{R}_{n,d}' \textrm{ and } s \in I.
\end{equation*}
In this case, $\Gamma_{\bg} $ is referred to as the \textit{cone generated by $\bg$}.
\end{definition}

We now take into account the non-degeneracy condition established in \eqref{g a non deg 6}. Given $\ba = (a_{d+1}, \dots, a_n) \in \mathcal{R}_{n,d}'$ and $0 < \delta \ll 1$ consider the collection $\mathfrak{G}_{n,d}^{\ba}(\delta)$ of all $(n-d)$-tuples of functions
\begin{equation*}
    \bg = (g_{d+1}, \dots, g_n) \in [C^{d+1}(I;\R^d)]^{n-d}
\end{equation*}
with the property
\begin{equation}\label{g sum}
    g_{\ba} := \sum_{j=d+1}^n a_j\cdot g_j \in \mathfrak{G}_{d}(\delta),
\end{equation}
where $\mathfrak{G}_d(\delta)$ is the class of model curves introduced in \S\ref{moment red sec}.

In \eqref{g a non deg 6} we showed that the curves $g_{\ba}$ relevant to our study are \textit{non-degenerate}, which is a weaker condition than $g_{\ba} \in \mathfrak{G}_d(\delta)$ (provided $0 < \delta \ll 1$). However, by a localisation and scaling argument similar to that used in \S\ref{moment red sec}, we will always be able to assume the condition \eqref{g sum} holds in what follows (see the proof of Lemma~\ref{2 cone rescaling lemma} for details of the rescaling).

 Given $\bg \in \mathfrak{G}_{n,d}^{\ba}(\delta)$, $s \in [-1,1]$ and $0 < r \leq 1$, define the $n \times n$ matrix $[\bg]_{\ba, s,r}$ as in \eqref{Frenet box 3}; that is,
\begin{equation}\label{plate matrix}
    [\bg]_{\ba, s,r} :=
    \begin{pmatrix}
    [g_{\ba}]_{s,r} & \bg(s) \\
    0 & \mathrm{I}_{n-d}
    \end{pmatrix}.
\end{equation}
In view of \eqref{Frenet box 4}, one wishes to study decoupling with respect to the plates
\begin{equation*}
  \theta(s;r) :=  [\bg]_{\ba, s,r} \big( [-2,2]^n \big) \cap \mathcal{R}_{n,d}. 
\end{equation*}
In some cases it will be useful to highlight the choice of function $\bg$ by writing $\theta(\bg; s;r)$ for $\theta(s;r)$. Note that each of these plates lies in an $r$-neighbourhood of the cone $\Gamma_{\bg} $. We think of the union of all plates $\theta(s;r)$ as $s$ varies over the domain $[-1,1]$ as forming an anisotropic $r$-neighbourhood of $\Gamma_{\bg} $, similar to the situation for curves described in \S\ref{BDG subsec}.

Rather than work with the $\theta(s;r)$ directly, certain truncated versions are considered. 

\begin{definition}\label{a plate definition} For $0 < r \leq 1$, $\ba = (a_{d+1}, \dots, a_n) \in \mathcal{R}_{n,d}'$ and $K \geq 1$ an \textit{$(\ba,K)$-truncated $r$-plate} for $\Gamma_{\bg} $ is a set of the form
\begin{equation*}
    \theta^{\, \ba, K}(s;r) := [\bg]_{\ba, s,r} \big( [-2,2]^n \big) \cap Q(\ba,K^{-1}) 
\end{equation*}
for some $s \in I$ and
\begin{equation*}
 Q(\ba,K^{-1}) :=   \big\{\xi \in \hat{\R}^n : |\xi_j - a_j|\leq K^{-1} \textrm{ for $d+1 \leq j \leq n$} \big\}.
\end{equation*}
\end{definition}

\begin{definition}A collection $\Theta^{\, \ba,K}(r)$ of $(\ba,K)$-truncated $r$-plates is an \textit{$(\ba, K)$-truncated plate decomposition for $\bg$} if it consists of $\theta^{\, \ba,K}(\bg;s;r)$ for $s$ varying over a $r$-net in $I$.
\end{definition}

Theorem~\ref{Frenet decoupling theorem} is a consequence of the following decoupling inequality for cones $\Gamma_{\bg} $.

\begin{proposition}\label{cone dec prop} Let $2 \leq d \leq n-1$ and $\varepsilon > 0$. There exists some integer $K \geq 1$ such that for all $0 < r \leq 1$, $\ba \in \mathcal{R}_{n,d}'$ and $\bg \in \mathfrak{G}_{n,d}^{\ba}(\delta)$ for  $0 \leq \delta \ll 1$ the following holds. If $\Theta^{\,\ba,K}(r)$ is an $(\ba,K)$-truncated $r$-plate decomposition for $\Gamma_{\bg} $ and $2 \leq p \leq d(d+1)$, then
\begin{equation*}
    \big\|\sum_{\theta \in \Theta^{\,\ba,K}(r)} f_{\theta} \big\|_{L^p(\R^{n})} \lesssim_{\varepsilon} r^{-\varepsilon} \Big(\sum_{\theta \in \Theta^{\,\ba,K}(r)} \|f_{\theta}\|_{L^p(\R^{n})}^2\Big)^{1/2}
\end{equation*}
holds for any tuple of functions $(f_{\theta})_{\theta \in \Theta^{\,\ba,K}(r)}$ satisfying $\supp \hat{f}_{\theta} \subseteq \theta$.
\end{proposition}

Proposition~\ref{cone dec prop} follows from the Bourgain--Demeter--Guth result (namely, Theorem~\ref{BDG theorem}) via an argument from \cite{BD2015}, where decoupling estimates for the light cone in $\R^n$ were obtained as a consequence of decoupling estimates for the paraboloid in $\R^{n-1}$. The key observation is that, at suitably small scales, the cone $\Gamma_{\bg} $ can be approximated by a cylinder over the curve $g$. This approximation is only directly useful for relatively large $r$ values, but rescaling and induction-on-scale arguments allow one to leverage this observation in the small $r$ setting. Arguments of this kind originate in \cite{PS2007} and have been used repeatedly in the context of decoupling theory: see, for instance, \cite{BHS2020, GO2020, GZ2020, Oh2018}. 

The details of the proof of Proposition~\ref{cone dec prop} are postponed until \S\ref{decoupling proof subsection} below. In the following subsection, we show that Proposition~\ref{cone dec prop} implies Theorem~\ref{Frenet decoupling theorem}.




\subsection{Relating the decoupling regions} Theorem~\ref{Frenet decoupling theorem} may now be deduced as a consequence of Proposition~\ref{cone dec prop} using the geometric observations from \S\ref{geo obs sec}.

\begin{proof}[Proof of Theorem~\ref{Frenet decoupling theorem}, assuming Proposition~\ref{cone dec prop}] First note that it suffices to show the desired decoupling inequality in the restricted range $2 \leq p \leq d(d+1)$; the estimate for the remaining range $d(d+1) \leq p \leq \infty$ then follows by an interpolation argument and a trivial estimate for $p=\infty$.

Let $\gamma \in \mathfrak{G}_d(\delta)$ for $0 \leq \delta \ll 1$. As previously noted, we may restrict attention to the portion of $\gamma$ over $I = [-\delta, \delta]$ so that the Frenet vectors satisfy \eqref{Frenet loc}. Fix $2 \leq d \leq n-1$, $0 < r \leq 1$ and $\mathcal{P}_{d-1}(r)$ a Frenet box decomposition of $\gamma$. 

Define $\bg = (g_{d+1}, \dots, g_n)$ as in \S\ref{geo obs sec} so that the $g_{\mathbf{a}}$ are non-degenerate. Let $\varepsilon > 0$ be given and take $K \geq 1$ an integer satisfying the properties described in Proposition~\ref{cone dec prop}.   

Let $(f_{\pi})_{\pi \in \mathcal{P}_{d-1}(r)}$ be a tuple of functions satisfying the Fourier support hypothesis from the statement of Theorem~\ref{Frenet decoupling theorem}. If $\pi = \pi_{d-1,\gamma}(s;r) \in \mathcal{P}_{d-1}(r)$, then, recalling \eqref{Frenet box 4}, we have
\begin{equation}\label{Frenet dec 1}
    \supp \hat{f}_{\pi} \subseteq \pi_{d-1,\gamma}(s;r) \subseteq\bigcap_{\ba \in \mathcal{R}'_{n,d}} [\bg]_{\ba,s,Cr}\big([-2,2]^n\big) \cap \mathcal{R}_{n,d}.
\end{equation}

The frequency domain is decomposed according to the $Q(\ba,K^{-1})$ from Definition~\ref{a plate definition}. In particular, let
\begin{equation*}
    \mathcal{R}'_{n,d}(K) := K^{-1} \Z^{n-d} \cap \mathcal{R}'_{n,d}
\end{equation*} 
so that the sets $Q(\ba,K^{-1})$ for $\ba \in \mathcal{R}'_{n,d}(K)$ are finitely-overlapping and cover of $\mathcal{R}_{n,d}$. Form a smooth partition of unity $(\psi_{ \ba, K^{-1}})_{\ba \in \mathcal{R}'_{n,d}(K)}$ adapted to the sets $Q(\ba,K^{-1})$ and define the frequency projection operators $P_{\,\ba}$ via the Fourier transform by
\begin{equation*}
    \big( P_{\,\ba}f \big)\;\widehat{}\; := \psi_{\ba,K^{-1}} \cdot \hat{f}.
\end{equation*}
These operators are bounded on $L^p$ for $1 \leq p \leq \infty$ uniformly in $\ba$ and $K$ and, furthermore,
\begin{equation*}
    f_{\pi} = \sum_{\ba \in \mathcal{R}'_{n,d}(K)} P_{\,\ba}f_{\pi} \qquad \textrm{for all $\pi \in \mathcal{P}_{d-1}(r)$.}
\end{equation*}
Since $\# \mathcal{R}'_{n,d}(K) \lesssim_{n,\delta,\varepsilon} 1$, by the triangle inequality and the $L^p$ boundedness of the $P_{\,\ba}$, it suffices to show that
\begin{equation}\label{Frenet dec 2}
    \big\|\sum_{\pi \in \mathcal{P}_{d-1}(r)} P_{\,\ba}f_{\pi} \big\|_{L^p(\R^{n})} \lesssim_{n,\delta,\varepsilon} r^{-(1/2-1/p)-\varepsilon} \Big(\sum_{\pi \in \mathcal{P}_{d-1}(r)} \|P_{\,\ba}f_{\pi}\|_{L^p(\R^{n})}^p\Big)^{1/p}
\end{equation}
uniformly in $\ba \in \mathcal{R}'_{n,d}(K)$. However, recalling \eqref{Frenet dec 1}, each function $P_{\,\ba}f_{\pi}$ has frequency support in the set 
\begin{equation*}
  \theta^{\,\ba,K}(s,C r) = [\bg]_{\ba,s,C r}\big([-2,2]^n\big) \cap Q(\ba,K^{-1})
\end{equation*}
and so an $\ell^2$ version of \eqref{Frenet dec 2} follows as a consequence of Proposition~\ref{cone dec prop}.\footnote{Strictly speaking, Proposition~\ref{cone dec prop} requires the additional hypothesis $\bg \in \mathfrak{G}_{n,d}^{\ba}(\delta)$. However, by a rescaling argument (see the proof of Lemma~\ref{2 cone rescaling lemma}), the decoupling result generalises to arbitrary $\bg$ for which $g_{\mathbf{a}}$ is non-degenerate (albeit no longer with a uniform constant).}  The desired $\ell^p$-decoupling \eqref{Frenet dec 2} follows by applying H\"older's inequality to the $\ell^2$-sum. 
\end{proof}




\subsection{Proof of Proposition~\ref{cone dec prop}}\label{decoupling proof subsection} It remains to prove the decoupling Proposition~\ref{cone dec prop}. This is achieved using the argument outlined at the end of \S\ref{cone dec subsec}.

\begin{definition}[Decoupling constant] For $2 \leq d \leq n-1$, $0 <r \leq 1$, $p \geq 2$, $0 < \delta \ll 1$, $\ba \in \mathcal{R}_{n,d}'$ and $K \geq 1$ let $\mathfrak{D}_{n,d}^{\ba}(K;r)$ denote the infimum over all $C \geq 1$ for which
\begin{equation*}
    \big\|\sum_{\theta \in \Theta^{\,\ba,K}(r)} f_{\theta} \big\|_{L^p(\R^{n})} \leq C \Big(\sum_{\theta \in \Theta^{\,\ba,K}(r)} \|f_{\theta}\|_{L^p(\R^{n})}^2\Big)^{1/2}
\end{equation*}
 holds whenever:
\begin{enumerate}[i)]
    \item $\Theta^{\,\ba,K}(r)$ is an $(\ba,K)$-truncated $r$-plate decomposition for $\Gamma_{\bg} $ for some $\bg \in \mathfrak{G}_{n,d}^{\ba}(\delta)$,
    \item$(f_{\theta})_{\theta \in \Theta^{\,\ba,K}(r)}$ is a tuple of functions satisfying $\supp \hat{f}_{\theta} \subseteq \theta$.
\end{enumerate} 
\end{definition}

Thus, in this notation, Proposition~\ref{cone dec prop} states that  for all $\varepsilon > 0$ there exists some $K \geq 1$, depending only on $n$ and $\varepsilon$, such that 
\begin{equation}\label{cone dec 1}
    \mathfrak{D}_{n,d}^{\ba}(K; r) \lesssim_{\varepsilon} r^{-\varepsilon} \qquad \textrm{for all $\ba \in \mathcal{R}_{n,d}'$}
\end{equation}

\begin{remark} The definition of the decoupling constants also depends on $p$ and $\delta$ but, for simplicity, these parameters are omitted in the notation. 
\end{remark}

In conjunction to Theorem~\ref{BDG theorem}, one needs a simple scaling lemma.

\begin{lemma}[Generalised Lorentz rescaling]\label{2 cone rescaling lemma} If $0 < r < \rho < 1$, then
\begin{equation*}
    \mathfrak{D}_{n,d}^{\ba}(K;r) \lesssim \mathfrak{D}_{n,d}^{\ba}(K; \rho) \mathfrak{D}_{n,d}^{\ba}(K; r/\rho).
\end{equation*}
\end{lemma}

Temporarily assuming this result, Proposition~\ref{cone dec prop} follows by a simple induction-on-scale argument. 

\begin{proof}[Proof of Proposition~\ref{cone dec prop}] Let $\varepsilon > 0$, $0 < \delta \ll 1$ and $\ba = (a_{d+1}, \dots, a_n) \in \mathcal{R}_{n,d}'$ be given. Henceforth, $K = K(\varepsilon) \geq 1$ is thought of as a fixed number, depending only on $n$ and $\varepsilon$, which is chosen sufficiently large to satisfy the forthcoming requirements of the proof. It will be shown, by an induction-on-scale in the $r$ parameter, that \eqref{cone dec 1} holds for all $0 < r \leq 1$.

If $(100 K)^{-d} < r \leq 1$, then it follows from the triangle and Cauchy--Schwarz inequalities that 
\begin{equation*}
    \mathfrak{D}_{n,d}^{\ba}(K; r) \leq \mathbf{C}(\varepsilon)
\end{equation*}
for some constant $\mathbf{C}(\varepsilon) \geq 1$ depending only on $n$ and $\varepsilon$. This serves as the base case of an inductive argument.\smallskip 

It remains to establish the inductive step. To this end, fix some $0 < r \leq (100K)^{-d}$ and assume the following holds. \medskip

\noindent \textbf{Induction hypothesis.} If $r_{\circ} \geq 2 r$, then $\mathfrak{D}_{n,d}^{\ba}(K; r_{\circ}) \leq \mathbf{C}(\varepsilon) r_{\circ}^{-\varepsilon}$.\medskip

Given $0 < r < \rho < 1/2$, one may combine the generalised Lorentz rescaling lemma with the induction hypothesis to conclude that
\begin{equation}\label{cone dec 2}
    \mathfrak{D}_{n,d}^{\ba}(K;r) \lesssim \mathfrak{D}_{n,d}^{\ba}(K; \rho) \mathfrak{D}_{n,d}^{\ba}(K; r/\rho) \leq \mathbf{C}(\varepsilon) \rho^{\varepsilon} r^{-\varepsilon}\mathfrak{D}_{n,d}^{\ba}(K; \rho).
\end{equation}

Fix $\rho := K^{-1/d}$. Favourable bounds for $\mathfrak{D}_{n,d}^{\ba}(K; \rho)$ can be obtained in this case via an appeal to Theorem~\ref{BDG theorem}. Let $\mathrm{proj}_d \,\colon \hat{\R}^n \to \hat{\R}^d$ denote the orthogonal projection onto the coordinate plane spanned by $\vec{e}_1, \dots, \vec{e}_d$. The key observation is that any $(\ba,K)$-truncated $\rho$-plate $\theta^{\,\ba,K}(\bg; s;\rho)$ on $\Gamma_{\bg} $ essentially projects into a $\rho$-slab $\alpha(g_{\ba}; s; \rho)$ on $g_{\ba} = \sum_{j = d+1}^n a_j \cdot g_j$ under this mapping, where $\alpha(g_{\ba}; s; \rho)$ is as defined in Definition~\ref{slab def}. In particular,
\begin{equation}\label{cone dec 3}
    \mathrm{proj}_d\, \theta^{\,\ba,K}(\bg; s;\rho) \subseteq \alpha(g_{\ba}; s; C\rho)
\end{equation}
for some choice of constant $C \geq 1$ depending only on $n$. To see this, fix $\xi \in \theta^{\,\ba,K}(\bg; s;\rho)$ and note that $\xi = [\bg]_{\,\ba,s,\rho} \cdot \eta$ for some $\eta \in [-2,2]^n$ whilst
\begin{equation}\label{cone dec 4}
    |\xi_j - a_j| \leq 1/K\textrm{ for $d+1 \leq j \leq n$,}
\end{equation}
by Definition~\ref{a plate definition}. By the definition of the matrix $[\bg]_{\,\ba,s,\rho}$ in \eqref{plate matrix}, it follows that
\begin{align*}
\xi' &= [g_{\ba}]_{s,\rho}\cdot\eta' + \sum_{j=d+1}^n \eta_jg_j(s), \\
\xi_j &= \eta_j \qquad \textrm{for $d+1 \leq j \leq n$}
\end{align*}
where $\xi' := \mathrm{proj}_d\, \xi$ and $\eta' := \mathrm{proj}_d\, \eta \in [-2,2]^d$. In particular,
\begin{equation*}
    \xi' - g_{\ba}(s) = [g_{\ba}]_{s,\rho} \Big(\eta' + [g_{\ba}]_{s,\rho}^{-1} \cdot \sum_{j =d+1}^n (\eta_j - a_j)g_j(s) \Big)
\end{equation*}
and, for the choice of $\rho = K^{-1/d}$ specified above,
\begin{align*}
\big|[g_{\ba}]_{s,\rho}^{-1} \cdot \sum_{j =d+1}^n (\eta_j - a_j)g_j(s) \big| &\leq  \|[g_{\ba}]_{s,\rho}^{-1} \|_{\mathrm{op}}\cdot \sum_{j =d+1}^n |\xi_j -  a_j||g_j(s)|\\
&\lesssim \rho^{-d}K^{-1} \leq 1.
\end{align*}
The second inequality follows from the hypothesis $g_{\ba} \in \mathfrak{G}_d(\delta)$ from \eqref{g sum}, which implies that $\|[g_{\ba}]_{s,\rho}^{-1} \|_{\mathrm{op}} \lesssim \rho^{-d}$ (with a uniform constant), and the condition \eqref{cone dec 4}. Recalling Definition~\ref{slab def}, it follows that $\xi' \in \alpha(g_{\ba}; s;C\rho)$, as claimed. 

Let $\Theta^{\,\ba,K}(\rho)$ be an $(\ba,K)$-truncated $\rho$-plate decomposition of $\Gamma_{\bg} $ and $(f_{\theta})_{\theta \in \Theta^{\,\ba,K}(\rho)}$ be a tuple of functions satisfying $\supp \hat{f}_{\theta} \subseteq \theta$. For any $2 \leq p \leq d(d+1)$ and $\tilde{\varepsilon} > 0$, by \eqref{cone dec 3} and Theorem~\ref{BDG theorem} it follows that
\begin{equation*}
    \big\|\sum_{\theta \in \Theta^{\,\ba,K}(\rho)} f_{\theta}(\,\cdot\,,x'') \big\|_{L^p(\R^{d})} \lesssim_{\tilde{\varepsilon}} \rho^{-\tilde{\varepsilon}} \Big(\sum_{\theta \in \Theta^{\,\ba,K}(\rho)} \|f_{\theta}(\,\cdot\,,x'')\|_{L^p(\R^{d})}^2\Big)^{1/2}
\end{equation*}
for all $x'' \in \R^{n-d}$. Taking the $L^p$-norm of both sides of this inequality with respect to $x''$ and using Minkowski's inequality to bound the resulting right-hand side, one deduces that
\begin{equation}\label{cone dec 5}
\mathfrak{D}_{n,d}^{\ba}(K; \rho) \lesssim_{\tilde{\varepsilon}} \rho^{-\tilde{\varepsilon}}.
\end{equation}

Taking $\tilde{\varepsilon} := \varepsilon/2$ in \eqref{cone dec 5} and substituting this inequality into \eqref{cone dec 2}, one deduces that 
\begin{equation*}
\mathfrak{D}_{n,d}^{\ba}(K;r) \leq \big(C_{\varepsilon} \rho^{\varepsilon/2}\big) \mathbf{C}(\varepsilon) r^{\varepsilon},
\end{equation*}
where the $C_{\varepsilon}$ factor arises from the various implied constants in the above argument. Thus, if $K$ is chosen from the outset to be sufficiently large, depending only on $n$ and $\varepsilon$, then
\begin{equation*}
    C_{\varepsilon}\rho^{\varepsilon/2} = C_{\varepsilon}K^{-\varepsilon /2d} \leq 1
\end{equation*} 
and the induction closes.
\end{proof}
It remains to prove the Lorentz rescaling lemma. Before presenting the argument, it is useful to introduce an extension of Definition~\ref{rescaled curve def} to tuples of curves $\bg$.

\begin{definition}\label{definition rescaled curve tuple} Let $\bg = (g_{d+1}, \dots, g_n) \in \mathfrak{G}_{n,d}^{\ba}(\delta)$ and $g_{\ba} := \sum_{j = d+1}^n a_j \cdot g_j$, as in \eqref{g sum} and \eqref{plate matrix}. Define the $({\ba};b,\rho)$-rescaling of $\bg$ to be the $(n-d)$-tuple
\begin{equation*}
   \bg_{\ba,b,\rho} = (g_{\ba, b,\rho, d+1}, \dots, g_{\ba, b,\rho,n}) \in [C^{d+1}(I, \R^d)]^{n-d} 
\end{equation*} 
given by
\begin{equation}\label{rescaled curve tuple}
    \bg_{\ba,b,\rho}(t) = [g_{\ba}]_{b,\rho}^{-1}\big( \bg(b+\rho t) - \bg(b) \big).
\end{equation}
Here $\bg_{\ba,b,\rho}(t)$ and $\bg(t)$ are understood to be the $d \times (n-d)$ matrices whose columns are the component functions of $\bg_{\ba,b,\rho}$ and $\bg$, respectively, evaluated at $t \in I$.
\end{definition}

As a consequence of this definition, the function 
\begin{equation*}
 g_{\ba,b,\rho}  := \sum_{j=d+1}^n a_j \cdot g_{\ba,b,\rho,j}   
\end{equation*}
 is precisely the $(b,\rho)$-rescaling of $g_{\ba} := \sum_{j = d+1}^n a_j \cdot g_j$. Thus, the notation $g_{\ba,b,\rho}$ used here is consistent in the sense that $g_{\ba,b,\rho} = (g_{\ba})_{b,\rho}$ and, furthermore, since $ g_{\ba,b,\rho} := (g_{\ba,b,\rho})_{\ba}$, one has
\begin{equation}\label{rescaled plate matrix}
    [\bg_{\ba, b,\rho}]_{\ba, u,h} :=
    \begin{pmatrix}
    [g_{\ba,b,\rho}]_{u,h} & \bg_{\ba, b,\rho}(s) \\
    0 & \mathrm{I}_{n-d}
    \end{pmatrix}.
\end{equation}

\begin{proof}[Proof of Lemma~\ref{2 cone rescaling lemma}] Fix $\bg \in \mathfrak{G}_{n,d}^{\ba}(\delta)$, an $(\ba,K)$-truncated $r$-plate decomposition $\Theta^{\,\ba,K}(r)$ for $\Gamma_{\bg} $ and let $(f_{\theta})_{\theta \in \Theta^{\,\ba,K}(r)}$ be a tuple of functions satisfying  $\supp \hat{f}_{\theta} \subseteq \theta$. By a simple pigeonholing argument, there exists an $(\ba,K)$-truncated $\rho$-plate decomposition $\Theta^{\,\ba,K}(\rho)$ such that 
\begin{equation*}
    \big\| \sum_{\theta \in \Theta^{\,\ba,K}(r)} f_{\theta} \Big\|_{L^p(\R^n)} \lesssim \big\| \sum_{\theta' \in \Theta^{\,\ba,K}(\rho)} f_{\theta'} \Big\|_{L^p(\R^n)}
\end{equation*}
where
\begin{equation*}
f_{\theta'} := \sum_{\substack{\theta \in \Theta^{\,\ba,K}(r) \\ \theta \subset \theta'}} f_{\theta} \quad \textrm{for all $\theta' \in \Theta^{\,\ba,K}(\rho)$.}
\end{equation*}
Since $\supp \hat{f}_{\theta'} \subseteq \theta'$, by definition 
\begin{equation}\label{Lorentz 1}
    \big\| \sum_{\theta \in \Theta^{\,\ba,K}(r)} f_{\theta} \Big\|_{L^p(\R^n)} \lesssim \mathfrak{D}_{n,d}^{\ba}(K;\rho) \Big(\sum_{\theta' \in \Theta^{\,\ba,K}(\rho)}\|  f_{\theta'} \|_{L^p(\R^n)}^2 \Big)^{1/2}.
\end{equation}

The goal here is to show that
\begin{equation}\label{Lorentz 2}
    \|f_{\theta'}\|_{L^p(\R^n)} \lesssim \mathfrak{D}_{n,d}^{\ba}(K;r/\rho) \Big(\sum_{\substack{\theta \in \Theta^{\,\ba,K}(\rho)\\ \theta \subset \theta'}}\|  f_{\theta} \|_{L^p(\R^n)}^2 \Big)^{1/2}
\end{equation}
for each $\theta' \in \Theta^{\,\ba,K}(\rho)$. Indeed, once this is established, by combining \eqref{Lorentz 1} and \eqref{Lorentz 2} with the definition of $\mathfrak{D}_{n,d}^{\ba}(K;r)$, one deduces the desired result.

Fix an $(\ba,K)$-truncated $\rho$-plate $\theta(b;\rho) \in \Theta^{\,\ba,K}(\rho)$ and recall that 
\begin{equation*}
    \theta(b;\rho) = \big\{\xi \in \hat{\R}^n : ([\bg]_{\ba,b,\rho})^{-1}\xi \in [-2,2]^n\big\} \cap  Q(\ba,K^{-1})
\end{equation*}
for $ Q(\ba,K^{-1})$ as defined in Definition~\ref{a plate definition}. Note that the preimage of $\theta(b;\rho)$ under the $[\bg]_{\ba,b,\rho}$ mapping is the set
\begin{equation*}
   \big([\bg]_{\ba,b,\rho}\big)^{-1}\theta(b;\rho) = [-2,2]^n \cap  Q(\ba,K^{-1}).
\end{equation*}
 On the other hand, the $(\ba,K)$-truncated $r$-plate $\theta(s;r)\equiv \theta^{\,\ba,K}(s;r)$ is transformed under $\big([\bg]_{\ba,b,\rho}\big)^{-1}$ into
\begin{equation}\label{Lorentz 3}
   \big\{\xi \in \hat{\R}^n : \big([\bg]_{\ba,s,r}\big)^{-1} \cdot [\bg]_{\ba,b,\rho}\,\xi \in [-2,2]^n \big\} \cap  Q(\ba,K^{-1}). 
\end{equation}
The key observation is that
\begin{equation}\label{Lorentz 4}
    \big([\bg]_{\ba,s,r}\big)^{-1}\cdot [\bg]_{\ba,b,\rho} = \big([\bg_{\ba,b,\rho}]_{\ba,\frac{s-b}{\rho},\frac{r}{\rho}}\big)^{-1}
\end{equation}
so that \eqref{Lorentz 3} corresponds to an $(\ba,K)$-truncated $r/\rho$-plate for the cone generated over the rescaled curve tuple $\bg_{\ba, b,\rho}$. Once this established, \eqref{Lorentz 2} follows easily by a change of variable. Indeed, taking $\theta' = \theta(b,\rho)$ and defining the functions $\tilde{f}_{\theta'}$ and $\tilde{f}_{\theta}$ for $\theta \in \Theta^{\,\ba,K}(r)$ via the Fourier transform by
\begin{equation*}
    \big(\tilde{f}_{\theta'}\big)\;\widehat{}\; := \hat{f}_{\theta'} \circ [\bg]_{\ba,b,\rho} \qquad \textrm{and} \qquad \big(\tilde{f}_{\theta}\big)\;\widehat{}\; := \hat{f}_{\theta} \circ [\bg]_{\ba,b,\rho},
\end{equation*}
by a linear change of variable the desired inequality \eqref{Lorentz 2} is equivalent to 
\begin{equation*}
    \|\tilde{f}_{\theta'}\|_{L^p(\R^n)} \lesssim \mathfrak{D}_{n,d}^{\ba}(K;r/\rho) \Big(\sum_{\substack{\theta \in \Theta^{\,\ba,K}(\rho)\\ \theta \subset \theta'}}\|  \tilde{f}_{\theta} \|_{L^p(\R^n)}^2 \Big)^{1/2}.
\end{equation*}
However, since the preceding observations show that the $\tilde{f}_{\theta}$ are Fourier supported on $(\ba,K)$-truncated $r/\rho$-plates for the cone generated over the rescaled curve tuple $\bg_{\ba,b,\rho}$, this bound follows directly from the definition of the decoupling constant. 

To prove \eqref{Lorentz 4}, first note that it suffices to show 
\begin{equation*}
    [\bg]_{\ba,b,\rho}\cdot [\bg_{\ba, b,\rho}]_{\ba,\frac{s-b}{\rho},\frac{r}{\rho}} = [\bg]_{\ba,s,r}. 
\end{equation*}
Recalling \eqref{rescaled plate matrix} (taking $u := (s-b)/\rho$ and $h := r/\rho$) and carrying out the block matrix multiplication, this is equivalent to the pair of identities
\begin{subequations} 
\begin{align}\label{Lorentz 5}
        [g_{\ba}]_{b,\rho}\cdot [g_{\ba, b,\rho}]_{\frac{s-b}{\rho},\frac{r}{\rho}} &= [g_{\ba}]_{s,r}, \\
        \label{Lorentz 6}
        [g_{\ba}]_{b,\rho}\cdot \bg_{\ba, b,\rho}\big(\tfrac{s-b}{\rho}\big)+ \bg(b) &= \bg(s).
\end{align}
\end{subequations}
Note that \eqref{Lorentz 5} is an identification of $d \times d$ matrices, whilst \eqref{Lorentz 6} is an identification of $d \times (n-d)$ matrices.

Recall the definition of the matrix 
\begin{equation*}
    [g_{\ba, b,\rho}]_{x} =
    \begin{bmatrix}
    g_{\ba, b,\rho}^{(1)}(x) & \cdots & g_{\ba, b,\rho}^{(d)}(x)
    \end{bmatrix}.
\end{equation*}
From the discussion following Definition~\ref{definition rescaled curve tuple}, the curve $g_{\ba, b,\rho}$ is as defined in Definition~\ref{rescaled curve def} and, in particular, is given by
\begin{equation*}
    g_{\ba, b,\rho}(t) = [g_{\ba}]_{b,\rho}^{-1}\big(g_{\ba}(b + \rho t) - g_{\ba}(b)\big).
\end{equation*}
Combining these definitions with the chain rule, 
\begin{equation*}
    [g_{\ba,b,\rho}]_{x} = [g_{\ba}]_{b,\rho}^{-1} \cdot [g_{\ba}]_{b+\rho x,\rho} \qquad \textrm{for $x \in \R$ with $b+\rho x \in [-1,1]$.}
\end{equation*}
Taking $x = \tfrac{s-b}{\rho}$ and right multiplying the above display by $D_{r/\rho}$ immediately implies \eqref{Lorentz 5}. On the other hand, \eqref{Lorentz 6} follows directly from the definition \eqref{rescaled curve tuple}. 
\end{proof}

\appendix


\section{Reduction to a frequency localised estimate}\label{Sob-from-dy appendix}

Here we discuss the passage from frequency localised used in \S\ref{sec:bandlimited}. In particular, we fill in the details of the argument reducing Theorem~\ref{non-degenerate theorem} to Theorem~\ref{Sobolev theorem}. This follows very quickly by using a special case of a result proved in \cite{PRS2011}. 

\subsection{A Calder\'on--Zygmund estimate} For each $k\in \bbN$ we are given operators $T_k$ defined 
on Schwartz function $f\in \cS(\bbR^n)$
by
\begin{equation*}
    T_k f(x) := \int_{\R^n} K_k(x,y) f(y) \,\ud y
\end{equation*}
where each $K_k$ is a continuous and bounded kernel
(with no other quantitative assumptions). 
Let $\zeta\in \cS(\bbR^n)$, define  $\zeta_k := 2^{k n}\zeta(2^k\,\cdot\,)$
and set $P_k f:= \zeta_k \ast f.$

\begin{theorem}[\cite{PRS2011}] \label{thm:prs}   Let   $\eps>0$ and $1<p_0<p<\infty$. Assume for some $A \geq 1$ the operators $T_k$ satisfy 
\begin{equation*}
    \sup_{k>0} 2^{k/p_0}\|T_k\|_{L^{p_0}(\R^n)\to L^{p_0}(\R^n) } \le A, \,
\end{equation*}
\begin{equation*}
    \sup_{k>0} 2^{k/p}\|T_k\|_{L^p(\R^n)\to L^p(\R^n)} \le A\,.
\end{equation*}
Furthermore, assume that there exist $B \ge 1$ and   for each cube $Q$  
a measurable set $\cE_Q$ such that
\begin{equation*}
     |\cE_Q| \le  
B \max\{ \diam(Q)^{n-1}, \diam(Q)^n \} 
\end{equation*}
and such that, for every $k\in \bbN$ and   every cube $Q$ with $2^k \diam(Q)\ge 1$,
\Be\label{Linftyhyp}
 \sup_{x\in Q}\int_{\bbR^n\setminus \cE_Q}|K_k(x,y)| \,\ud y \le
A  \max\{ (2^k\diam(Q))^{-\eps}, 2^{-k\eps} \}.
\Ee
Then for every $r>0$ we have 
\begin{equation*}
 \Big\| \Big(\sum_{k=1}^{\infty} 2^{k r/p} | P_k T_k f_k|^r\Big)^{1/r} \Big\|_{L^p(\R^n)} \lc
 \Big(\sum_{k=1}^{\infty}\|f_k\|_{L^p(\R^n)}^p\Big)^{1/p}   
\end{equation*}
for any sequence of functions $(f_k)_{k=1}^\infty \in \ell^p(L^p)$,
where the implicit constant depends only on 
$A$, $B$, $r$, $\eps$, $p$, $p_0$, $n$ and $\zeta$. 
\end{theorem}

\subsection{Application} 
We consider a regular curve given by  $ t \mapsto \gamma(t) $, $t\in [0,1]$, and let 
$A_{\gamma} $ be as in \eqref{averaging operator}; that is, $A_{\gamma} f=\mu_\ga \ast f$. 

Let $\beta_\circ\in C^\infty_c (\hat \bbR^n)$ be, say, radial, supported in 
$\{1/2<|\xi|<1\} $ and $\beta_\circ (\xi)>0$ for $2^{-1/2} \le |\xi|\le 2^{1/2}$ and define $L_j f := \beta_\circ (2^{-j}D) f$. We make the assumption that 
 for some $p_0\ge 2$ we have
\Be \label{eqn:annuli} \|L_j A_{\ga}  f\|_{L^{p_0}(\R^n)} \le C 2^{-j/p_0} \|f\|_{L^{p_0}(\R^n)}, \quad j\ge 0. 
\Ee 
For a non-degenerate curve in $\bbR^4$ such inequalities were proved in the previous sections.

Theorem~\ref{thm:prs} facilitates the following reduction. 

\begin{proposition} \label{prop:Sobolev-from-dyadic}  Assumption \eqref{eqn:annuli}  implies that $ A_{\gamma}$ maps $L^p$ boundedly to $L^p_{1/p}$ for $p_0<p<\infty$.
\end{proposition}

This result can be used to complete the argument described in \S\ref{sec:bandlimited} and thereby reduce Theorem~\ref{non-degenerate theorem} to Theorem~\ref{Sobolev theorem}.

\begin{proof}[Proof of Proposition~\ref{prop:Sobolev-from-dyadic}]
Let $\Gamma := \{\gamma(t): t\in I\}$ where $I$ is a compact interval containing the support of $\chi$.
Let $Q$ be a cube with center $x_Q$ and define 
\begin{equation*}
  \cE_Q := \{y\in \bbR^n: \mathrm{dist}(y-x_Q, \Gamma) \le 10 \,\diam (Q)\}.  
\end{equation*}
Thus if $Q$ is small then $\cE_Q$ is a tubular neighborhood of $x_Q+\Gamma$ of width $O(\diam(Q))$.  It is not hard  to see that there is a constant $C$ (depending on $B$) such that 
\begin{equation*}
     \meas (\cE_Q) \le C\begin{cases}
\diam(Q)^{n-1} &\text{ if } \diam(Q)\le 1, 
\\
\diam(Q)^n &\text{ if } \diam(Q) \ge 1.
\end{cases} 
\end{equation*}

Let $\upsilon \in C^\infty_c(\bbR^n)$ be supported in $\{x:|x|<1/4\} $, with the  property that $\widehat \upsilon(\xi) >0$ on the support of $\beta_\circ$, $\widehat \upsilon(0)=0$ and $\nabla \widehat \upsilon(0)=0$.\footnote{To construct such a function  take any compactly supported real valued $u\in C^\infty_c(\R^n)$ with $\int_{\R^n} u=1$, form $u_C=C^d u(C\,\cdot\,)$ for sufficiently large $C$ to ensure $\widehat u_C>0$ on $\supp \beta_\circ$ and then take the Laplacian, $\upsilon=\Delta u_C$, to also get the condition $\widehat \upsilon(0)=0$.}
Let $\upsilon_k := 2^{k n} \upsilon(2^k\,\cdot\,)$ and 
define 
\begin{equation*}
    T_k f(x) := \upsilon_k*\mu_\ga*f \qquad \textrm{and}\qquad K_k(x,y) := \upsilon_k*\mu_\ga (x-y).
\end{equation*}
By the support properties of $\upsilon_k$ we have
\begin{equation*}
   K_k(x,y) =0 \qquad \textrm{ if $x\in Q$, $y\in \bbR^n \setminus \cE_Q$ and $2^k \diam(Q)>1$} 
\end{equation*}
and thus \eqref{Linftyhyp} holds trivially.

We claim that the assumption \eqref{eqn:annuli}  
also implies
\begin{equation*}
   \sup_{k \in \N} 2^{k/p_0} \|T_k\|_{L^{p_0}(\R^n)\to L^{p_0}(\R^n) }  <\infty.
\end{equation*}
To see this, choose 
$\widetilde  \beta_\circ\in C^\infty_c(\hat{\R}^n)$ supported in 
$\{\xi:1/2<|\xi|<2\}$ such that 
\begin{equation*}
    \sum_{j\in\bbZ} \beta_\circ(2^{-j}\xi)
\widetilde \beta_\circ(2^{-j}\xi) =1
\end{equation*}
for all $\xi\neq 0$ and define $\widetilde L_j := \widetilde \beta_\circ(2^{-j} D) $. Thus,
\begin{equation*}
\|T_k\|_{L^{p_0}\to L^{p_0}} \le \sum_{j\in \bbZ} 
\|L_j\widetilde L_jT_k\|_{L^{p_0} \to L^{p_0}}
\le \sum_{j\in \bbZ} \|L_jA_\gamma\|_{L^{p_0} \to L^{p_0}} \|\widetilde L_j\upsilon_k\|_1.
\end{equation*}
By straightforward calculations, using scaling and the cancellation and Schwartz  properties of $\upsilon$ and $\cF^{-1}[\beta_\circ]$, one has
$\|\widetilde L_j\upsilon_k\|_1= O(2^{-|j-k|})$ for all $j\in \bbZ$. Using this together with the hypothesis
\eqref{eqn:annuli}, we get 
\begin{equation*}
    \|T_k\|_{L^{p_0}\to L^{p_0}} \leq \sum_{j\in \bbZ} \min \{2^{-j/p}, 1\} 2^{-|j-k|} \lc 2^{-k/p};
\end{equation*}
note for $j < 0$ we use the trivial bounds $\| L_j A_\gamma f\|_{p} \lesssim \| f\|_p$.
Since $\|T_k\|_{L^\infty\to L^\infty}=O(1)$, interpolation therefore yields
\begin{equation*}
    \sup_{k\in \N} 2^{k/p} \|T_k\|_{L^{p}(\R^n)\to L^{p}(\R^n) }  <\infty \qquad \textrm{for all $p_0\le p\le \infty$.}
\end{equation*}

Let $\beta\in C^\infty_c(\hat{\R}^n)$ be supported in $\{\xi:1/4<|\xi|<4\}$ such that $\beta_\circ\beta=\beta_\circ$ and $p_0 <p<\infty$. Defining $f_k := \beta(2^{-k} D) f$, by Theorem~\ref{thm:prs} we obtain 
 \begin{align}\notag
 \Big\| \Big(\sum_{k=1}^{\infty} [2^{k /p} | \beta_\circ(2^{-k} D) A_{\ga} f|]^r\Big)^{1/r} \Big\|_{L^p(\R^n)} &=
 \Big\| \Big(\sum_{k=1}^{\infty} 2^{k r/p} | P_k T_k f_k|^r\Big)^{1/r} \Big\|_{L^p(\R^n)}
\\  
\notag
& \lc \Big(\sum_{k=1}^{\infty}\|f_k\|_{L^p(\R^n)}^p\Big)^{\frac 1p} 
 \\
 \label{TLBes}
 &\lc
 \Big(\sum_{k=1}^{\infty}\|\beta(2^{-k}D) f\|_{L^p(\R^n)}^p\Big)^{\frac 1p}.
\end{align} 

 Since $\ell^r\hookrightarrow \ell^2$ for $r\le 2$ and
 $\ell^2\hookrightarrow \ell^p$ for $2\le p$, we deduce that
 \Be\notag
\Big\| \Big(\sum_{k=1}^{\infty} [2^{k /p} | \beta_\circ(2^{-k} D) A_{\ga} f|]^2\Big)^{1/2} \Big\|_{L^p(\R^n)}\lc 
 \Big\| \Big(\sum_{k=1}^{\infty}|\beta(2^{-k} D) f|^2\Big)^{1/2}\Big\|_{L^p(\R^n)}.
 \Ee
This, together with the obvious low frequency $L^p$ estimates, yield the asserted Sobolev bound via Littlewood--Paley inequalities.
\end{proof}

\begin{remark} Using Besov and Triebel-Lizorkin spaces one gets from \eqref{TLBes} the stronger inequality $$A_{\ga}: B^0_{p,p} \to F^{1/p}_{p,r} $$ for all $r>0$.
\end{remark}


\section{Derivative bounds for implicitly defined functions}

The following lemma is a particular instance of a more general lemma on derivative bounds for implicitly defined functions found in \cite[Appendix C]{BGHS-helical}.

\begin{lemma}\label{imp deriv lem}
Let $\Omega \subseteq \R^n$ be an open set, $I \subseteq \R$ an open interval, $g \colon I \to \R^n$ a $C^\infty$ mapping and $y \colon \Omega \to I$ a $C^{\infty}$ mapping such that
\begin{equation*}
    \inn{g \circ y(\bm{x})}{\bm{x}}=0 \qquad \textrm{for all $\bm{x} \in \Omega$.}
\end{equation*}
For $\be \in S^{n-1}$ let $\nabla_{\be}$ denote the directional derivative operator with respect to $\bm{x}$ in the direction of $\be$. Suppose $A,M_1,M_2>0$ are constants such that
\begin{equation}\label{multi imp deriv 2}
    \left\{\begin{array}{rcl}
    |\inn{g'\circ y(\bm{x})}{\bm{x}}| &\geq& A M_2, \\[2pt]
    |\inn{g^{(N)}\circ y(\bm{x})}{\bm{x}}| &\lesssim_N& A M_2^{N}  \\[2pt]
    |\inn{g^{(N)}\circ y(\bm{x})}{\be}| &\lesssim_N& A M_1 M_2^{N} 
\end{array} \right. \qquad\textrm{for all $N \in \N$ and all $\bm{x} \in \Omega$} 
\end{equation}
Then the function $y$ satisfies
\begin{equation}\label{multi imp der bound}
    |\nabla_{\be}^{N} y(\bm{x})| \lesssim_N M_1^N M_2^{-1} \qquad \textrm{for all $\bm{x} \in \Omega$ and all $N \in \N_0$.}
\end{equation}
Furthermore, for any $C^\infty$ function $h \colon I \to \R^n$ for which there exists some constant $B>0$ satisfying
\begin{equation}\label{multi imp deriv 3}
    \left\{\begin{array}{rcl}
    |\inn{h^{(N)}\circ y(\bm{x})}{\bm{x}}| &\lesssim_N& B M_2^{N}  \\[2pt]
    |\inn{h^{(N)}\circ y(\bm{x})}{\be}| &\lesssim_N& B M_1 M_2^{N} 
\end{array} \right. \qquad\textrm{for all $N \in \N$ and all $\bm{x} \in \Omega$,} 
\end{equation}
one has
\begin{equation}\label{multi Faa di Bruno eq 2}
    \big|\nabla_{\be}^N \inn{h \circ y(\bm{x})}{\bm{x}}\big| \lesssim_N B M_1^N  \qquad \textrm{for all $\bm{x} \in \Omega$ and all $N \in \N$.}
\end{equation}
\end{lemma}

The following example illustrates how Lemma~\ref{imp deriv lem} is applied in practice in this article. 

\begin{example}[Application to the case $J=3$]\label{deriv ex}
Let $\gamma \in \mathfrak{G}_4(\delta_0)$, and $\theta: \hat{\R}^4 \backslash\{0\} \to I_0$ satisfying
\begin{equation*}
    \inn{\gamma'' \circ \theta (\xi)}{\xi}=0.
\end{equation*}
We apply the previous result with $g=\gamma''$ and $h=\gamma'$. If $B \leq A$ the conditions \eqref{multi imp deriv 2} and \eqref{multi imp deriv 3} read succinctly as
\begin{equation*}
    \left\{\begin{array}{rcl}
    |\inn{\gamma^{(3)}\circ \theta(\xi)}{\xi}| &\geq& A M_2, \\[2pt]
    |\inn{\gamma^{(1+N)}\circ \theta(\xi)}{\xi}| &\lesssim_N& B M_2^{N}  \\[2pt]
    |\inn{\gamma^{(1+N)}\circ \theta(\xi)}{\be}| &\lesssim_N& B M_1 M_2^{N} 
\end{array} \right.
\qquad\textrm{for all $N \in \N$ and all $\xi \in \Omega \subset \hat{\R}^4 \backslash \{0\}$},
\end{equation*}
which imply
\begin{equation*}
    |\nabla_{\bm{e}}^N \theta(\xi)| \lesssim_N M_1^N M_2^{-1} \quad \text{ and } \quad |\nabla_{\bm{e}}^N \inn{\gamma' \circ \theta(\xi)}{\xi}| \lesssim_N BM_1^N.
\end{equation*}
for all $N \in \N$ and all $\xi \in \Omega \subset \hat{\R}^4 \backslash \{0\}$.

The applications in the different cases $J=4$ are similar, with the choices $(g,h)=(\gamma^{(3)}, \gamma'')$, $(g,h)=(\gamma^{(3)}, \gamma')$ and  $(g,h)=(\gamma'', \gamma')$.
\end{example}




\section{Integration-by-parts lemmata}

\subsection{Non-stationary phase} For $a \in C^{\infty}_c(\R)$ supported in an interval $I \subset \R$ and $\phi \in C^{\infty}(I)$, define the oscillatory integral
\begin{equation*}
    \mathcal{I}[\phi, a] := \int_{\R} e^{i \phi(s)} a(s)\,\ud s.
\end{equation*}
The following lemma is a standard application of integration-by-parts.

\begin{lemma}[Non-stationary phase]\label{non-stationary lem} Let $R \geq 1$ and $\phi, a$ be as above. Suppose that for each $j \in \N_0$ there exist constants $C_j \geq 1$ such that the following conditions hold on the support of $a$:
\begin{enumerate}[i)]
    \item $|\phi'(s)| > 0$,
    \item $|\phi^{(j)}(s)| \leq C_j R^{-(j-1)}|\phi'(s)|^j\,\,$ for all $j \geq 2$,
    \item $|a^{(j)}(s)| \leq C_j R^{-j}|\phi'(s)|^j\,\,$ for all $j \geq 0$.
\end{enumerate}
Then for all $N \in \N_0$ there exists some constant $C(N)$ such that
\begin{equation*}
    |\mathcal{I}[\phi, a]| \leq C(N) \cdot |\supp a| \cdot R^{-N}.
\end{equation*}
Moreover, $C(N)$ depends on $C_1, \dots, C_N$ but is otherwise independent of $\phi$ and $a$ and, in particular, does not depend on $r$. 
\end{lemma}

A detailed proof of this lemma can be found in \cite[Appendix D]{BGHS-helical}.

\subsection{Kernel estimates} The following lemma, which is used to obtain $L^{\infty}$ bounds for the multipliers, is based on integration-by-parts in the $\xi$ variable. 

\begin{lemma}\label{gen Linfty lem} Let $a \in C_c^{\infty}(\hat{\R}^n \times I_0)$, $\sigma > 0$, $\lambda_j > 0$ for $1 \leq j \leq n$ and $\{\bm{v}_1, \dots, \bm{v}_n\}$ be an orthonormal basis of $\R^n$. Suppose the following conditions hold:
\begin{enumerate}[i)]
    \item $|\{s \in \R : (\xi;s) \in \supp a \textrm{ for some } \xi \in \hat{\R}^n \}| \leq \sigma$,
    \item $\xisupp a \subseteq \{\xi \in \hat{\R}^n : |\inn{\xi}{v_j}| \leq \lambda_j \textrm{ for $1 \leq j \leq n$}\}$,
    \item $|\nabla_{\bm{v}_j}^N a(\xi;s)| \lesssim_N \lambda_j^{-N}$ for all $(\xi;s) \in \hat{\R}^n \times \R$, $1 \leq j \leq n$ and $N \in \N_0$. 
\end{enumerate}
 Then 
\begin{equation*}
 \|m[a]\|_{M^{\infty}(\R^n)} \lesssim \sigma.
\end{equation*}
\end{lemma}

Here $\nabla_{\bm{v}} := \bm{v} \cdot \nabla$ denotes the directional derivative with respect to $\xi$ in the direction of $\bm{v} \in S^{n-1}$.

\begin{proof}[Proof of Lemma~\ref{gen Linfty lem}] For $f \in \Scn$ we have $m[a](D)f = K[a] \ast f$ where the kernel $K[a]$ is given by
\begin{equation*}
    K[a](x) = \int_{\R} \mathcal{F}^{-1}a(\,\cdot\,;s)(x + \gamma(s)) \chi(s)\,\ud s.
\end{equation*}
Here $\mathcal{F}^{-1}$ denotes the inverse Fourier transform in the $\xi$ variable. Consequently, 
\begin{equation*}
    \|m[a](D)\|_{M^{\infty}(\R^n)} \leq \|K[a]\|_{L^1(\R^n)} \leq \int_{\R} \|\mathcal{F}^{-1}a(\,\cdot\,;s)\|_{L^1(\R^n)} \chi(s)\,\ud s.
\end{equation*}
By the hypothesis i) on the $s$-support, the problem is therefore reduced to showing
\begin{equation}\label{gen Linfty 1}
    \sup_{s \in \R} \|\mathcal{F}^{-1}a(\,\cdot\,;s)\|_{L^1(\R^n)} \lesssim 1.
\end{equation}
However, the conditions ii) and iii), combined with a standard integration-by-parts argument, imply
\begin{equation*}
    |\mathcal{F}^{-1}a(\,\cdot\,;s)(x)| \lesssim_N \Big(\prod_{j=1}^n \lambda_j\Big) \Big(1 + \sum_{j=1}^n \lambda_j|\inn{x}{\bm{v}_j}| \Big)^{-N} \qquad \textrm{for all $(x;s) \in \R^n \times \R$ and all $N \geq 0$,}
\end{equation*}
from which the desired bound \eqref{gen Linfty 1} follows. 
\end{proof}




\bibliography{Reference}

\begin{thebibliography}{10}

\bibitem{BGHS-helical}
David Beltran, Shaoming Guo, Jonathan Hickman, and Andreas Seeger.
\newblock Sharp ${L}^p$ bounds for the helical maximal function.
\newblock Preprint: \verb+arXiv:2102.08272+.

\bibitem{BHS2020}
David Beltran, Jonathan Hickman, and Christopher~D. Sogge.
\newblock Variable coefficient {W}olff-type inequalities and sharp local
  smoothing estimates for wave equations on manifolds.
\newblock {\em Anal. PDE}, 13(2):403--433, 2020.

\bibitem{Bentsen}
Geoffrey Bentsen.
\newblock {$L^p$} regularity for a class of averaging operators on the
  {H}eisenberg group.
\newblock {T}o appear in {I}ndiana {U}niv. {M}ath. {J}., {P}reprint:
  \verb+arXiv:2002.01917+.

\bibitem{BD2015}
Jean Bourgain and Ciprian Demeter.
\newblock The proof of the {$l^2$} decoupling conjecture.
\newblock {\em Ann. of Math. (2)}, 182(1):351--389, 2015.

\bibitem{BDG2016}
Jean Bourgain, Ciprian Demeter, and Larry Guth.
\newblock Proof of the main conjecture in {V}inogradov's mean value theorem for
  degrees higher than three.
\newblock {\em Ann. of Math. (2)}, 184(2):633--682, 2016.

\bibitem{BGGIST2007}
L.~Brandolini, G.~Gigante, A.~Greenleaf, A.~Iosevich, A.~Seeger, and
  G.~Travaglini.
\newblock Average decay estimates for {F}ourier transforms of measures
  supported on curves.
\newblock {\em J. Geom. Anal.}, 17(1):15--40, 2007.

\bibitem{Christ1995}
Michael Christ.
\newblock Failure of an endpoint estimate for integrals along curves.
\newblock In {\em Fourier analysis and partial differential equations
  ({M}iraflores de la {S}ierra, 1992)}, Stud. Adv. Math., pages 163--168. CRC,
  Boca Raton, FL, 1995.

\bibitem{deLeeuw1965}
Karel de~Leeuw.
\newblock On {$L_{p}$} multipliers.
\newblock {\em Ann. of Math. (2)}, 81:364--379, 1965.

\bibitem{DR1986}
Javier Duoandikoetxea and Jos\'{e}~L. Rubio~de Francia.
\newblock Maximal and singular integral operators via {F}ourier transform
  estimates.
\newblock {\em Invent. Math.}, 84(3):541--561, 1986.

\bibitem{GLYZ}
Shaoming Guo, Zane~Kun Li, Po-Lam Yung, and Pavel Zorin-Kranich.
\newblock A short proof of $\ell^2$ decoupling for the moment curve.
\newblock {T}o appear in {A}mer. {J}. {M}ath., Preprint:
  \verb+arXiv:1912.09798+.

\bibitem{GO2020}
Shaoming Guo and Changkeun Oh.
\newblock Remarks on {W}olff's inequality for hypersurfaces.
\newblock {\em Math. Proc. Cambridge Philos. Soc.}, 168(2):249--259, 2020.

\bibitem{GZ2020}
Shaoming Guo and Pavel Zorin-Kranich.
\newblock Decoupling for moment manifolds associated to
  {A}rkhipov-{C}hubarikov-{K}aratsuba systems.
\newblock {\em Adv. Math.}, 360:106889, 56, 2020.

\bibitem{HormanderIII}
Lars H\"{o}rmander.
\newblock {\em The analysis of linear partial differential operators. {III}},
  volume 274 of {\em Grundlehren der Mathematischen Wissenschaften [Fundamental
  Principles of Mathematical Sciences]}.
\newblock Springer-Verlag, Berlin, 1985.
\newblock Pseudodifferential operators.

\bibitem{OS1999}
Daniel Oberlin and Hart~F. Smith.
\newblock A {B}essel function multiplier.
\newblock {\em Proc. Amer. Math. Soc.}, 127(10):2911--2915, 1999.

\bibitem{Oh2018}
Changkeun Oh.
\newblock Decouplings for three-dimensional surfaces in {$\mathbb{R}^6$}.
\newblock {\em Math. Z.}, 290(1-2):389--419, 2018.

\bibitem{PRS2011}
Malabika Pramanik, Keith~M. Rogers, and Andreas Seeger.
\newblock A {C}alder\'{o}n-{Z}ygmund estimate with applications to generalized
  {R}adon transforms and {F}ourier integral operators.
\newblock {\em Studia Math.}, 202(1):1--15, 2011.

\bibitem{PS2007}
Malabika Pramanik and Andreas Seeger.
\newblock {$L^p$} regularity of averages over curves and bounds for associated
  maximal operators.
\newblock {\em Amer. J. Math.}, 129(1):61--103, 2007.

\bibitem{SW2011}
Andreas Seeger and James Wright.
\newblock Problems on averages and lacunary maximal functions.
\newblock In {\em Marcinkiewicz centenary volume}, volume~95 of {\em Banach
  Center Publ.}, pages 235--250. Polish Acad. Sci. Inst. Math., Warsaw, 2011.

\bibitem{Stein1970}
Elias~M. Stein.
\newblock {\em Singular integrals and differentiability properties of
  functions}.
\newblock Princeton Mathematical Series, No. 30. Princeton University Press,
  Princeton, N.J., 1970.

\bibitem{Stein1993}
Elias~M. Stein.
\newblock {\em Harmonic analysis: real-variable methods, orthogonality, and
  oscillatory integrals}, volume~43 of {\em Princeton Mathematical Series}.
\newblock Princeton University Press, Princeton, NJ, 1993.
\newblock With the assistance of Timothy S. Murphy, Monographs in Harmonic
  Analysis, III.

\bibitem{Triebel1978}
Hans Triebel.
\newblock {\em Interpolation theory, function spaces, differential operators},
  volume~18 of {\em North-Holland Mathematical Library}.
\newblock North-Holland Publishing Co., Amsterdam-New York, 1978.

\bibitem{Wolff2000}
T.~Wolff.
\newblock Local smoothing type estimates on {$L^p$} for large {$p$}.
\newblock {\em Geom. Funct. Anal.}, 10(5):1237--1288, 2000.

\end{thebibliography}

\bibliographystyle{plain}
\end{document}